\documentclass[12pt,reqno]{amsart}
\usepackage{a4wide}
\numberwithin{equation}{section}

\usepackage{amsmath}
\usepackage{color}
\usepackage{txfonts}
\usepackage{amsfonts}
\usepackage{amsmath,amsthm,amssymb,amscd}
\usepackage{latexsym}
\usepackage[numbers,sort&compress]{natbib}
\allowdisplaybreaks

\numberwithin{equation}{section}

\newcommand\R{\mathbb R}

\newtheorem{theorem}{Theorem}[section]
\newtheorem{proposition}[theorem]{Proposition}
\newtheorem{corollary}[theorem]{Corollary}
\newtheorem{lemma}[theorem]{Lemma}

\theoremstyle{definition}

\newtheorem{remark}[theorem]{Remark}

\begin{document}

\title[Quantization analysis of Moser-Trudinger equations]{Quantization analysis of Moser-Trudinger equations in the Poincar\'e disk and applications}
\author{Lu Chen, Qiaoqiao Hua, Guozhen Lu, Shuangjie Peng and Chunhua Wang}
\address{School of Mathematics and Statistics, Beijing Institute of Technology, Beijing 100081, P. R. China}
\email{chenlu5818804@163.com}
\address{School of Mathematics and Statistics,
Central China Normal University,
Wuhan, 430079, China}
\email{hqq@mails.ccnu.edu.cn}
\address{Department of Mathematics,
University of Connecticut,
Storrs, CT 06269, USA}
\email{guozhen.lu@uconn.edu}
\address{School of Mathematics and Statistics, and Key Lab NAA--MOE, Central China Normal University, Wuhan 430079, China}
\email{sjpeng@ccnu.edu.cn}
\address{School of Mathematics and Statistics, and Key Lab NAA--MOE, Central China Normal University, Wuhan 430079, China}
\email{chunhuawang@ccnu.edu.cn}

\thanks{The first, second and fifth authors were supported by National Key Research and Development of China (No.2022YFA1006900). 
The first author was partly supported by the National Natural Science Foundation of China
(No. 12622108).
The third author was supported by grants from the Simons Foundation. The fourth author was supported by National Key R\&D Program (No.2023YFA1010002). The fifth author was also supported by National Natural Science Foundation of China (No.12471106).}

\begin{abstract}
In the recent  years, quantization analysis for the Moser-Trudinger equations and the existence of positive
critical points for the Moser-Trudinger functional on compact manifolds have been extensively studied.
However, the corresponding counterpart on non-compact manifolds still remains open.
In this paper, we make an attempt in the Poincar\'e disk, a complete and noncompact Riemannian manifold of negative constant curvature.
We first establish the quantitative properties for positive solutions to the Moser-Trudinger equation in the two-dimensional Poincar\'e disk $\mathbb{B}^2$:
\begin{equation*}\label{mt1}
	\left\{
	\begin{aligned}
		&-\Delta_{\mathbb{B}^2}u=\lambda ue^{u^2},\ x\in\mathbb{B}^2,\\
		&u\to0,\ \text{when}\ \rho(x)\to\infty,\\
		&\|\nabla_{\mathbb{B}^2} u\|_{L^2(\mathbb{B}^2)}^2\leq M_0,
	\end{aligned}
	\right.
\end{equation*}	
where $0<\lambda<\frac{1}{4}=\inf\limits_{u\in W^{1,2}(\mathbb{B}^2)\backslash\{0\}}\frac{\|\nabla_{\mathbb{B}^2}u\|_{L^2(\mathbb{B}^2)}^2}{\|u\|_{L^2(\mathbb{B}^2)}^2}$, $\rho(x)$ denotes the geodesic distance between $x$ and the origin and $M_0$ is a fixed large positive constant (see Theorem \ref{thm1}). Furthermore, by carefully analyzing the asymptotic expansion of the Dirichlet energy \(\|\nabla_{\mathbb{B}^2} u\|_{L^2(\mathbb{B}^2)}^2\) as \(\lambda \to 0\), and establishing that
\[
\lim_{\substack{\lambda \to 1/4}} \int_{\mathbb{B}^2} |\nabla_{\mathbb{B}^2} u_\lambda|^2 \, dV_{\mathbb{B}^2} = +\infty
\]
for~the ~nonzero~least energy solution,
we are able to prove the existence of a critical point for the Moser-Trudinger functional \(F(u) = \int_{\mathbb{B}^2} (e^{u^2} - 1) \, dV_{\mathbb{B}^2}\) under the constraint \(\int_{\mathbb{B}^2} |\nabla_{\mathbb{B}^2} u|^2 \, dV_{\mathbb{B}^2} = \Lambda\) for \(\Lambda \in (4\pi, +\infty)\), up to some M\"{o}bius transformation. This result is substantially different from the non-existence of critical points for the Moser-Trudinger functional in the Euclidean disk.
 Finally, by applying a local Pohozaev identity around the origin and exploiting the uniqueness of solutions to the Cauchy initial value problem for ODEs, we prove that the Moser-Trudinger equation has a unique positive solution up to some M\"{o}bius transformation when $\lambda$ is close to $0$. The sharp pointwise estimates and high-order asymptotic expansions both near and away from the origin, as well as the choice of the truncation radius, play an important role in our analysis.
\end{abstract}

\maketitle {\small {\bf Keywords:} Moser-Trudinger equations;
Poincar\'e disk; Quantization analysis; Uniqueness.
 \\

  }

\section{Introduction}

In this paper, we are concerned with the following Moser-Trudinger equations in the Poincar\'e disk $\mathbb{B}^2$
\begin{equation}\label{uniq}
	\left\{
	\begin{aligned}
		&-\Delta_{\mathbb{B}^2}u_\lambda=\lambda u_\lambda e^{u_\lambda^2},\ x\in\mathbb{B}^2,\\
		&u_\lambda\to0,\ \text{when}\ \rho(x)\to\infty,\\
		%		&\|\nabla_{\mathbb{B}^2} u_\lambda\|_{L^2(\mathbb{B}^2)}^2\leq M_0,
	\end{aligned}
	\right.
\end{equation}
where $0<\lambda<\frac{1}{4}$ and $\rho(x)=\ln\frac{1+|x|}{1-|x|}$ denotes the geodesic distance between $x$ and the origin in the Poincar\'{e} disk $\mathbb{B}^2$.
We will mainly study the quantization analysis of problem \eqref{uniq}
and its applications to the existence and uniqueness of critical points of the Moser-Trudinger functional. Quantization analysis of the Moser-Trudinger equations has attracted much attention due to its importance in applications to PDEs and geometric analysis. Let us first present a brief history of the main results in this direction.

\vskip 0.1cm
Quantization analysis for elliptic equations with the Moser-Trudinger growth dates back to the work of Druet in \cite{D}, who proved the following result:

\vspace{.2cm}
\noindent\textbf{Theorem A.}
Let $u_{k}$ be a family of solutions satisfying the equation
\begin{equation}\label{adeq2}
	\begin{cases}
		-\Delta u_k =\lambda_{k} u_ke^{u_k^2},\quad\quad & x\in \Omega, \\
		u_k>0,\quad\quad & x\in \Omega,\\
		0<\lambda_{k}<\lambda_1(\Omega),\\
		u_k=0,\quad\quad &x\in \partial \Omega,
	\end{cases}
\end{equation}
with $u_k$ being bounded in $W^{1,2}_0(\Omega)$. If $u_k$ blows up, then after passing to a subsequence, one has $\lambda_k\rightarrow \lambda_0$, $u_k\rightharpoonup u_0$ in $W^{1,2}_0(\Omega)$ and there exists some integer $N$ such that
$$\lim\limits_{k\rightarrow +\infty}\|\nabla u_k\|_{L^2(\Omega)}^2= \|\nabla u_0\|_{L^2(\Omega)}^2+4\pi N.$$

When $\Omega$ is a disk, all positive solutions  must be radially symmetric through the standard moving-plane method. (see e.g. \cite{ChenLi}, \cite{GNN}.)  Applying the refined radial asymptotic expansions of $u_k$ around the origin, Malchiodi and Martinazzi \cite{MM-JEMS} proved that $u_0=0$ and $N=1$.
%Furthermore, they showed that the Moser-Trudinger functional $I(u)=\int_{B_1}\left(e^{u^2}-1\right)dx$ under the constraint $\int_{B_1}|\nabla u|^2dx=\gamma$ does not admit any positive critical point for $\gamma$ sufficiently large. Chen, Lu, Xue and Zhu \cite{CLXZ} provided a new and simple proof for Malchiodi-Martinazzi's results, avoiding complicated blow-up analysis procedure.
For general bounded domains, Druet and Thizy \cite{DT} furthermore proved that $u_0=0$ and $N$ is equal to the number of concentration points. Further results concerning  the existence of solutions
to elliptic equations (see \cite{BS-2012,GK-2014}) or inequalities (see \cite{KS-2016, LuYangQ1, LuYangQ2,MST-2013})
in the hyperbolic space have been established.
The quantization analysis of the Moser-Trudinger equations in a bounded domain has been also extended by Chen, Lu, Xue and Zhu in \cite{CLXZ} to a hyperbolic ball. Their results can be stated as follows:

\vspace{.2cm}
\noindent\textbf{Theorem B.}
Assume that $u_\lambda$ are a family of solutions satisfying
\begin{equation}\label{adeq3}
	\begin{cases}
		-\Delta_{\mathbb{B}^2} u_\lambda =\lambda u_\lambda e^{u_\lambda^2},\quad\quad & x\in B_{\mathbb{B}^2}(0,R), \\
		u_\lambda>0,\quad\quad & x\in B_{\mathbb{B}^2}(0,R),\\
		0<\lambda<\lambda_1(B_{\mathbb{B}^2}(0,R)),\\
		u_\lambda=0,\quad\quad &x\in \partial B_{\mathbb{B}^2}(0,R),
	\end{cases}
\end{equation}
where $B_{\mathbb{B}^2}(0,R)$ denotes the Poincar\'e disk centered at the origin with the geodesic radius equal to $R<\infty$, $-\Delta_{\mathbb{B}^2}$ denotes the Laplace-Beltrami operator in $\mathbb{B}^2$,
$\lambda_1(B_{\mathbb{B}^2}(0,R))$ is the first eigenvalue of the operator $-\Delta_{\mathbb{B}^2}$ with the Dirichlet boundary in $B_{\mathbb{B}^2}(0,R)$. Then $u_\lambda$ is radially symmetric and unique.
Denote by $dV_{\mathbb{B}^2}=\left(\frac{2}{1-|x|^2}\right)^2dx$ the hyperbolic volume element and $\nabla_{\mathbb{B}^2}=\Big(\frac{1-|x|^2}{2}\Big)^{2}\nabla_{\R^{2}}$ the gradient operator in $\mathbb{B}^2$. When $\lambda\rightarrow \lambda_0$, there holds
\vskip 0.1cm

(i) if $\lambda_0=0$, then $u_\lambda$ blows up at the origin, and $|\nabla_{\mathbb{B}^2}u_\lambda|^2dV_{\mathbb{B}^2}\rightharpoonup 4\pi \delta_0$, $\lambda u_\lambda^2 e^{u_\lambda^2}dV_{\mathbb{B}^2}\rightharpoonup 4\pi \delta_0$;
\vskip 0.1cm

(ii) if $\lambda_0\in (0,\lambda_1(B_{\mathbb{B}^2}(0,R)))$, then $u_\lambda\rightarrow u_0$ in $C^{2}(B_{\mathbb{B}^2}(0,R)))$ and $u_0$ is a positive radial solution of the equation
\begin{equation}\label{ball}
	\begin{cases}
		-\Delta_{\mathbb{B}^2} u_0=\lambda_0 u_0e^{u_0^2},\quad\quad & x\in B_{\mathbb{B}^2}(0,R), \\
		u_0>0,\quad\quad & x\in B_{\mathbb{B}^2}(0,R), \\
		u_0=0,\quad\quad &x\in \partial B_{\mathbb{B}^2}(0,R);
	\end{cases}
\end{equation}
\vskip 0.1cm

(iii) if $\lambda_0=\lambda_1(B_{\mathbb{B}^2}(0,R))$, then $u_\lambda\rightarrow 0$ in $C^{2}(B_{\mathbb{B}^2}(0,R))$.
\vspace{.2cm}

However, the above quantization analysis of the Moser-Trudinger equations in the whole Poincar\'e disk $\mathbb{B}^2$ is still unknown. This is because the quantization analysis of the Moser-Trudinger equations in hyperbolic balls is based on the uniqueness of positive solutions to \eqref{adeq3}, which is not available for the corresponding equation in the Poincar\'e disk due to the presence of singularity of Hardy-potential near the boundary.
We are motivated to solve this problem. Our first main result states that

\begin{theorem}\label{thm1}
	Assume that $u_{k}\in W^{1,2}(\mathbb{B}^2)$ are a family of positive solutions satisfying
	\begin{equation}\label{eq1}
		\left\{
		\begin{aligned}
			&-\Delta_{\mathbb{B}^2}u_k=\lambda_k u_ke^{u_k^2},\ x\in\mathbb{B}^2,\\
			&u_k\to0,\ \text{when}\ \rho(x)\to\infty,\\
			&\|\nabla_{\mathbb{B}^2} u_k\|_{L^2(\mathbb{B}^2)}^2\leq M_0,
		\end{aligned}
		\right.
	\end{equation}
	where $0<\lambda_k<\frac{1}{4}$ and $M_0>0$ is a fixed large constant.
	Assuming that $\lim\limits_{k\rightarrow \infty} \lambda_k=\lambda_0$, up to some M\"{o}bius transformation $T_a$ (see 2.1 of Section 2) and some subsequence still denoted by $u_k$, we have
	\vskip 0.1cm
	
	(i) if $\lambda_0=0$, when $k\rightarrow \infty$, $u_{k}$ blows up at the origin, $\|\nabla_{\mathbb{B}^2} u_{k}\|_{L^2(\mathbb{B}^2)}^2\rightarrow4\pi$, $|\nabla_{\mathbb{B}^2} u_{k}|^2 dV_{\mathbb{B}^2}\rightharpoonup4\pi\delta_0$ and $\lambda_k u_{k}^2e^{u_{k}^2} dV_{\mathbb{B}^2}\rightharpoonup4\pi\delta_0$;
	\vskip 0.1cm
	
	(ii) if $\lambda_0\in(0,\frac{1}{4})$, then $u_{k}\to u_0$ in $W^{1,2}(\mathbb{B}^2)$ as $k\to\infty$, where $u_0$ is a radially symmetric positive solution of the equation
	\begin{equation}\label{eq3}
		\left\{
		\begin{aligned}
			&-\Delta_{\mathbb{B}^2}u_0=\lambda_0 u_0 e^{u_0^2},\ x\in\mathbb{B}^2,\\
			&u_0\to0,\ \text{when}\ \rho(x)\to\infty;\\
		\end{aligned}
		\right.
	\end{equation}
	\vskip 0.1cm
	
	(iii) if $\lambda_0=\frac{1}{4}$, then $u_{k}\to0$ in $C^{2}(\mathbb{B}^2)$ as $k\rightarrow \infty$.
\end{theorem}

\begin{remark}\label{rem1}
	From the proof of Theorem \ref{thm1}, we know that when $\lim\limits_{k\rightarrow \infty}\lambda_k=0$, the blow-up behavior of $u_k$ is uniquely determined. Hence, one needn't pick subsequences and can directly derive
	that the positive solutions $u_{\lambda}$ of equation \eqref{uniq} must blow up at the origin as $\lambda\rightarrow 0$, $\|\nabla_{\mathbb{B}^2} u_{\lambda}\|_{L^2(\mathbb{B}^2)}^2\rightarrow4\pi$, $|\nabla_{\mathbb{B}^2} u_{\lambda}|^2 dV_{\mathbb{B}^2}\rightharpoonup4\pi\delta_0$ and $\lambda u_{\lambda}^2e^{u_{\lambda}^2} dV_{\mathbb{B}^2}\rightharpoonup4\pi\delta_0.$
\end{remark}

%\begin{remark}
%We conjecture that $\lim\limits_{k\rightarrow \infty}\int_{\mathbb{B}^2}|\nabla_{\mathbb{B}^2} u_k|^{2}dV_{\mathbb{B}^2}=0$ when $\lim\limits_{k\rightarrow\infty}\lambda_k=\frac{1}{4}$. Indeed, from Theorem \ref{thm1} (iii) we know that if $\lambda_0=\frac{1}{4}$, then $u_{k}\to u_{0}=0$ in $C^{2}(\mathbb{B}^2)$ as $k\rightarrow \infty$. If we can improve the upper bound estimate in \eqref{b3-2}  of Lemma \ref{lem-b3}, then we have
%	\begin{equation}\nonumber
%		\begin{aligned}
%			\int_{\mathbb{B}^{2} }|\nabla _{\mathbb{B}^{2} }u_{k}	|^{2}d V_{\mathbb{B}^{2}}
%			={}&\int_{B_{\mathbb{B}^2}(0,R) }|\nabla _{\mathbb{B}^{2} }u_{k}	|^{2}d V_{\mathbb{B}^2}
%			+\int_{B^{c}_{\mathbb{B}^2}(0,R) }|\nabla _{\mathbb{B}^{2}}	u_{k}
%			|^{2}d V_{\mathbb{B}^2}\\
%			={}&  \int_{B_{\mathbb{B}^2}(0,R) }|\nabla _{\mathbb{B}^{2} }u_{0}|^{2}d V_{\mathbb{B}^2}+o_\lambda(1)+\omega_{2}
%			\int_{R}^{+\infty}le^{[-(1+\sqrt{1-4\lambda})-\delta]t}dt\\
%			={}&o_\lambda(1)+\omega_{2}\int_{R}^{+\infty}le^{[-(1+\sqrt{1-4\lambda})-\delta]t}dt
%			=o_\lambda(1)+o_R(1),
%		\end{aligned}
%	\end{equation}
%	as $\lambda\rightarrow \frac{1}{4}$ and $R\rightarrow +\infty$, where $l=\sinh~ t=\frac{e^{t}-e^{-t}}{2}$ and $\delta>0$ is small.
%	Therefore, there holds
%	$$
%	\lim_{k\rightarrow \infty}\int_{\mathbb{B}^{2} }|\nabla _{\mathbb{B}^{2} }u_{k}	|^{2}d V_{\mathbb{B}^{2}}=0.
%	$$
%	However, it is unfortunate that we can not improve the estimate \eqref{b3-2}. We think the main reason is that
%	we do not know how to obtain better bounds for equations \eqref{b3-7} and \eqref{b3-10}.
%\end{remark}

It is very effective to apply the moving plane method or the method of moving spheres to prove the symmetry and uniqueness of solutions (see e.g. \cite{CL, ChenLi, LZ}). We first adopt the hyperbolic moving-plane method in integral forms initially developed by Li, Lu and Yang \cite{LiLuYang} (see also \cite{cww-2026,llw-2023,llw-2025}) to show the hyperbolic symmetry of positive solutions (see Proposition 2.1). Up to some M\"{o}bius transformation, we can assume that positive solutions are radially decreasing.
Then, by employing the ODE theory and a fixed point argument, we obtain the blow-up behaviors of positive solutions (see Proposition \ref{pro1}), which is helpful to the proof of Theorem \ref{thm1}. The key ingredient here is to derive a monotonicity formula (see Lemma \ref{lem3}), which provides us with a crucial decay estimate away from the blow-up point so that we can establish the uniform boundedness of $\lambda_k c_k^2$, where $c_k=u_k(0)=\max\limits_{x\in \mathbb{B}^2}u_k(x)$.
In order to prove Theorem \ref{thm1}, we shall discuss three cases respectively. For the case $\lambda_0=0$, we first show that the Dirichlet energy has a positive lower bound as $k\to\infty$ and then use a contradiction argument to establish that the positive solution must blow up at the origin. This together with Proposition \ref{pro1} completes the study for the first case.
For the case $\lambda_0\in(0,\frac{1}{4})$, applying Proposition \ref{pro1} again, we can derive the $C^2(\mathbb{B}^2)$-convergence by standard elliptic regularity theory. To obtain the convergence in $W^{1,2}(\mathbb{B}^2)$, we need to deal with the singularity of the Hardy-potential near the boundary, where a decay property for radial positive solutions is essential.
For the case $\lambda_0=\frac{1}{4}$, we can also get the analogous $C^2(\mathbb{B}^2)$-convergence. The main difficulty here is to show that the limit of $u_k$ in $C^2(\mathbb{B}^2)$ is zero, since the first eigenvalue $\frac{1}{4}$ of the operator $-\Delta_{\mathbb{B}^2}$ with the Dirichlet boundary in $\mathbb{B}^2$ cannot be achieved. To overcome this difficulty, we rely on the decay properties of entire positive solutions, which leads to the non-existence of positive solutions to \eqref{uniq} when $\lambda=\frac{1}{4}$ (see Appendix B).

\vskip 0.1cm
Exploiting the existence of critical points for the Moser-Trudinger functional $M(u)=\int_{\Omega}\big(e^{u^2}-1\big)dx$ on the bounded domain under the constraint
$\int_{\Omega}|\nabla u|^2dx=\gamma$ for $\gamma>4\pi$ has been a challenging problem. By a variational method and the monotonicity of $M(u)$, Struwe \cite{ST} proved that there exists $\gamma^{*}>4\pi$ such that $M(u)$ has at least two positive critical points for almost every $\gamma\in (4\pi,\gamma^{*})$. Lamm, Robert and Struwe \cite{LRS} further introduced the Moser-Trudinger flow and strengthened it to every $\gamma\in (4\pi,\gamma^{*})$. del Pino, Musso and Ruf \cite{del-M-R} characterized some of these critical points when $\gamma$ is close to $4\pi$. When $\Omega$ is a disk, Malchiodi and Martinazzi \cite{MM-JEMS} applied a refined blow-up analysis for radial critical points of $M(u)$ to derive that $M(u)$ does not admit any positive critical point for $\gamma$ sufficiently large. It is conjectured that if $\Omega$ is a simply connected domain, the above non-existence result still holds. When $\Omega$ is a finite hyperbolic ball, Chen, Lu, Xue and Zhu \cite{CLXZ} established the multiplicity and non-existence of positive critical points for $\gamma>4\pi$. We also mention that recently De Marchis, Malchiodi, Martinazzi and Thizy \cite{MMMT} proved that when $\Omega$ is a closed manifold, the Moser-Trudinger functional $M(u)=\int_{\Omega}\big(e^{u^2}-1\big)dV_g$ under the constraint $\int_{\Omega}\big(|\nabla u|^2+|u|^2\big)dV_g=\gamma$ always has a nontrivial positive critical point for any $\gamma>0$.

\vskip 0.1cm
Theorem \ref{thm1} will be used to investigate the existence of critical points for the Moser-Trudinger functional $\int_{\mathbb{B}^2}\big(e^{u^2}-1\big)dV_{\mathbb{B}^2}$ under the constraint $\int_{\mathbb{B}^2}|\nabla_{\mathbb{B}^2}u|^2dV_{\mathbb{B}^2}=\gamma$ for $\gamma> 4\pi$. We have the following result.

\begin{theorem}\label{thm2}
	There holds that the Moser-Trudinger functional
	\begin{equation}\label{F}\nonumber
		F(u):=\int_{\mathbb{B}^2}\big(e^{u^2}-1\big) dV_{\mathbb{B}^2}
	\end{equation}
	under the constraint $\int_{\mathbb{B}^2}|\nabla_{\mathbb{B}^2}u|^2 dV_{\mathbb{B}^2}=\Lambda$ has at least one positive critical point for $\Lambda\in(4\pi,+\infty)$ up to some M\"{o}bius transformation.
\end{theorem}

Our proof of Theorem \ref{thm2} consists of two main steps.

\medskip
\noindent\textbf{Step 1.} Applying Theorem  \ref{thm1} and Remark \ref{rem1}, we know that, up to a M\"obius transformation, the positive solutions $u_\lambda$ of \eqref{eq1} (with $\lambda_k$ replaced by $\lambda$) blow up at $0$ as $\lambda\to 0$ and satisfy
\[
\lim_{\lambda\to 0}\int_{\mathbb{B}^2}|\nabla_{\mathbb{B}^2}u_\lambda|^2\,dV_{\mathbb{B}^2}=4\pi.
\]
In this step we aim to prove that the Dirichlet energy approaches $4\pi$ from above as $\lambda\to 0$. This will be achieved by expanding the solution near the origin up to order $o(c_\lambda^{-4})$ (see Lemma \ref{lem6}) and then estimating the energy deficit
\[
\int_{\mathbb{B}^2}|\nabla_{\mathbb{B}^2}u_\lambda|^2\,dV_{\mathbb{B}^2}-4\pi
\]
more precisely (Lemma \ref{lem7}) to determine the sign of the error term of order $c_\lambda^{-4}$, where $c_\lambda=u_\lambda(0)=\max_{x\in\mathbb{B}^2}u_\lambda(x)$.

\medskip
\noindent\textbf{Step 2.} We then prove that
\begin{equation*}
    \lim_{\lambda\rightarrow \frac{1}{4}} \int_{\mathbb{B}^2} |\nabla_{\mathbb{B}^2} u_{\lambda}|^2 \, dV_{\mathbb{B}^2} = +\infty
\end{equation*}
for~the~nonzero~least~energy~solution.
 To establish this energy blow-up, we employ a simple contradiction argument combining the endpoint Poincar\'e--Sobolev inequality with Theorem \ref{thm1}(iii), thereby avoiding any additional blow-up analysis.
 In addition, using the continuous dependence of the least--energy branch on $\lambda$ for $0<\lambda<1/4$, we can apply the intermediate value theorem. Combining the local result near $4\pi$ (from Step 1) with this blow--up at $1/4$ and the continuity, we conclude that for every $\Lambda>4\pi$ the Moser--Trudinger functional $F(u)$ under the constraint
\[
\int_{\mathbb{B}^2}|\nabla_{\mathbb{B}^2}u|^2\,dV_{\mathbb{B}^2}=\Lambda
\]
possesses at least one positive critical point.

\begin{remark}
	Since the existence of extremals for the Moser-Trudinger inequality in the Poincar\'e disk is still unknown,
	it is impossible to apply the perturbed method developed by Struwe in \cite{ST} to obtain the existence of critical points for the Moser-Trudinger functional under the constraint $\int_{\mathbb{B}^2}|\nabla_{\mathbb{B}^2}u|^2 dV_{\mathbb{B}^2}=\Lambda$ for $\Lambda$ sufficiently close to $4\pi$. By a delicate expansion for the deficit $\big(\|\nabla_{\mathbb{B}^2} u_\lambda\|_{L^2(\mathbb{B}^2)}^2-4\pi\big)$ as $\lambda\rightarrow 0$, we overcome this difficulty.
\end{remark}

Another application of Theorem \ref{thm1} is the uniqueness of positive solutions to Moser-Trudinger equations, which is an important topic in partial differential equations and geometric analysis. It is stated as follows.

\begin{theorem}\label{unique}
	There exists $\lambda_0>0$ such that the positive solutions of Moser-Trudinger equation \eqref{uniq}
	with finite Dirichlet energy are unique up to some M\"{o}bius transformation for any $0<\lambda<\lambda_0$.
\end{theorem}

There have been extensive studies on  the uniqueness  of positive solutions to
elliptic equations with nonlinear polynomial growth
(\cite{DLY,Glangetas,GILY,GPY}) and with exponential nonlinearities (\cite{BJLY-2019,BJLY-2019-2,cww-2026,LY-2018}).
However, much less is known for the
uniqueness results for Moser-Trudinger equation \eqref{uniq}.
In \cite{LPP-2022}, Luo, Pan and Peng
studied the uniqueness of the low energy solutions to \eqref{uniq}
on a bounded domain in $\R^{2}.$
To prove Theorem \ref{unique}, we mainly apply a contradiction argument and some local Pohozaev identities.
Specifically, we assume that there exist two solutions $u^{(1)}_{\lambda}(x)$
and $u^{(2)}_{\lambda}(x)$ of \eqref{uniq}.
Note that we can show the
positive solutions of equation \eqref{uniq} are radially decreasing up to some M\"{o}bius transformation (see Proposition \ref{pro}). By  the uniqueness of solutions to the Cauchy initial value problem for ODEs, we only need to prove $u^{(1)}_{\lambda}(0)=u^{(2)}_{\lambda}(0).$
To this end, we wish to prove $u^{(1)}_{\lambda}(x)=u^{(2)}_{\lambda}(x)$
in a Poincar\'{e} disk with a small radius and the center of the disk being at 0.

\vskip 0.1cm
A pivotal and technically demanding ingredient in our uniqueness proof is the establishment of a precise asymptotic relationship between the parameter $\lambda$ and the concentration level $c_{\lambda}:=u_\lambda(0)=\max\limits_{x\in \mathbb{B}^2}u_\lambda(x)$ as $\lambda\to 0$ (see Proposition \ref{add-lem6}). This high-order characterization is fundamentally new and substantially transcends the approaches available in the existing literature. We develop a completely different approach based on precise pointwise estimates of solutions both near and away from the concentration point to establish the relationship between the parameter $\lambda$ and $c_{\lambda}$. More precisely, we employ a progressive iteration scheme. First, we establish the preliminary relation $\lambda c_\lambda^2\to 1$ (Proposition \ref{lem-7-31-2}) via a pointwise comparison with the Green's function (Lemma \ref{lem-7-31-1}). Then we refine the expansion by proving that $c_\lambda^2(c_\lambda u_\lambda - 4\pi G)$ converges to a limit function $W$  in $C^{1,\gamma}_{\rm loc}(\mathbb{B}^2\setminus\{0\})$ (Lemma \ref{ad}) and further prove that the higher-order residual $c_\lambda^2[c_\lambda^2(c_\lambda u_\lambda - 4\pi G) - W]$ also converges (Lemma \ref{ad-lem2}). These results provide, for the first time, a clear picture of how $c_\lambda u_\lambda$ deviates from $4\pi G$ on the whole Poincar\'e disk. In this process, the choice of the truncation radius $r_\lambda \widetilde{s_\lambda}$ with $\widetilde{s_\lambda}^2 = e^{(1-\sigma)c_\lambda^2}$ plays an important role, as it allows us to separate the domain into a region where blow-up analysis applies and a region where the Moser-Trudinger inequality provides sufficient control.  We point out that such a third-order expansion in Proposition \ref{add-lem6} is not just a technical refinement but is truly necessary: the local Pohozaev identity used in the uniqueness proof (Proposition \ref{prop-A}) involves error terms that can only be controlled when $c_\lambda$ is expanded up to the order $O(\lambda^{\frac{3}{2}})$.

\vskip 0.1cm
Our paper is organized as follows. In Section \ref{s2}, we gather some
auxiliary facts about hyperbolic spaces and prove the hyperbolic symmetry of positive solutions of Moser-Trudinger equations in the Poincar\'{e} disk.
% introducing the reflection of two-dimensional hyperboloid model and Poincar\'{e} disk, proving the
%hyperbolic symmetry of positive solutions of Moser-Trudinger equation in Poincar\'{e} disk.
Theorems \ref{thm1} and \ref{thm2} are proved in Sections \ref{s4} and \ref{s5} respectively, involving the quantization analysis of solutions and the existence of positive critical points for the Moser-Trudinger functional.
In Section \ref{cc}, we investigate the relation between $\lambda$ and $c_\lambda$ when $\lambda$ is close to 0 and we prove Theorem \ref{unique} in Section \ref{s6}.
Finally, we collect some previous results and basic preliminaries in Appendix \ref{sa} and give the decay estimates of positive solutions in Appendix \ref{sb}.

\vspace{.2cm}
\noindent{\bf Notations.} We will use $B_r$ to denote a disk centered at the origin with radius $r$ in $\mathbb{R}^2$.
%For $p\in[1,\infty]$ and $u\in L^p(\mathbb{B}^2)$, we denote $\|u\|_p=\|u\|_{L^p(\mathbb{B}^2)}$.
$C$ represents a positive constant that may change from place to place, and
$O(\lambda)$, $o(\lambda)$, $o_\lambda(1)$ mean $|O(\lambda)|\leq C|\lambda|$, $o(\lambda)/\lambda\to0$ and $o_\lambda(1)\to0$ as $\lambda\to0$, respectively. Given a set $M$ and two functions $f_1,f_2:M\to\mathbb{R}$, we write $f_1\lesssim f_2$ if there exists $C>0$ such that $f_1(m)\leq Cf_2(m)$ for all $m\in M$. The symbol $\gtrsim$ is defined analogously.
%Finally, $f_1\sim f_2$ means that $f_1\lesssim f_2$ and $f_1\gtrsim f_2$.

\section{Auxiliary facts}\label{s2}

\subsection{Hyperbolic spaces and M\"{o}bius transformations}\label{sect:2.1}
We introduce the ball model $\mathbb{B}^n$ of the hyperbolic space. Denote $\mathbb{B}^n$ by the Poincar\'{e} ball, which is a unit ball $B^n$ equipped with the usual Poincar\'{e} metric $g=\big(\frac{2}{1-|x|^2}\big)^2g_e$, where $g_e$ represents
the standard Euclidean metric. The hyperbolic volume element can be written as $dV_{\mathbb{B}^n}=\big(\frac{2}{1-|x|^2}\big)^ndx$ and the geodesic distance from the origin to $x\in \mathbb{B}^n$ is given by $\rho(x)= \ln \frac{1+|x|}{1-|x|}$. Let $B_{\mathbb{B}^n}(0,R)$ denote the hyperbolic ball centered at the origin with the radius equal to $R$. The associated Laplace-Beltrami operator $\Delta_{\mathbb{B}^n}$ and the gradient operator $\nabla_{\mathbb{B}^n}$ are given respectively by
%\begin{equation}\label{laplace}
%	\Delta_{\mathbb{B}^n}=\frac{1-|x|^2}{4}\left((1-|x|^2)\sum_{i=1}^{n}\frac{\partial^2}{\partial x_i^2}+2(n-2)\sum_{i=1}^{n}x_i\frac{\partial}{\partial x_i}\right),\ \ \nabla_{\mathbb{B}^n}=\frac{1-|x|^2}{2}\left(\frac{\partial}{\partial x_1},\cdots,\frac{\partial}{\partial x_n}\right).
%\end{equation}
$$\Delta_{\mathbb{B}^n}=\frac{1-|x|^2}{4}\Big((1-|x|^2)\sum_{i=1}^{n}\frac{\partial^2}{\partial x_i^2}+2(n-2)\sum_{i=1}^{n}x_i\frac{\partial}{\partial x_i}\Big),$$ and $$\nabla_{\mathbb{B}^n}=\Big(\frac{1-|x|^2}{2}\Big)^{2}\nabla_{\R^{n}}.$$
Denote by $W^{1,2}(\mathbb{B}^n)$ the first order Sobolev space in the Poincar\'{e} ball $\mathbb{B}^n$, i,e. the completion of $C_{c}^{\infty}(\mathbb{B}^n)$ under the norm $\left(\int_{\mathbb{B}^n}|\nabla_{\mathbb{B}^n}\cdot|^2dV_{\mathbb{B}^n}\right)^{\frac{1}{2}}$. In the Poincar\'{e} ball $\mathbb{B}^n$,
for any $u\in W^{1,2}(\mathbb{B}^n)$, there holds
\begin{equation}\label{Poincare}
	\int_{\mathbb{B}^n}|\nabla_{\mathbb{B}^n}u|^2dV_{\mathbb{B}^n}\geq \frac{(n-1)^2}{4}\int_{\mathbb{B}^n}|u|^2dV_{\mathbb{B}^n}.
\end{equation}

Under the polar coordinate system, the hyperbolic metric $g$ can be decomposed into
$$g=d\rho^2+ \sinh^2 \rho d\sigma,$$
where $d\sigma$ is the standard sphere metric. Then it is not difficult to check that
$\Delta_{\mathbb{B}^n}$ and $\nabla_{\mathbb{B}^n}$ can be written as
$$\Delta_{\mathbb{B}^n}=\frac{\partial^2}{\partial \rho^2}+(n-1)\coth\rho\frac{\partial}{\partial \rho}+\frac{1}{\sinh^2 \rho}\Delta_{\mathbb{S}^{n-1}},\ \ \nabla_{\mathbb{B}^n}=\left(\frac{\partial}{\partial \rho},\  \frac{1}{\sinh \rho}\nabla_{\mathbb{S}^{n-1}}\right),$$
respectively. We can also write
\begin{equation}
	\begin{split}
		\int_{\mathbb{B}^n}f(x)dV_{\mathbb{B}^n}&=\int_{0}^{1}\int_{S^{n-1}}f(r\xi)r^{n-1}\Big(\frac{2}{1-r^2}\Big)^nd\xi dr\\
		&=\int_{0}^{+\infty}\int_{S^{n-1}}f\left(\tanh(\frac{\rho}{2})\xi\right)(\sinh \rho)^{n-1}d\xi d\rho
	\end{split}
\end{equation}
under the the polar coordinate system.

For each $a\in \mathbb{B}^n$, we define the M\"{o}bius transformations $T_a$ by (see \cite{Ahlfors})
\begin{equation}\label{Mc}
	T_a(x)=\frac{|x-a|^2a-(1-|a|^2)(x-a)}{1-2x\cdot a+|x|^2a^2},
\end{equation}
where $x\cdot a$ denotes the scalar product in $\mathbb{R}^n$. It is known that the volume element $dV_{\mathbb{B}^n}$ on $\mathbb{B}^n$ is invariant with respect to the M\"{o}bius transformations, which deduces that for any $\varphi\in L^{1}(\mathbb{B}^n)$,
there holds $$\int_{\mathbb{B}^n}|\varphi\circ \tau_{a}|dV_{\mathbb{B}^n}=\int_{\mathbb{B}^n}|\varphi|dV_{\mathbb{B}^n}.$$
Furthermore, the commutativity of M\"{o}bius transformations $T_a$ with the operator $-\Delta_{\mathbb{B}^n}$ still holds. That is to say that for any $\phi\in C^{\infty}_c(\mathbb{B}^n)$, there holds
$$\int_{\mathbb{B}^n}-\Delta_{\mathbb{B}^n}(\phi\circ \tau_a)(\phi\circ \tau_a)dV_{\mathbb{B}^n}=\int_{\mathbb{B}^n}(-\Delta_{\mathbb{B}^n}\phi)\circ \tau_a\cdot (\phi\circ \tau_a)dV_{\mathbb{B}^n}=\int_{\mathbb{B}^n}-\Delta_{\mathbb{B}^n}\phi\cdot \phi dV_{\mathbb{B}^n}.$$

Using the M\"{o}bius transformations, we can define the geodesic distance from $x$ to $y$ in $\mathbb{B}^n$ as follows
$$\rho(x,y)=\rho(T_{x}(y))=\rho(T_{y}(x))= \ln \frac{1+|T_{y}(x)|}{1-|T_{y}(x)|}.$$
Also using the M\"{o}bius transformations again, we can define the convolution of measurable functions $f$ and $g$ on $\mathbb{B}^n$ by (see \cite{liu}) $$(f\ast g)(x)=\int_{\mathbb{B}^n}f(y)g(T_x(y))dV_{\mathbb{B}^n}(y),$$
where $dV_{\mathbb{B}^n}(y)=\big(\frac{2}{1-|y|^2}\big)^{n}dy$.

\subsection{Hardy-Littlewood-Sobolev inequality in $\mathbb{B}^2$}
It is well known that for any $u\in W^{1,2}(\mathbb{B}^2)$, there holds
$$
\int_{\mathbb{B}^2}\big(|\nabla_{\mathbb{B}^2}u|^2-\lambda |u|^2\big)dV_{\mathbb{B}^2}
\geq C_{\lambda, q} \Big(\int_{\mathbb{B}^2}|u|^qdV_{\mathbb{B}^2}\Big)^{\frac{2}{q}}
$$
for any $0<\lambda<\frac{1}{4}$ and $2\leq q<+\infty$. If we define $f=(-\Delta_{\mathbb{B}^2}-\lambda)^{\frac{1}{2}}u$, by duality, the above inequality is equivalent to $$
\Big(\int_{\mathbb{B}^2}|f|^{q'}dV_{\mathbb{B}^2}\Big)^{\frac{2}{q'}}\geq C_{\lambda, q}\int_{\mathbb{B}^2}|(-\Delta_{\mathbb{B}^2}-\lambda)^{-\frac{1}{2}}f|^2dV_{\mathbb{B}^2}
$$
with $\frac{1}{q}+\frac{1}{q'}=1$, which can be written as
$$
\int_{\mathbb{B}^2}\int_{\mathbb{B}^2}f(x)G_{\lambda}(x,y)f(y)dV_{\mathbb{B}^2}(y)dV_{\mathbb{B}^2}(x)\leq C_{\lambda, q}^{-1}\Big(\int_{\mathbb{B}^2}|f|^{q'}dV_{\mathbb{B}^2}\Big)^{\frac{2}{q'}},
$$
where $G_{\lambda}(x,y)$ is the Green's function of the operator $-\Delta_{\mathbb{B}^2}-\lambda$ in the Poincar\'{e} disk $\mathbb{B}^2$.
Furthermore, we can also derive that for any $f$ and $g\in L^{q'}(\mathbb{B}^2)$, there holds
$$\int_{\mathbb{B}^2}\int_{\mathbb{B}^2}f(x)G_{\lambda}(x,y)g(y)dV_{\mathbb{B}^2}(y)dV_{\mathbb{B}^2}(x)\leq C_{\lambda, q}^{-1}\|f\|_{L^{q'}(\mathbb{B}^2)}\|g\|_{L^{q'}(\mathbb{B}^2)}.$$
If we define $I_{\lambda}(f)=\int_{\mathbb{B}^2}G_{\lambda}(x,y)f(y)dV_{\mathbb{B}^2}(y)$, then it follows that
$$\|I_{\lambda}(f)\|_{L^q(\mathbb{B}^2)}\leq C_{\lambda, q}^{-1} \|f\|_{L^{q'}(\mathbb{B}^2)}.$$

\subsection{Totally geodesic lines of two-dimensional hyperboloid model }
Let $\mathbb{R}^{2,1}=(\mathbb{R}^{2+1},g)$ be the Minkowski space, where the metric $$ds^2=dx_1^2+dx_2^2-dx_{3}^2.$$ The hyperboloid model of hyperbolic space $\mathbb{H}^2$ is the submanifold $$\{x\in \mathbb{R}^{2,1}:x_1^2+x_2^2-x_{3}^2=-1, x_{3}>0\}.$$ The totally geodesic line $U_{x_2}$ along $x_2$ direction through the origin of $\mathbb{R}^2$ can be defined as
$U_{x_2}=\{x\in \mathbb{H}^2:x_2=0\}$. The general totally geodesic line along $x_2$ direction can be generated through hyperbolic rotation. More precisely, define $A_{x_2}^t=Id_{\mathbb{R}^{1}}\bigotimes\widetilde{A}^{t} $, where $\widetilde{A}^{t}$ is the hyperbolic rotation on $\mathbb{R}^{2}$ and is given by
\begin{align*}
	\widetilde{A}^{t}=\bigg{(}
	\begin{array}{c}
		\cosh t, \ \sinh t\\
		\sinh t,\  \cosh t
	\end{array} \bigg{)}.
\end{align*}
Then $\{U_{x_2}^{t}:=A_{x_2}^{t}(U_{x_2})\}_{t\in \mathbb{R}}$ is a family of totally geodesic lines along $x_2$ direction and they are pairwise disjoint and constitute the whole hyperbolic space $\mathbb{H}^2$. The totally geodesic line $U_{\nu}$ along $\nu\in \mathbb{S}^{1}$ direction through the origin of $\mathbb{R}^2$ can be defined as $U_{\nu}=\{x\in \mathbb{H}^2:(x_1,x_2)\cdot\nu=0\}$. For any $x'=(x_1,x_2)\in \mathbb{R}^2$, we can write $x'=(x'\cdot\nu)\nu+y'$ and $y'$ is orthogonal to $\nu$. $U_{\nu}^{t}=Id_{\nu^{\perp}}\bigotimes\widetilde{A}^{t}$. Simple calculation gives that
$$U_{\nu}^{t}=\left(\sinh t x_{3}\nu+y', \cosh t x_{3}\right), \ \ (x', x_{3})\in U_{\nu}.$$
It is easy to check that $\{U_{\nu}^{t}\}_{t\in \mathbb{R}}$ are pairwise disjoint and also constitute the whole hyperbolic space $\mathbb{H}^2$.

%%The reflection is defined by $I_{x_n}(x_1, x_2, \cdot\cdot\cdot, x_n, x_{n+1})=(x_1, x_2, \cdot\cdot\cdot, -x_n, x_{n+1})$. Let $U^{x_n}=\mathbb{H}^n\bigcap \{x_n=0\}$ and
%%$U_t^{x_n}=A_t^{x_n}(U^{x_n})$, then $\mathbb{H}^n$ is foliated by $U_t^{x_n}$, i.e. $\mathbb{H}^n=\bigcup_{t\in \mathbb{R}}U_t^{x_n}$. Moreover, if we define $I_{x_2}^t^{x_n}=A_{t}^{x_n}\circ I_{x_n}\circ A_{-t}^{x_n}$, then it is easy to verify that $I_{x_2}^t(U_t^{x_n})=U_t^{x_n}$. We call the $U^{x_n}_t$ the hyperbolic hyperplane in $\mathbb{H}^n$ along the $x_n$ direction. Similarly, we can define hyperbolic hyperplane along the any direction $\nu \in \mathbb{S}^{n-1}$.%%

\subsection{Totally geodesic lines of Poincar\'{e} disk model and reflection}
Let $\phi$ be the isometric map from two-dimensional hyperboloid model to Poincar\'{e} disk. The $\phi$ can be obtained by stereographic projection from the hyperboloid to the plane $\{x_{3}=0\}$ taking the vertex from which to project to be $(0,0,-1)$.
In fact, we can write the map $\phi$ as
$\phi: x\in \mathbb{H}^2\mapsto \frac{x'}{x_{3}+1}\in \mathbb{\mathbb{B}}^2$. Under the map $\phi$, by careful calculation, one can check that
the totally geodesic line $U_{\nu}$ of hyperboloid $\mathbb{H}^2$ becomes $\{x \in \mathbb{B}^2:x\cdot \nu=0\}$, which is also a geodesic line of Poincar\'{e} disk. Furthermore $\{\phi(U_{\nu}^t)\}_{t\in \mathbb{R}}$ are pairwise disjoint and constitute Poincar\'{e} disk $\mathbb{B}^2$. Let us recall M\"{o}bius transformations in the Poincar\'{e} disk. For each $a\in \mathbb{B}^2$, we define the M\"{o}bius transformations $T_a$ by (see \cite{Ahlfors})
$$T_a(x)=\frac{|x-a|^2a-(1-|a|^2)(x-a)}{1-2x\cdot a+|x|^2a^2},$$
where $x\cdot a$ denotes the scalar product in $\mathbb{R}^2$. The M\"{o}bius transformation includes rotation and is isometric from $\mathbb{B}^2$ to $\mathbb{B}^2$. It is well known that for any fixed $\nu \in \mathbb{S}^{1}$, $\{T_a\left(\phi(U_{\nu})\right)\}_{a\in \mathbb{B}^2}$ generates all geodesic lines of Poincar\'{e} disk.

For the totally geodesic line $\phi(U_{x_2})=\{x\in \mathbb{B}^2:x_2=0\}$ along $x_2$ direction in $\mathbb{B}^2$, it is easy to define the reflection $I_{x_2}$ about $\phi(U_{x_2})$ through $I_{x_2}(x_1, x_2)=(x_1, -x_2)$ for $x=(x_1, x_2)\in \mathbb{B}^2$. Obviously, through the geodesic line equation, one can easily verify that $I_{x_2}$ maps the geodesic line to the geodesic line. One can similarly define the reflection $I_{\nu}$ about $\phi(U_{\nu})$ through $I_{\nu}(x_1, x_2)=-(x\cdot\nu)\nu+y$, where $x=(x_1, x_2)=(x\cdot\nu)\nu+y$ and $y\in \nu^{\perp}$. Obviously, $I_{\nu}(\phi(U_{\nu}))=\phi(U_{\nu})$.

Now, we start to define the reflection about the general totally geodesic line in $\mathbb{B}^2$. According to the definition of totally geodesic line, any totally geodesic line can be written as $T_a(\phi(U_{x_2}))$ for some $a>0$. The reflection $I_{x_2}^{a}$ about $T_a(\phi(U_{x_2}))$ can be defined through $$I_{x_2}^{a}(x)=T_{a} \circ I_{x_2}\circ T_{a}$$
since $T_a^{-1}=T_a$. Simple calculation gives that $I_{x_2}^{a}\left(T_a(\phi(U_{x_2}))\right)=T_a\circ I_{x_2}\circ T_a\left(T_a(\phi(U_{x_2}))\right)=T_a\circ I_{x_2}(\phi(U_{x_2}))=T_a(\phi(U_{x_2}))$.

\subsection{Symmetry for Moser-Trudinger equations in $\mathbb{B}^2$}
\begin{proposition}\label{pro}
	Assume that $u\in W^{1,2}(\mathbb{B}^2)$ is a positive solution of equation
	\begin{equation}\label{blowup1}
		\left\{
		\begin{aligned}
			&-\Delta_{\mathbb{B}^2}u=\lambda u e^{u^2},\ x\in\mathbb{B}^2,\\
			&u\to0,\ \text{when}\ \rho(x)\to\infty,
		\end{aligned}
		\right.
	\end{equation}
	where $0<\lambda<\frac{1}{4}$.
	Then there exists $P\in \mathbb{B}^2$ such that $u$ is a constant on the geodesic sphere centered at $P$. Furthermore, $u(x)$ is strictly decreasing about the geodesic distance $\rho(x)=d(x,P)$.
\end{proposition}
\begin{proof}
	By Green's formula, we know that $$u(x)=\int_{\mathbb{B}^2}G_{\lambda}(x,y)\lambda u(e^{u^2}-1)dV_{\mathbb{B}^2},$$
	where $G_{\lambda}(x,y)$ is the Green's function of the operator $-\Delta_{\mathbb{B}^2}-\lambda$ in the Poincar\'{e} disk $\mathbb{B}^2$ (The Green's function $G_{\lambda}(x,y)$ is a decreasing function about $\rho(x,y)$, one can refer to Lu and Yang  \cite{LuYangQ3} for details). We only need to show the symmetry of positive solutions to this integral equation. We adopt the moving-plane method in integral forms of the hyperbolic space
	developed by Li, Lu and Yang \cite{LiLuYang}. (see also \cite{llw-2023, llw-2025}.) To perform the moving-plane arguments, we fix one specific direction $\nu\in\mathbb{S}^{1}$ and for such direction consider $t>0$, then the hyperbolic hyperplane $\phi(U_{\nu}^t)$ splits the Poincar\'{e} disk $\mathbb{B}^2$ into two parts. We denote by $\Sigma_{t}=\bigcup_{s<t}\phi(U_{\nu}^t)$ and $u_{t}(x)=u(I^{t}_\nu(x))$.
	Direct calculations give that for any $x\in \Sigma_{t}$, there hold
	\begin{equation}\begin{split}\nonumber
			u(x)&=\int_{\Sigma_{t}}G_{\lambda}(x,y)\lambda u(e^{u^2}-1)dV_{\mathbb{B}^2}+\int_{\mathbb{B}^2 \backslash \Sigma_{t}}G_{\lambda}(x,y)\lambda u(e^{u^2}-1)dV_{\mathbb{B}^2}\\
			&=\int_{\Sigma_{t}}G_{\lambda}(x,y)\lambda u(e^{u^2}-1)dV_{\mathbb{B}^2}+\int_{\Sigma_{t}}G_{\lambda}(x, I_{t}(y))\lambda u_t(e^{u_t^2}-1)dV_{\mathbb{B}^2}
	\end{split}\end{equation}
	and
	\begin{equation}\begin{split}\nonumber
			u_t(x)&=\int_{\Sigma_{t}}G_{\lambda}(I_t(x),y)\lambda u(e^{u^2}-1)dV_{\mathbb{B}^2}+\int_{\mathbb{B}^2 \backslash \Sigma_{t}}G_{\lambda}(I_t(x), y)\lambda u(e^{u^2}-1)dV_{\mathbb{B}^2}\\
			&=\int_{\Sigma_{t}}G_{\lambda}(I_t(x),y)\lambda u(e^{u^2}-1)dV_{\mathbb{B}^2}+\int_{\Sigma_{t}}G_{\lambda}(I_t(x), I_t(y))\lambda u_t(e^{u_t^2}-1)dV_{\mathbb{B}^2}.
	\end{split}\end{equation}
	Since $I_\nu^t$ is an isometry, we have $G_\lambda(I_\nu^t(x),y)=G_\lambda(x, I_\nu^t(y))$ and $G_\lambda(I_\nu^t(x), I_\nu^t(y))=G_\lambda(x,y)$. Hence we can derive that
	\begin{equation}\begin{split}\label{identity}
			u_t(x)-u(x)&=\int_{\Sigma_{t}}G_{\lambda}(I_t(x),y)\Big(\lambda u(e^{u^2}-1)-\lambda u_t(e^{u_t^2}-1)\Big)dV_{\mathbb{B}^2}\\
			&\quad +\int_{\Sigma_{t}}G_{\lambda}(x,y)\Big(\lambda u_t(e^{u_t^2}-1)-\lambda u(e^{u^2}-1)\Big)dV_{\mathbb{B}^2}\\
			&=\int_{\Sigma_{t}}\left(G_{\lambda}(x,y)-G_{\lambda}(I_\nu^t(x),y)\right)\Big(\lambda u_t(e^{u_t^2}-1)-\lambda u(e^{u^2}-1)\Big)dV_{\mathbb{B}^2}.\\
	\end{split}\end{equation}
	Step 1. We compare the values of $u_t(x)$ and $u(x)$.
	We first show that for $t$ sufficiently negative, there holds
	%we are going
	%to show that
	\begin{equation}\label{starting}u_t(x)\geq u(x),\ \ \forall x \in \Sigma_{t}.
	\end{equation}
	Define $\Sigma_{t}^{u}=\{x\in \Sigma_{t}:\ u(x)> u_t(x)\}.$
	%we will show that for $t$ sufficient negative,
	%$u(x)\leq u_t(x)$ for any $x\in \Sigma_{t}$.
	By \eqref{identity} and the mean-value theorem, we derive
	\begin{equation}\begin{split}\nonumber
			u(x)-u_{t}(x)&=\int_{\Sigma_{t}}\left(G_{\lambda}(x,y)-G_{\lambda}(I_t(x),y)\right)\Big(\lambda u(e^{u^2}-1)-\lambda u_t(e^{u_t^2}-1)\Big)dV_{\mathbb{B}^2}\\
			&\leq \int_{\Sigma_{t}^{u}}\left(G_{\lambda}(x,y)-G_{\lambda}(I_t(x),y)\right)\Big(\lambda u(e^{u^2}-1)-\lambda u_t(e^{u_t^2}-1)\Big)dV_{\mathbb{B}^2}\\
			&\leq \int_{\Sigma_{t}^{u}}G_{\lambda}(x,y)\Big(\lambda u(e^{u^2}-1)-\lambda u_t(e^{u_t^2}-1)\Big)dV_{\mathbb{B}^2}\\
			& \leq \int_{\Sigma_{t}^{u}}G_{\lambda}(x,y)\Big(\lambda (e^{u^2}-1)(2u^2+1)\Big)(u-u_t)dV_{\mathbb{B}^2}.
	\end{split}\end{equation}
	Through Hardy-Littlewood-Sobolev inequality on the hyperbolic space, we derive that
	\begin{equation}\begin{split}
			\|u(x)-u_{t}(x)\|_{L^q(\Sigma_{t}^{u})}&\leq C_{\lambda, q}^{-1}\|\Big(\lambda (e^{u^2}-1)(2u^2+1)\Big)(u-u_t)\|_{L^{q'}(\Sigma_{t}^{u})}\\
			&\lesssim \|\Big((e^{u^2}-1)(2u^2+1)\Big)\|_{L^{r}(\Sigma_{t}^{u})}\|u-u_t\|_{L^q(\Sigma_{t}^{u})},
	\end{split}\end{equation}
	where $2\leq q<+\infty$ and $\frac{1}{q'}=\frac{1}{q}+\frac{1}{r}$. By Moser-Trudinger inequality in the Poincar\'{e} disk $\mathbb{B}^2$ (see \cite{MS-2010} and \cite{LuTang})
	\begin{equation}\label{MT}
		\sup_{u\in W^{1,2}(\mathbb{B}^2),\ \|\nabla_{\mathbb{B}^2}u\|_{L^2(\mathbb{B}^2)}^2\leq1}\int_{\mathbb{B}^2}\big(e^{4\pi u^2}-1\big)dV_{\mathbb{B}^2}<+\infty,
	\end{equation}
	we know that $(e^{u^2}-1)(2u^2+1)$ is $L^r$ integrable and $\|(e^{u^2}-1)(2u^2+1)\|_{L^{r}(\Sigma_{t}^{u})}$ is sufficiently small for sufficiently negative $t$. This deduces that $\Sigma_{t}^{u}$ must be empty for sufficiently negative $t$. Then we accomplish the proof of Step 1.
	\vskip0.1cm
	
	Step 2. Inequality \eqref{starting} provides a starting point to move the plane $\phi(U_{\nu}^t)$. Define
	$$t_0=\sup \{\ t:\ u_s(x)\geq u(x), \ s\leq t,\ \ \forall x\in \Sigma_{s}\}.$$
	We will show that $u_{t_0}(x)=u(x)$ for any $x\in \Sigma_{t_0}$. Obviously $$u_{t_0}(x)\geq u(x),\ \ \forall x\in \Sigma_{t_0}.$$ If there exists some point $x_0$ such that $u_{t_0}(x_0)>u(x_0)$, then by \eqref{identity}, we get
	$$u_{t_0}(x)>u(x),\ \ \forall x\in \Sigma_{t_0}.$$ Next, we will show that the plane can be moved further to the right. More precisely, there exists an $\epsilon>0$ such that for any $t\in [t_0, t_0+\epsilon)$, there holds
	$$u_{t}(x)\geq u(x),\ \ \forall x\in \Sigma_{t}.$$
	Let $$\overline{\Sigma_{t_0}^{u}}=\{x\in \Sigma_{t_0}\ |\ u(x)\geq u_{t_0}(x)\}.$$
	Obviously $\overline{\Sigma_{t_0}^{u}}$ has the measure zero and $\lim\limits_{t\rightarrow t_0}\Sigma_{t}^{u}\subseteq \overline{\Sigma_{t_0}^{u}}$. Then it follows that there exists an $\epsilon>0$ such that for any $t\in [t_0, t_0+\epsilon)$, the integral $\|\big((e^{u^2}-1)(2u^2+1)\big)\|_{L^{r}(\Sigma_{t}^{u})}$ is sufficiently small. Recall that  $$\|u(x)-u_{t}(x)\|_{L^q(\Sigma_{t}^{u})}\lesssim  \|\big((e^{u^2}-1)(2u^2+1)\big)\|_{L^{r}(\Sigma_{t}^{u})}\|u-u_t\|_{L^q(\Sigma_{t}^{u})}.$$
	This deduces that for any $t\in [t_0, t_0+\epsilon)$, there holds $u_{t}(x)\geq u(x),\ \forall x\in \Sigma_{t}$, which is a contradiction with the definition of $t_0$. Hence we conclude that
	$$u_{t_0}(x)=u(x),\ \ \forall x\in \Sigma_{t_0}.$$
	
	Step 3. Since the direction $\nu$ is arbitrary, we conclude that there exists $P\in \mathbb{B}^2$ such that $u(x)$ is strictly decreasing about the geodesic distance $\rho(x)=d(x,P)$.
\end{proof}

\section{Quantization analysis of Moser-Trudinger equations}\label{s4}
From Proposition \ref{pro}, we know that a positive solution to equation \eqref{eq1} is radially decreasing about the origin up to some M\"{o}bius transformation. Without loss of generality, we may assume that the positive solutions are radially decreasing.
In this section, we are devoted to investigating blow-up behaviors of equation \eqref{eq1} inspired from \cite{MM-JEMS}, which is crucial to prove Theorem \ref{thm1}.

%Denote $c_{\lambda}:=u_{\lambda}(0)=\max_{\mathbb{B}^2}u_{\lambda}$.
\begin{proposition}\label{pro1}
	Assume that $u_k\subset W^{1,2}(\mathbb{B}^2)$ are a family of solutions satisfying
	\begin{equation}\label{blowup}
		\left\{
		\begin{aligned}
			&-\Delta_{\mathbb{B}^2}u_k=\lambda_k u_k e^{u_k^2},\ &x\in\mathbb{B}^2,\\
			&u_k>0,\ &x\in\mathbb{B}^2,\\
			&u_k\to0,\ \text{when}\ \rho(x)\to\infty,\\
			&\|\nabla_{\mathbb{B}^2}u_k\|_{L^2(\mathbb{B}^2)}^2\leq M_0,
		\end{aligned}
		\right.
	\end{equation}
	where $0<\lambda_k<\frac{1}{4}$. If $c_k:=u_k(0)=\max_{\mathbb{B}^2}u_k\to\infty$ as $k\to\infty$, then up to a subsequence, we have $\lambda_k\to0$, $\|\nabla_{\mathbb{B}^2} u_{k}\|_{L^2(\mathbb{B}^2)}^2\to4\pi$, $|\nabla_{\mathbb{B}^2} u_{k}|^2 dV_{\mathbb{B}^2}\rightharpoonup4\pi\delta_0$ and $\lambda_k u_{k}^2e^{u_{k}^2} dV_{\mathbb{B}^2}\rightharpoonup4\pi\delta_0$.
\end{proposition}

In order to prove Proposition \ref{pro1}, we rewrite \eqref{blowup} into the following equation in the unit ball $B_1$ of Euclidean space $\mathbb{R}^2$ by the conformal invariance between $(B_1,dV_{\mathbb{B}^2})$ and $(B_1,dx)$:
\begin{equation}\label{eq4}
	\left\{
	\begin{aligned}
		&-\Delta u_k=\lambda_k u_k e^{u_k^2}\left(\frac{2}{1-|x|^2}\right)^2,&x&\in B_1,\\
		&u_k>0,&x&\in B_1,\\
		&u_k=0,&x&\in\partial B_1.\\
	\end{aligned}
	\right.
\end{equation}

\begin{lemma}\label{lem1}
	Let $r_{k}>0$ be such that $\lambda_k r_{k}^2 c_{k}^2e^{c_{k}^2}=1$. If $c_k\to\infty$ as $k\to\infty$, then there exist constants $\alpha\in(0,1)$  and $\tilde{C}>0$ such that $r_{k}\leq \tilde{C}e^{-\frac{1}{2}\alpha^2c_{k}^2}$,
	\begin{equation}\label{eq5}
		\eta_{k}(x):=
		c_k(u_{k}(r_{k} x)-c_{k})\to\eta_0(x):=- \ln(1+|x|^2)\quad\text{in}\ C_{\rm{loc}}^1(\mathbb{R}^2),
	\end{equation}
	and
	\begin{equation}\label{eq6}
		\lim_{R\to\infty}\lim_{k\to\infty}\int_{B_{R r_{k}}}\lambda_k u_{k}^2 e^{u_{k}^2}\left(\frac{2}{1-|x|^2}\right)^2 dx=\int_{\mathbb{R}^2}4e^{2\eta_0} dx=4\pi.
	\end{equation}
\end{lemma}

\begin{proof}
	First, we claim that after passing to a subsequence, $\|\nabla_{\mathbb{B}^2} u_{k}\|_{L^2(\mathbb{B}^2)}^2\geq 4\pi$. Assume by contradiction that $\|\nabla_{\mathbb{B}^2} u_{k}\|_{L^2(\mathbb{B}^2)}^2<4\pi$, that is, $$\int_{B_1}|\nabla u_{k}|^2 dx<4\pi.$$
	%For a fixed small constant $\delta>0$, define a cut-off function $\varphi\in C_0^{\infty}(B_1)$ satisfying $0\leq\varphi\leq1$, $\varphi=1$ in $B_{\delta}$ and $\varphi=0$ in $B_1\backslash B_{2\delta}$.
	%	Then we have $$\int_{B_{\delta}}|\nabla u_{k}|^2 dx<4\pi$$ for any fixed constant $\delta\in(0,1)$. By Moser-Trudinger inequality and Sobolev inequality in bounded domain of $\mathbb{R}^2$,
	%	\begin{equation}\label{eq7}
		%		\sup_{u\in W^{1,2}(B_{\delta}),\ \|\nabla u\|_{L^2(B_{\delta})}^2\leq4\pi}\int_{B_{\delta}}e^{u^2} dx<\infty,
		%	\end{equation}
	For any $\delta\in(0,1)$, applying Moser-Trudinger inequality and Sobolev inequality in bounded domains of $\mathbb{R}^2$ respectively,
	we get that $e^{u_{k}^2}$ and $u_{k}$ are bounded in $L^p(B_{\delta})$ for any $p>1$. Then it follows that $u_{k}e^{u_{k}^2}\in L^r(B_{\delta})$ for any $r>1$. This deduces that $u_{k}\in L^{\infty}(B_{\delta})$ by Moser iteration techniques, which is a contradiction with $u_{k}(0)=c_{k}\to\infty$.
	
	Recalling that $\|\nabla_{\mathbb{B}^2}u_{k}\|_{L^2(\mathbb{B}^2)}^2\leq M_0$, so we can choose $\alpha\in(0,1)$ such that $\|\nabla_{\mathbb{B}^2}(\alpha u_{k})\|_{L^2(\mathbb{B}^2)}^2<4\pi$. By Moser-Trudinger inequality \eqref{MT},
	%	\begin{equation}\label{eq8}
		%\sup_{u\in W^{1,2}(\mathbb{B}^2),\ \|\nabla_{\mathbb{B}^2} %u\|_{L^2(\mathbb{B}^2)}^2\leq4\pi}\int_{\mathbb{B}^2}(e^{u^2}-1) dV_{\mathbb{B}^2}<\infty,
		%	\end{equation}
	we get
	\begin{equation}\label{eq9}\nonumber
		\int_{\mathbb{B}^2}\left(e^{\alpha^2u_{k}^2}-1\right) dV_{\mathbb{B}^2}\leq C_0
	\end{equation}
	for some $C_0>0$, which together with Poincar\'{e} inequality \eqref{Poincare}
	%	\begin{equation}\label{eq10}
		%		\frac{1}{4}\int_{\mathbb{B}^2}|u|^2 dV_{\mathbb{B}^2}\leq\int_{\mathbb{B}^2}|\nabla_{\mathbb{B}^2} u|^2 dV_{\mathbb{B}^2},\quad\text{for any}\ u\in W^{1,2}(\mathbb{B}^2),
		%	\end{equation}
	implies that
	\begin{equation}\label{eq11}\nonumber
		\begin{aligned}
			\frac{1}{\lambda_k}\leq{}&\frac{e^{(1-\alpha^2)c_{k}^2}\int_{\mathbb{B}^2}u_{k}^2e^{\alpha^2 u_{k}^2} dV_{\mathbb{B}^2}}{\int_{\mathbb{B}^2}|\nabla_{\mathbb{B}^2} u_{k}|^2 dV_{\mathbb{B}^2}}\\
			\leq{}&\frac{e^{(1-\alpha^2)c_{k}^2}\left(\int_{\mathbb{B}^2}c_{k}^2(e^{\alpha^2 u_{k}^2}-1) dV_{\mathbb{B}^2}+\int_{\mathbb{B}^2}u_{k}^2 dV_{\mathbb{B}^2}\right)}{\int_{\mathbb{B}^2}|\nabla_{\mathbb{B}^2} u_k|^2 dV_{\mathbb{B}^2}}\\
			\leq{}&\frac{e^{(1-\alpha^2)c_{k}^2}(c_{k}^2C_0+4M_0)}{4\pi},
		\end{aligned}
	\end{equation}
	where the claim above is used in the last inequality.
	Considering $\lambda_k r_{k}^2 c_{k}^2e^{c_{k}^2}=1$, we deduce that $r_{k}\leq \tilde{C}e^{-\frac{1}{2}\alpha^2c_{k}^2}$ directly where $\tilde{C}=((C_0+4M_0)/4\pi)^{\frac{1}{2}}$.
	
	Set $v_{k}:=u_{k}(r_{k} x)-c_{k}$. We claim that $v_{k}\to0$ in $C_{\rm{loc}}^1(\mathbb{R}^2)$ as $k\to\infty$. Indeed, $v_k$ satisfies
	\begin{equation}\label{eq12}\nonumber
		-\Delta v_{k}(x)=\frac{1}{c_{k}}\frac{u_{k}(r_{k} x)}{c_{k}}\exp\left\{\displaystyle(u_{k}(r_{k} x)+c_{k})(u_{k}(r_{k} x)-c_{k})\right\}\left(\frac{2}{1-r_{k}^2|x|^2}\right)^2\to0
	\end{equation}
	in $C_{\rm{loc}}^0(\mathbb{R}^2)$,
	since $0\leq\frac{u_{k}(r_{k} x)}{c_{k}}\leq1$, $(u_{k}(r_{k} x)+c_{k})(u_{k}(r_{k} x)-c_{k})\leq0$, $\frac{1}{c_{k}}\to0$ and $\left(\frac{2}{1-r_{k}^2|x|^2}\right)^2\to4$ in $C_{\rm{loc}}^0(\mathbb{R}^2)$. Observing that $v_{k}\leq0$, $\Delta v_{k}$ is locally bounded and $v_{k}(0)=0$, thus the claim follows from Harnack inequality.
	
	As a result, $\eta_k$ satisfies
	\begin{equation}\label{eq13}\nonumber
		-\Delta \eta_{k}=\frac{u_{k}(r_{k} x)}{c_{k}}\exp\left\{\left(\frac{u_{k}(r_{k} x)}{c_{k}}+1\right)\eta_{k}\right\}\left(\frac{2}{1-r_{k}^2|x|^2}\right)^2\quad\text{in}\ B_{\frac{1}{r_{k}}},
	\end{equation}
	where
	\begin{equation}\label{eq14}\nonumber
		\frac{u_{k}(r_{k} x)}{c_{k}}\to1,\quad\left(\frac{2}{1-r_{k}^2|x|^2}\right)^2\to4\quad\text{in}\ C_{\rm{loc}}^0(\mathbb{R}^2).
	\end{equation}
	Observing that $\eta_{k}\leq0$, $\Delta \eta_{k}$ is locally bounded and $\eta_{k}(0)=0$, we have $\eta_{k}\to\eta^\ast$ in $C_{\rm{loc}}^1(\mathbb{R}^2)$ by Harnack inequality, where $-\Delta\eta^\ast=4e^{2\eta^\ast}$ and $\eta^\ast(0)=0$. On the other hand, we know that
	\begin{equation}\label{eq15}
		-\Delta\eta_0=4e^{2\eta_0}\quad\text{in}\ \mathbb{R}^2,\quad\eta_0(0)=0.
	\end{equation}
	Hence by the uniqueness of solutions to the Cauchy problem \eqref{eq15}, we obtain $\eta^\ast=\eta_0$.
	
	Finally, \eqref{eq6} follows from Lebesgue dominated convergence theorem.
\end{proof}

\begin{lemma}\label{lem2}
	Let $w_{k}:=c_{k}^2(\eta_{k}-\eta_0)$. If $c_k\to\infty$ as $k\to\infty$, then we have $w_{k}\to w_0$ in $C_{\rm{loc}}^1(\mathbb{R}^2)$, where
	\begin{equation}\label{eq16}
		w_0(r):=\eta_0(r)+\frac{2r^2}{1+r^2}-\frac{1}{2}\eta_0^2(r)+\frac{1-r^2}{1+r^2}\int_{1}^{1+r^2}\frac{ \ln t}{1-t} dt
	\end{equation}
	is the unique solution to the ODE
	\begin{equation}\label{eq17}
		-\Delta w_0=4e^{2\eta_0}(\eta_0+\eta_0^2+2w_0),\quad w_0(0)=0,\quad w_0'(0)=0.
	\end{equation}
	Moreover, $w_0$ satisfies
	\begin{equation}\label{eq18}
		\int_{\mathbb{R}^2}(-\Delta w_0) dx=4\pi,
	\end{equation}
	and
	\begin{equation}\label{eq19}
		\sup_{r\in[0,\infty)}|w_0(r)-\eta_0(r)|<\infty.
	\end{equation}
\end{lemma}

\begin{proof}
	Notice that $\eta_k$ satisfies
	\begin{equation}\label{eq20}
		-\Delta\eta_{k}=\left(1+\frac{\eta_{k}}{c_{k}^2}\right)\exp\left\{\left(2+\frac{\eta_{k}}{c_{k}^2}\right)\eta_{k}\right\}\left(\frac{2}{1-r_{k}^2|x|^2}\right)^2.
	\end{equation}
	Using \eqref{eq15} and \eqref{eq20} we compute
	\begin{equation}\label{eq21}
		\begin{aligned}
			-\Delta w_{k}={}&c_{k}^2\left[\left(1+\frac{\eta_{k}}{c_{k}^2}\right)\left(\frac{2}{1-r_{k}^2|x|^2}\right)^2\exp\left\{\left(2+\frac{\eta_{k}}{c_{k}^2}\right)\eta_{k}\right\}-4e^{2\eta_0}\right]\\
			={}&4e^{2\eta_0}\Bigg[c_{k}^2\left(1+\frac{\eta_0}{c_{k}^2}+\frac{\eta_{k}-\eta_0}{c_{k}^2}\right)\left(\frac{1}{1-r_{k}^2|x|^2}\right)^2\\
			{}&\quad\quad\quad\cdot\exp\left\{2(\eta_{k}-\eta_0)+\frac{\eta_0^2+2\eta_0(\eta_{k}-\eta_0)+(\eta_{k}-\eta_0)^2}{c_{k}^2}\right\}-c_{k}^2\Bigg].
		\end{aligned}
	\end{equation}
	By Lemma \ref{lem1} we know for every $R>0$, $\eta_{k}(r)-\eta_0(r)=o(1)$ as $k\to\infty$ uniformly for $r\in[0,R]$. By using a Taylor expansion:
	\begin{align*}\label{eq22}\nonumber
		{}&\exp\left\{2(\eta_{k}-\eta_0)+\frac{\eta_0^2+2\eta_0(\eta_{k}-\eta_0)+(\eta_{k}-\eta_0)^2}{c_{k}^2}\right\}\\
		={}&1+2(\eta_{k}-\eta_0)+\frac{\eta_0^2}{c_{k}^2}+o(1)c_{k}^{-2}+o(1)(\eta_{k}-\eta_0),
	\end{align*}
	with $o(1)\to0$ as $k\to\infty$ uniformly for $r\in[0,R]$, we get
	\begin{equation}\label{eq23}
		-\Delta w_{k}=4e^{2\eta_0}(\eta_0+\eta_0^2+2w_{k}+o(1)+o(1)w_{k}),
	\end{equation}
	with $o(1)\to0$ as $k\to\infty$ uniformly for $r\in[0,R]$. Observing that $w_{k}(0)=w_{k}'(0)=0$, from ODE theory it follows that $w_{k}(r)$ is uniformly bounded in $r\in[0,R]$ and by elliptic estimates we have $w_{k}\to w^\ast$ in $C_{\rm{loc}}^1(\mathbb{R}^2)$, where $w^\ast$ satisfies \eqref{eq17}.
	
	Noting that the solution to the Cauchy problem \eqref{eq17} is unique, hence in order to prove $w^\ast=w_0$, it is enough to show that $w_0$ solves \eqref{eq17}.
	%Firstly, it is easy to see that $w_0(0)=0$. Then, directly computing yields that
	%	\begin{equation}\label{eq24}\nonumber
		%		\frac{d}{dr}\left(-\frac{1}{2}\eta_0^2(r)+\frac{1-r^2}{1+r^2}\int_{1}^{1+r^2}\frac{ \ln t}{1-t} dt\right)=-\frac{2 \ln(1+r^2)}{r(1+r^2)}-\frac{4r}{(1+r^2)^2}\int_{1}^{1+r^2}\frac{ \ln t}{1-t} dt,
		%	\end{equation}
	%	we have
	First, direct computing yields that $w_0(0)=0$,
	\begin{equation}\label{eq25}
		w_0'(r)=\frac{2r(1-r^2)}{(1+r^2)^2}-\frac{2 \ln(1+r^2)}{r(1+r^2)}-\frac{4r}{(1+r^2)^2}\int_{1}^{1+r^2}\frac{ \ln t}{1-t} dt
	\end{equation}
	and $w_0'(0)=0$. So, using $\Delta w_0(r)=w_0''(r)+\frac{1}{r}w_0'(r)$, we get
	\begin{equation}\label{eq26}\nonumber
		\begin{aligned}
			-\Delta w_0={}&\frac{16r^2}{(1+r^2)^3}-\frac{12 \ln(1+r^2)}{(1+r^2)^2}+\frac{8(1-r^2)}{(1+r^2)^3}\int_{1}^{1+r^2}\frac{ \ln t}{1-t} dt\\
			={}&4e^{2\eta_0}\left[\frac{4r^2}{1+r^2}+3\eta_0+\frac{2(1-r^2)}{1+r^2}\int_{1}^{1+r^2}\frac{ \ln t}{1-t} dt\right]\\
			={}&4e^{2\eta_0}(\eta_0+\eta_0^2+2w_0).
		\end{aligned}
	\end{equation}
	By the divergence theorem and \eqref{eq25} we get directly
	\begin{equation}\label{eq27}\nonumber
		\int_{\mathbb{R}^2}(-\Delta w_0) dx=-\lim_{r\to\infty}2\pi rw_0'(r)=4\pi.
	\end{equation}
	Moreover, from \eqref{eq25} it is easily seen that
	\begin{equation}\label{eq28}\nonumber
		\lim_{r\to0^{+}}(1+r^2)(w_0'(r)-\eta_0'(r))=\lim_{r\to\infty}(1+r^2)(w_0'(r)-\eta_0'(r))=0,
	\end{equation}
	which implies that
	\begin{equation}\label{eq29}
		\left|w_0'(r)-\eta_0'(r)\right|\leq\frac{C}{1+r^2}\quad\text{for}\ r\in[0,\infty).
	\end{equation}
	Therefore \eqref{eq19} follows from integrating \eqref{eq29} in $r$.
\end{proof}

\begin{lemma}\label{lem3}
	Fix $R_0\in(0,\infty)$ such that $w_0\leq-1$ on $[R_0,\infty)$, where $w_0$ is given by \eqref{eq16}. If $c_k\to\infty$ as $k\to\infty$, then we have
	\begin{equation}\label{eq30}
		\eta_{k}(r)\leq \eta_0(r)\quad\text{for}\ r\in[R_0,r_{k}^{-1}),
	\end{equation}
	or equivalently,
	\begin{equation}\label{eq31}
		u_{k}(r)\leq c_{k}-\frac{1}{c_{k}} \ln\left(1+\frac{r^2}{r_{k}^2}\right)\quad\text{for}\ r\in[R_0r_{k},1).
	\end{equation}
\end{lemma}

\begin{proof}
	Let $\eta:=\eta_0+\frac{w_0}{c_{k}^2}+\frac{\phi}{c_{k}^4}$ and
	\begin{equation}\label{eq32}\nonumber
		\Phi_{k}(r,\phi):=c_{k}^4\left[\left(1+\frac{\eta}{c_{k}^2}\right)\left(\frac{2}{1-r_{k}^2r^2}\right)^2\exp\left\{\left(2+\frac{\eta}{c_{k}^2}\right)\eta\right\}+\Delta\eta_0+\frac{\Delta w_0}{c_{k}^2}\right].
	\end{equation}
	Write $h:=(2+\frac{\eta}{c_{k}^2})\eta-2\eta_0$ and $\xi:=1+ \ln(1+r^2)$. By Lemmas \ref{lem1} and \ref{lem2} we know
	$|\eta_0|+|w_0|=O(\xi)$. Thus, with a Taylor expansion we deduce that
	\begin{equation}\label{eq33}\nonumber
		\begin{aligned}
			{}&\left(1+\frac{\eta}{c_{k}^2}\right)\left(\frac{2}{1-r_{k}^2r^2}\right)^2\exp\left\{\left(2+\frac{\eta}{c_{k}^2}\right)\eta\right\}\\
			={}&4e^{2\eta_0}\left(1+\frac{\eta_0}{c_{k}^2}+\frac{w_0}{c_{k}^4}+\frac{\phi}{c_{k}^6}\right)(1+O(r_{k}^2r^2))(1+h+O(h^2))\\
			={}&4e^{2\eta_0}\left(1+\frac{2w_0+\eta_0+\eta_0^2}{c_{k}^2}+\frac{2\phi}{c_{k}^4}+O(c_{k}^{-4}\phi)\left(O(c_{k}^{-4}\phi)+O(c_{k}^{-2}\xi^2)\right)+O(c_{k}^{-4}\xi^4)\right)\\
			{}&\cdot(1+O(r_{k}^2r^2))\\
			={}&4e^{2\eta_0}\left(1+\frac{2w_0+\eta_0+\eta_0^2}{c_{k}^2}+\frac{2\phi}{c_{k}^4}\right)\\
			{}&+4e^{2\eta_0}\left(O(c_{k}^{-4}\phi)\left(O(c_{k}^{-4}\phi)+O(c_{k}^{-2}\xi^2)\right)+O(c_{k}^{-4}\xi^4)+O(r_{k}^2r^2)\right)
		\end{aligned}
	\end{equation}
	as long as $|h|\leq1$ and $r_{k}r<1$, which is true provided that
	\begin{equation}\label{eq34}
		|c_{k}^{-4}\phi|\leq\delta,\quad c_{k}^{-1}\xi\leq\delta,\quad r_{k}r\leq\delta,
	\end{equation}
	for some $\delta>0$ small enough. Recall that $r_{k}^2=O(e^{-\alpha^2c_{k}^2})$ from Lemma \ref{lem1}. Then it follows from \eqref{eq15} and \eqref{eq17} that
	\begin{equation}\label{eq35}
		\Phi_{k}(r,\phi)=4e^{2\eta_0}\left(2\phi+O(\phi)\left(O(c_{k}^{-4}\phi)+O(c_{k}^{-2}\xi^2)\right)+O(\xi^4)+O(c_{k}^4e^{-\alpha^2c_{k}^2}r^2)\right)
	\end{equation}
	provided that \eqref{eq34}.
	
	Set $\eta_{k}:=\eta_0+\frac{w_0}{c_{k}^2}+\frac{\phi_{k}}{c_{k}^4}$. Then \eqref{eq20} is equivalent to
	\begin{equation}\label{eq36}
		-\Delta\phi_{k}=\Phi_{k}(r,\phi_{k}),\quad\phi_{k}(0)=0,\quad\phi_{k}'(0)=0.
	\end{equation}
	We now proceed to bound $\phi_{k}$ by the contraction mapping theorem. We restrict our attention to an integral $[0,s_{k}]$ with $s_{k}=o(1)e^{c_{k}}$ and functions $\phi:[0,s_{k}]\to\mathbb{R}$ satisfying $|\phi|\leq O(\xi)$, so that \eqref{eq34} holds as $k\to\infty$. Taking $\phi=\phi_{k}$ and $\psi=r\phi'$, then \eqref{eq36} is equivalent to
	\begin{equation}\label{eq37}
		\left\{
		\begin{aligned}
			&\phi'=\frac{1}{r}\psi,\\
			&\psi'=-r\Phi_{k}(r,\phi),\\
			&\phi(0)=\psi(0)=0.
		\end{aligned}
		\right.
	\end{equation}
	By \eqref{eq35} we see that, for a fixed time $T$,
	\begin{equation}\label{eq38}\nonumber
		\Phi_{k}(r,\phi)=4e^{2\eta_0}(2\phi+o(\phi)+O(1))
	\end{equation}
	uniformly on $[0,T]$. From ODE theory we deduce that $\phi$ and $\psi$ are uniformly bounded in $[0,T]$, i.e.
	\begin{equation}\label{eq39}
		|\phi(r)|\leq C(T),\quad|\psi(r)|\leq C(T),\quad\text{for}\ r\in[0,T],
	\end{equation}
	uniformly in $c_{k}$. Considering the functions
	\begin{equation}\label{eq40}\nonumber
		\begin{aligned}
			&F_{(\phi,\psi)}(r):=\phi(T)+\int_T^r\frac{1}{s}\psi(s) ds,\quad r\geq T,\\
			&G_{(\phi,\psi)}(r):=\psi(T)-\int_T^rs\Phi_{k}(s,\phi(s)) ds,\quad r\geq T,
		\end{aligned}
	\end{equation}
	and fixing $S=s_{k}>T$ with $s_{k}=o(1)e^{c_{k}}$, we define the norms
	\begin{equation}\label{eq41}\nonumber
		\|f\|_1:=\sup_{r\in(T,S]}\left|\frac{f(r)}{ \ln r- \ln T}\right|,\quad\|f\|_2:=2\sup_{r\in[T,S]}|f(r)|,
	\end{equation}
	for any function $f$. For a large constant $M>0$ to be fixed later, we define
	\begin{equation}\label{eq42}\nonumber
		B_M:=\{(\phi,\psi):\|\phi-\phi_{k}(T)\|_1\leq M,\ \|\psi\|_2\leq M,\ \phi(T)=\phi_{k}(T),\ \psi(T)=T\phi_{k}'(T)\}.
	\end{equation}
	
	Firstly, we claim that $(\phi,\psi)\to(F_{(\phi,\psi)},G_{(\phi,\psi)})$ sends $B_M$ into itself for suitable $M$ and $T$. Indeed, for $(\phi,\psi)\in B_M$, we have
	\begin{equation}\label{eq43}\nonumber
		|F_{(\phi,\psi)}(r)-\phi_{k}(T)|\leq\int_T^r\frac{1}{s}|\psi(s)| ds\leq\frac{1}{2}M( \ln r- \ln T),
	\end{equation}
	which implies that $\|F_{(\phi,\psi)}-\phi_{k}(T)\|_1\leq\frac{1}{2}M$. Furthermore, for $(\phi,\psi)\in B_M$, we have
	\begin{equation}\label{eq44}\nonumber
		|\phi(r)|\leq|\phi_{k}(T)|+|\phi(r)-\phi_{k}(T)|\leq C(T)+M( \ln r- \ln T)\leq C(T)+M \ln r
	\end{equation}
	for any $r\in[T,S]$. Then it follows from \eqref{eq35} that
	\begin{equation}\label{eq45}\nonumber
		\begin{aligned}
			|G_{(\phi,\psi)}(r)|\leq{}&|\psi(T)|+\int_T^r\frac{4s}{(1+s^2)^2}\left(2|\phi(s)|+o(\phi(s))+O(\xi^4(s))\right) ds\\
			\leq{}& C(T)+\int_T^{\infty}\frac{12s(C(T)+M \ln s)}{(1+s^2)^2} ds+\int_T^{\infty}\frac{4sC(1+ \ln(1+s^2))^4}{(1+s^2)^2} ds\\
			={}&C(T)\left(1+\int_T^{\infty}\frac{12s}{(1+s^2)^2} ds\right)\\
			{}&+M\int_T^{\infty}\frac{12s \ln s}{(1+s^2)^2} ds+C\int_T^{\infty}\frac{4s(1+ \ln(1+s^2))^4}{(1+s^2)^2} ds\\
		\end{aligned}
	\end{equation}
	for some fixed $C$ independent of $M$ and $c_{k}$. First choosing $T\geq1$ so large that
	\begin{equation}\label{eq46}
		\int_T^{\infty}\frac{12s \ln s}{(1+s^2)^2} ds<\frac{1}{2},
	\end{equation}
	and then $M$ such that
	\begin{equation}\label{eq47}\nonumber
		C(T)\left(1+\int_T^{\infty}\frac{12s}{(1+s^2)^2} ds\right)+C\int_T^{\infty}\frac{4s(1+ \ln(1+s^2))^4}{(1+s^2)^2} ds\leq\frac{1}{2}M,
	\end{equation}
	we obtain $\|G_{(\phi,\psi)}\|_2\leq M$, which implies $(F_{(\cdot,\cdot)},G_{(\cdot,\cdot)})$ maps $B_M$ into itself.
	
	Next, we show $(\phi,\psi)\to(F_{(\phi,\psi)},G_{(\phi,\psi)})$ is a contraction mapping on $B_M$. In fact, for $(\phi,\psi),(\tilde{\phi},\tilde{\psi})\in B_M$, we have
	\begin{equation}\label{eq48}\nonumber
		|F_{(\phi,\psi)}(r)-F_{(\tilde{\phi},\tilde{\psi})}(r)|\leq\int_T^r\frac{1}{s}|\psi(s)-\tilde{\psi}(s)| ds\leq\frac{1}{2}\|\psi-\tilde{\psi}\|_2( \ln r- \ln T),
	\end{equation}
	which implies that $\|F_{(\phi,\psi)}-F_{(\tilde{\phi},\tilde{\psi})}\|_1\leq\frac{1}{2}\|\psi-\tilde{\psi}\|_2$. On the other hand, from \eqref{eq35} and \eqref{eq46} we deduce that
	\begin{equation}\label{eq49}\nonumber
		\begin{aligned}
			|G_{(\phi,\psi)}(r)-G_{(\tilde{\phi},\tilde{\psi})}(r)|\leq{}&\int_T^rs|\Phi_{k}(s,\phi(s))-\Phi_{k}(s,\tilde{\phi}(s))| ds\\
			\leq{}&\int_T^r\frac{4s}{(1+s^2)^2}(2+o(1))|\phi(s)-\tilde{\phi}(s)| ds\\
			\leq{}&\frac{1}{2}\|\phi-\tilde{\phi}\|_2\int_T^{\infty}\frac{12s \ln s}{(1+s^2)^2} ds\leq\frac{1}{4}\|\phi-\tilde{\phi}\|_2.
		\end{aligned}
	\end{equation}
	Hence $\|G_{(\phi,\psi)}-G_{(\tilde{\phi},\tilde{\psi})}\|_2\leq\frac{1}{2}\|\psi-\tilde{\psi}\|_1$, and $(F_{(\cdot,\cdot)},G_{(\cdot,\cdot)})$ is exactly a contraction mapping on $B_M$. In particular, $(F_{(\cdot,\cdot)},G_{(\cdot,\cdot)})$ has a fixed point $(\phi_0,\psi_0)\in B_M$ satisfying \eqref{eq37}. Then, by the uniqueness of solutions to the Cauchy problem \eqref{eq37}, we obtain $(\phi_0,\psi_0)=(\phi_{k},r\phi_{k}')$ for $r\in[T,S]$, and
	\begin{equation}\label{eq50}
		|\phi_{k}(r)|\leq C(T)+M( \ln r- \ln T),\quad|\phi_{k}'(r)|\leq\frac{M}{2r},\quad\text{for}\ r\in[T,S].
	\end{equation}
	
	For $k$ large enough, fix $S=s_{k}=o(1)e^{c_{k}}$ such that $s_{k}\geq2c_{k}$. Since $w_0\leq-1$ on $[R_0,\infty)$, it follows from \eqref{eq39} and \eqref{eq50} that
	\begin{equation}\label{eq51}
		\eta_{k}(r)\leq\eta_0(r)-\frac{1}{c_{k}^2}+\frac{C(T)+M \ln r}{c_{k}^4}<\eta_0(r),\quad\text{for}\ r\in[R_0,s_{k}].
	\end{equation}
	Using \eqref{eq5}, \eqref{eq25} and \eqref{eq50}, we compute
	\begin{equation}\label{eq52}\nonumber
		\begin{aligned}
			&\int_{B_{s_{k}}}(-\Delta\eta_0) dx=\int_{B_{s_{k}}}4e^{2\eta_0} dx=4\pi\left(1-\frac{1}{1+s_{k}^2}\right)=4\pi-\frac{4\pi}{s_{k}^2}+\frac{o(1)}{s_{k}^2},\\
			&\int_{B_{s_{k}}}(-\Delta w_0) dx=-2\pi s_{k}w_0'(s_{k})=4\pi(1+o(1)),\\
			&\left|\int_{B_{s_{k}}}(-\Delta\phi_{k}) dx\right|=2\pi s_{k}|\phi_{k}'(s_{k})|=O(1).
		\end{aligned}
	\end{equation}
	Summing up we deduce that
	\begin{equation}\label{eq53}\nonumber
		\int_{B_{s_{k}}}(-\Delta\eta_{k}) dx=4\pi-\frac{4\pi}{s_{k}^2}+\frac{o(1)}{s_{k}^2}+\frac{4\pi(1+o(1))}{c_{k}^2}+\frac{O(1)}{c_{k}^4}>4\pi
	\end{equation}
	by the choice of $s_{k}$. Therefore we obtain
	\begin{equation}\label{eq54}\nonumber
		-2\pi r\eta_0'(r)=\int_{B_{r}}(-\Delta\eta_0) dx<4\pi<\int_{B_{s_{k}}}(-\Delta\eta_{k}) dx\leq\int_{B_{r}}(-\Delta\eta_{k}) dx=-2\pi r\eta_{k}'(r)
	\end{equation}
	for $r\in[s_{k},r_{k}^{-1})$, which together with \eqref{eq51} implies that $\eta_{k}(r)\leq\eta_0(r)$ for $r\in[R_0,r_{k}^{-1})$.
\end{proof}

\begin{proof}[\textbf{Proof of Proposition \ref{pro1}}]
	Assume that $\rho_{k}=\sqrt{\frac{1}{\lambda_k c_{k}^2}}<1$. Taking $r=\rho_{k}$ in \eqref{eq31}, we have
	\begin{equation}\label{eq55}\nonumber
		u_{k}(\rho_{k})\leq c_{k}-\frac{1}{c_{k}} \ln\left(1+\frac{\rho_{k}^2}{r_{k}^2}\right)=c_{k}-\frac{1}{c_{k}} \ln(1+e^{c_{k}^2})<0,
	\end{equation}
	which contradicts with $u_{k}>0$ in $B_1$. Hence $\rho_{k}\geq1$, so that $\lambda_k\to0$ as $k\to\infty$.
	
	Write $f_{k}(r):=\lambda_k u_{k}^2e^{u_{k}^2}\left(\frac{2}{1-r^2}\right)^2$. By Lemma \ref{lem3} we have, for $r\in[R_0r_{k},1)$,
	\begin{equation}\label{eq56}\nonumber
		\begin{aligned}
			r^2f_{k}(r)={}&\left(\frac{r}{r_{k}}\right)^2\left(\frac{u_{k}(r)}{c_{k}}\right)^2e^{(u_{k}(r)+{c_{k}})(u_{k}(r)-c_{k})}\left(\frac{2}{1-r^2}\right)^2\\
			\leq{}&\left(\frac{r}{r_{k}}\right)^2\left(\frac{u_{k}(r)}{c_{k}}\right)^2\left(1+\frac{r^2}{r_{k}^2}\right)^{-1-\frac{u_{k}(r)}{c_{k}}}\left(\frac{2}{1-r^2}\right)^2\\
			\leq{}&\left(\frac{u_{k}(r)}{c_{k}}\right)^2\left(1+\frac{r^2}{r_{k}^2}\right)^{-\frac{u_{k}(r)}{c_{k}}}\left(\frac{2}{1-r^2}\right)^2.
		\end{aligned}
	\end{equation}
	Define the function
	\begin{equation}\label{eq57}\nonumber
		g_{k}(t):=t^2\left(1+\frac{r^2}{r_{k}^2}\right)^{-t}\left(\frac{2}{1-r^2}\right)^2.
	\end{equation}
	Then for any fixed $r>0$, $g_{k}(t)$ satisfies
	\begin{equation}\label{eq58}\nonumber
		g_{k}\geq0,\quad g_{k}'(0)=0,\quad \lim_{t\to\infty}g_{k}(t)=0,\quad g_{k}'(t)=0\Leftrightarrow t=t_{k}:=\frac{2}{ \ln(1+r^2/r_{k}^2)},
	\end{equation}
	which implies that
	\begin{equation}\label{eq59}
		r^2f_{k}(r)\leq\frac{4}{ \ln^2(1+r^2/r_{k}^2)}\left(\frac{2}{1-r^2}\right)^2,\quad\text{for}\ r\in[R_0r_{k},1).
	\end{equation}
	Moreover, for any fixed $\delta\in(0,1)$, it follows from \eqref{eq59} that
	\begin{equation}\label{eq60}\nonumber
		r^2f_{k}(r)\leq\frac{1}{ \ln^2(r/r_{k})}\frac{4}{(1-\delta^2)^2},\quad\text{for}\ r\in[R_0r_{k},\delta],
	\end{equation}
	so that
	\begin{equation}\label{eq61}
		\begin{aligned}
			{}&\lim_{R\to\infty}\lim_{k\to\infty}\int_{B_{\delta}\backslash B_{Rr_{k}}}\lambda_k u_{k}^2 e^{u_{k}^2}\left(\frac{2}{1-|x|^2}\right)^2 dx\\
			={}&\lim_{R\to\infty}\lim_{k\to\infty}\int_{Rr_{k}}^{\delta} 2\pi rf_{k}(r) dr\\
			\leq{}&\frac{8\pi}{(1-\delta^2)^2}\lim_{R\to\infty}\lim_{k\to\infty}\int_{Rr_{k}}^{\delta}\frac{1}{r \ln^2(r/r_{k})} dr=0.
		\end{aligned}
	\end{equation}
	On the other hand, since $u_{k}(1)=0$ and $\|\nabla_{\mathbb{B}^2}u_{k}\|_{L^2(\mathbb{B}^2)}^2\leq M_0$, we get by H\"{o}lder inequality that, for any $r\in[\delta,1)$,
	\begin{align*}\label{log}\nonumber
		u_{k}(r)=\int_1^ru_{k}'(s) ds\leq{}&\int_{r}^1|u_{k}'(s)| ds\\
		\leq{}&\left(\int_{r}^12\pi s|u_{k}'(s)|^2 ds\right)^{\frac{1}{2}}\left(\frac{1}{2\pi} \ln \frac{1}{r}\right)^{\frac{1}{2}}\leq M_0^{\frac{1}{2}}\left(\frac{1}{2\pi} \ln \frac{1}{r}\right)^{\frac{1}{2}}.
	\end{align*}
	This infers that
	\begin{equation}\label{eq62}
		\begin{aligned}
			{}&\lim_{k\to\infty}\int_{B_1\backslash B_{\delta}}\lambda_k u_{k}^2 e^{u_{k}^2}\left(\frac{2}{1-|x|^2}\right)^2 dx\\
			\leq{}&\lim_{k\to\infty}\lambda_k e^{u_{k}^2(\delta)}\int_{B_1\backslash B_{\delta}}u_{k}^2\left(\frac{2}{1-|x|^2}\right)^2 dx\\
			={}&\lim_{k\to\infty}\lambda_k e^{u_{k}^2(\delta)}\int_{\mathbb{B}^2\backslash B_{\mathbb{B}^2}(0,\tilde{\delta})} u_{k}^2  dV_{\mathbb{B}^2}\\
			\leq{}&\lim_{k\to\infty}4\lambda_k e^{\frac{M_0}{2\pi} \ln\frac{1}{\delta}}\int_{\mathbb{B}^2}|\nabla_{\mathbb{B}^2}u_{k}|^2 dV_{\mathbb{B}^2}=0,
		\end{aligned}
	\end{equation}
	where $\tilde{\delta}$ satisfies $\delta=\frac{e^{\tilde{\delta}}-1}{e^{\tilde{\delta}}+1}$.
	Therefore, combining \eqref{eq6}, \eqref{eq61} and \eqref{eq62} yields $\|\nabla_{\mathbb{B}^2} u_{k}\|_{L^2(\mathbb{B}^2)}^2\to4\pi$, $|\nabla_{\mathbb{B}^2} u_{k}|^2 dV_{\mathbb{B}^2}\rightharpoonup4\pi\delta_0$ and $\lambda_k u_{k}^2e^{u_{k}^2} dV_{\mathbb{B}^2}\rightharpoonup4\pi\delta_0$ as $k\to\infty$.
\end{proof}

Now, we are in a position to prove Theorem \ref{thm1}.
\begin{proof}[\textbf{Proof of Theorem \ref{thm1}}]
	We shall deal with three cases in Theorem \ref{thm1} respectively.
	\vskip0.1cm
	
	\textbf{Case 1:} $\lambda_0=0$.
	We first show that $\|\nabla_{\mathbb{B}^2}u_{k}\|_{L^2(\mathbb{B}^2)}^2$ has a positive lower bound as $k\to\infty$. Assume for contradiction that $\|\nabla_{\mathbb{B}^2}u_{k}\|_{L^2(\mathbb{B}^2)}^2\to0$. Then $e^{u_{k}^2}-1$ is bounded in $L^p(\mathbb{B}^2)$ for any $p>1$ by applying Moser-Trudinger inequality \eqref{MT}. Denote $$\tilde{u}_{k}=\frac{u_{k}}{\|\nabla_{\mathbb{B}^2} u_{k}\|_{L^2(\mathbb{B}^2)}},$$ thus $\|\nabla_{\mathbb{B}^2} \tilde{u}_{k}\|_{L^2(\mathbb{B}^2)}^2=1$. By Poincar\'{e} inequality \eqref{Poincare}, we get $\tilde{u}_{k}$ is bounded in $L^q(\mathbb{B}^2)$ for any $q\geq2$. Since $\lambda_0=0$, we deduce
	\begin{equation}\label{eq65}\nonumber
		\begin{aligned}
			1={}&\lambda_k\int_{\mathbb{B}^2}\tilde{u}_{k}^2e^{u_{k}^2} dV_{\mathbb{B}^2}=\lambda_k\int_{\mathbb{B}^2}\tilde{u}_{k}^2(e^{u_{k}^2}-1) dV_{\mathbb{B}^2}+\lambda_k\int_{\mathbb{B}^2}\tilde{u}_{k}^2 dV_{\mathbb{B}^2}\\
			\leq{}&\lambda_k\left(\int_{\mathbb{B}^2}\tilde{u}_{k}^4 dV_{\mathbb{B}^2}\right)^\frac{1}{2}\left(\int_{\mathbb{B}^2}(e^{u_{k}^2}-1)^2 dV_{\mathbb{B}^2}\right)^\frac{1}{2}+\lambda_k\int_{\mathbb{B}^2}\tilde{u}_{k}^2 dV_{\mathbb{B}^2}\to0,
		\end{aligned}
	\end{equation}
	which is a contradiction.
	
	Now we proceed to prove that $c_k:=u_k(0)=\max\limits_{x\in \mathbb{B}^2}u_k(x)\to\infty$ as $k\to\infty$. Otherwise we have $u_{k}\in L^{\infty}(\mathbb{B}^2)$. Since $\|\nabla_{\mathbb{B}^2} u_{k}\|_{L^2(\mathbb{B}^2)}^2\leq M_0$, by Poincar\'{e} inequality \eqref{Poincare} we derive
	\begin{equation}\label{bd}
		\int_{\mathbb{B}^2}u_{k}^2e^{u_{k}^2} dV_{\mathbb{B}^2}\lesssim	\int_{\mathbb{B}^2}u_{k}^2 dV_{\mathbb{B}^2}\lesssim\int_{\mathbb{B}^2}|\nabla_{\mathbb{B}^2}u_{k}|^2 dV_{\mathbb{B}^2}\lesssim1.
	\end{equation}
	However, from \eqref{eq1} we have
	\begin{equation}\label{eq66}\nonumber
		\int_{\mathbb{B}^2}u_{k}^2e^{u_{k}^2} dV_{\mathbb{B}^2}=\frac{1}{\lambda_k}\int_{\mathbb{B}^2}|\nabla_{\mathbb{B}^2}u_{k}|^2 dV_{\mathbb{B}^2}\to\infty\quad\text{as}\ k\to\infty,
	\end{equation}
	contradicting \eqref{bd}. This together with Proposition \ref{pro1} accomplishes the proof of Case 1.
	\vskip 0.1cm
	
	\textbf{Case 2:} $\lambda_0\in(0,\frac{1}{4})$. Since $\|\nabla_{\mathbb{B}^2}u_{k}\|_{L^2(\mathbb{B}^2)}^2\leq M_0$, there exist $u_0\in W^{1,2}(\mathbb{B}^2)$ and a subsequence still denoted by $u_k$ such that $u_{k}\rightharpoonup u_0$ in $W^{1,2}(\mathbb{B}^2)$ as $k\to\infty$. Moreover, from Proposition \ref{pro1} we have $u_{k}\in L^{\infty}(\mathbb{B}^2)$, which leads to $u_{k}\to u_0$ in $C^2(\mathbb{B}^2)$ by standard elliptic regularity theory and thus $u_0$ satisfies \eqref{eq3}.

	Now we show that $u_{k}\to u_0$ in $W^{1,2}(\mathbb{B}^2)$. Since $\lambda_0<\frac{1}{4}$, by the equivalence of norms in $W^{1,2}(\mathbb{B}^2)$, we just need to show
	\begin{equation}
		\lim_{k\to\infty}\int_{\mathbb{B}^2}\left(|\nabla_{\mathbb{B}^2}u_{k}|^2-\lambda_0u_{k}^2\right)dV_{\mathbb{B}^2}=
		\int_{\mathbb{B}^2}\left(|\nabla_{\mathbb{B}^2}u_{0}|^2-\lambda_0u_{0}^2\right)dV_{\mathbb{B}^2}.
	\end{equation}
	From \eqref{eq1} and \eqref{eq3}, it suffices to prove that
	\begin{equation}\label{eq67}
		\lim_{k\to\infty}\int_{\mathbb{B}^2}\lambda_k u_{k}^2\big(e^{u_{k}^2}-1\big)dV_{\mathbb{B}^2}=\int_{\mathbb{B}^2}\lambda_0 u_{0}^2\big(e^{u_{0}^2}-1\big) dV_{\mathbb{B}^2}.
	\end{equation}
	Indeed, we can write
	\begin{align*}\label{eq68}\nonumber
		{}&\int_{\mathbb{B}^2}\lambda_k u_{k}^2\big(e^{u_{k}^2}-1\big) dV_{\mathbb{B}^2}\\
		={}&\int_{B_{\mathbb{B}^2}(0,R)}\lambda_k u_{k}\big(e^{u_{k}^2}-1\big) dV_{\mathbb{B}^2}+\int_{\mathbb{B}^2\backslash B_{\mathbb{B}^2}(0,R)}\lambda_k u_{k}^2\big(e^{u_{k}^2}-1\big) dV_{\mathbb{B}^2}\\
		:={}&I_1+I_2.
	\end{align*}
	For $I_1$, since $u_{k}\in L^{\infty}(\mathbb{B}^2)$, it follows from Lebesgue dominated convergence theorem that
	\begin{equation}\label{eq69}
		\lim_{R\to\infty}\lim_{k\to\infty}I_1=\int_{\mathbb{B}^2}\lambda_0 u_0^2\big(e^{u_0^2}-1\big)dV_{\mathbb{B}^2}.
	\end{equation}
	For $I_2$, since $\|\nabla_{\mathbb{B}^2}u_{k}\|_{L^2(\mathbb{B}^2)}^2\leq M_0$, we have by Poincar\'{e} inequality \eqref{Poincare} that, for any $r\in(0,1)$,
	\begin{align*}\label{eq70}\nonumber
		u_{k}^2(r)\frac{4\pi r^2}{1-r^2}={}&u_{k}^2(r)\int_{B_r}\left(\frac{2}{1-|x|^2}\right)^2 dx\\
		\leq{}&\int_{B_r}u_{k}^2\left(\frac{2}{1-|x|^2}\right)^2 dx\leq\int_{\mathbb{B}^2} u_{k}^2dV_{\mathbb{B}^2}\leq4M_0,
	\end{align*}
	which implies that
	\begin{equation}\label{eq71}
		u_{k}(r)\leq\frac{\sqrt{M_0(1-r^2)}}{\sqrt{\pi}r}.
	\end{equation}
	By using $e^{x^2}-1\leq x^2e^{x^2}$ for any $x\in\mathbb{R}$ and $u_{k}\in L^{\infty}(\mathbb{B}^2)$, we deduce
	\begin{equation}\label{eq72}\nonumber
		\begin{aligned}
			\lambda_k u_{k}^2\big(e^{u_{k}^2}-1\big)\leq \lambda_k u_{k}^4e^{u_{k}^2}\lesssim\frac{M_0^2(1-r^2)^2}{\pi^2 r^4}\in L^1(\mathbb{B}^2\backslash B_{\mathbb{B}^2}(0,R)).
		\end{aligned}
	\end{equation}
	Then by Lebesgue dominated convergence theorem we get $\lim\limits_{R\to\infty}\lim\limits_{k\to\infty}I_2=0$, which together with \eqref{eq69} yields \eqref{eq67}. Then the proof of Case 2 is completed.
	\vskip 0.1cm
	
	\textbf{Case 3:} $\lambda_0=\frac{1}{4}$. As what we did in the previous proof of (ii), we can similarly derive that up to a subsequence, $\lambda_k\to\frac{1}{4}$, $u_{k}\in L^{\infty}(\mathbb{B}^2)$ and $u_{k}\to u_0$ in $C^2(\mathbb{B}^2)$ as $k\to\infty$, where $u_0$ satisfies \eqref{eq3} with $\lambda_0$ replaced by $\frac{1}{4}$.
	It follows directly from Proposition \ref{pro-b5} of Appendix B that $u_{0}=0.$
\end{proof}

\section{Existence of positive critical points}\label{s5}
In this section, we will prove the existence of positive critical points for the Moser-Trudinger functional $$F(u)=\int_{\mathbb{B}^2}\big(e^{u^2}-1\big)dV_{\mathbb{B}^2}$$ under the constraint $\int_{\mathbb{B}^2}|\nabla_{\mathbb{B}^2}u|^2dV_{\mathbb{B}^2}=\Lambda$ for $\Lambda>4\pi$.

The ground state solution of equation \begin{equation}\label{ground}
	\left\{
	\begin{aligned}
		&-\Delta_{\mathbb{B}^2}u=\lambda u e^{u^2},\ &x\in\mathbb{B}^2,\\
		&u>0,\ &x\in\mathbb{B}^2,\\
		&u\to0,\ \text{when}\ \rho(x)\to\infty,\\
	\end{aligned}
	\right.
\end{equation}
is a positive critical point of the Moser-Trudinger functional $$F(u)=\int_{\mathbb{B}^2}\big(e^{u^2}-1\big)dV_{\mathbb{B}^2}$$ under the constraint $\int_{\mathbb{B}^2}|\nabla_{\mathbb{B}^2}u|^2dV_{\mathbb{B}^2}=\Lambda$ for some $\Lambda>0$. The existence of ground state solution is easily proved by same argument as \cite[Appendix]{CLXZ}. Furthermore, the functional energy $$I(u_\lambda)=\frac{1}{2}\int_{\mathbb{B}^2}|\nabla_{\mathbb{B}^2}u_\lambda|^2dV_{\mathbb{B}^2}
-\frac{1}{2}\int_{\mathbb{B}^2}\big(e^{u_\lambda^2}-1\big)dV_{\mathbb{B}^2}$$
of the ground state solution $u_\lambda$ to equation \eqref{ground} must satisfy $0<I_{\lambda}(u_\lambda)<2\pi$.
It is not difficult to prove that $\int_{\mathbb{B}^2}|\nabla_{\mathbb{H}}u_\lambda|^2dV_{\mathbb{B}^2}$ is bounded when $\lambda$ is away from $\frac{1}{4}$ through the definition of the ground state solution. Recalling Theorem \ref{thm1}, we know that $u_\lambda$ must blow up at the origin when $\lambda\rightarrow 0$ up to some M\"{o}bius transformation, Dirichlet energy $\int_{\mathbb{B}^2}|\nabla_{\mathbb{B}^2}u_\lambda|^2dV_{\mathbb{B}^2}$ is continuous about the parameter $\lambda\in (0, \frac{1}{4})$ and $\lim\limits_{\lambda\rightarrow 0}\int_{\mathbb{B}^2}|\nabla_{\mathbb{B}^2}u_\lambda|^2dV_{\mathbb{B}^2}=4\pi$. If we can prove when $\lambda$ approaches to zero, $\int_{\mathbb{B}^2}|\nabla_{\mathbb{B}^2}u_\lambda|^2dV_{\mathbb{B}^2}$ approaches to $4\pi$ from the above, then one can deduce that the Moser-Trudinger functional $$F(u)=\int_{\mathbb{B}^2}\big(e^{u^2}-1\big)dV_{\mathbb{B}^2}$$ under the constraint $\int_{\mathbb{B}^2}|\nabla_{\mathbb{B}^2}u|^2dV_{\mathbb{B}^2}=\Lambda$ for $\Lambda>4\pi$ and $\Lambda$ close to $4\pi$ has a positive critical point, which accomplishes the proof of Theorem \ref{thm2}. Now, it remains to prove that $\int_{\mathbb{B}^2}|\nabla_{\mathbb{B}^2}u_\lambda|^2dV_{\mathbb{B}^2}$ approaches to $4\pi$ from the above when $\lambda\rightarrow 0$. We need the following lemmas.

\begin{lemma}[Proposition 15, \cite{MM-CV}]\label{lem4}
	Let $f\in C^0(\mathbb{R}^2)$ be radially symmetric and satisfy $f(r)=O( \ln^q r)$ as $r\to\infty$ for some $q\geq0$. If $w\in C^2(\mathbb{R}^2)$ is a radially symmetric solution of
	\begin{equation}\label{1}
		-\Delta w=4e^{2\eta_0}(f+2w),
	\end{equation}
	where $\eta_0$ is as in \eqref{eq5}, then $\Delta w\in L^1(\mathbb{R}^2)$ and we have
	\begin{equation}\label{2}
		w(r)=\beta \ln r+O(1),\quad w'(r)=\frac{\beta}{r}+O\left(\frac{ \ln^{\bar{q}} r}{r^3}\right)\quad\text{as}\ r\to\infty,
	\end{equation}
	where $\bar{q}=\max\{1,q\}$ and
	\begin{equation}\label{3}
		\beta:=\frac{1}{2\pi}\int_{\mathbb{R}^2}\Delta w dx.
	\end{equation}
\end{lemma}

\begin{lemma}[Proposition 16, \cite{MM-CV}]\label{lem5}
	Let $f,w$ and $\beta$ be as in Lemma \ref{lem4}. Then
	\begin{equation}\label{4}
		\beta=-\frac{2}{\pi}\int_{\mathbb{R}^2}\frac{|x|^2-1}{(1+|x|^2)^3}f(x) dx.
	\end{equation}
\end{lemma}

\begin{lemma}\label{lem6}
	Let $\eta_{\lambda}:=c_{\lambda}(u_{\lambda}(r_{\lambda}x)-c_{\lambda})$,  $w_{\lambda}:=c_{\lambda}^2(\eta_{\lambda}-\eta_0)$ and $z_{\lambda}:=c_{\lambda}^2(w_{\lambda}-w_0)$, where $c_\lambda=u_\lambda(0)=\max\limits_{x\in \mathbb{B}^2}u_\lambda(x)$, $\lambda r_{\lambda}^2c_{\lambda}^2e^{c_{\lambda}^2}=1$,  and $\eta_0,w_0$ are defined in Lemmas \ref{lem1} and \ref{lem2} respectively. If $\lambda\to0$, then we have $z_{\lambda}\to z_0$ in $C_{\rm{loc}}^1(\mathbb{R}^2)$, where $z_0$ is the unique solution to the ODE
	\begin{equation}\label{5}
		-\Delta z_0=4e^{2\eta_0}\left(w_0+2w_0^2+4\eta_0w_0+2w_0\eta_0^2+\eta_0^3+\frac{1}{2}\eta_0^4+2z_0\right)
	\end{equation}
	with $z_0(0)=0,\  z_0'(0)=0.$ Moreover, $z_0$ satisfies
	\begin{equation}\label{6}
		z_0(r)=\beta \ln r+O(1)
	\end{equation}
	as $r\to\infty,$ where
	\begin{equation}\label{7}
		\beta=\frac{1}{2\pi}\int_{\mathbb{R}^2}\Delta z_0 dx=-6-\frac{\pi^2}{3}.
	\end{equation}
\end{lemma}

\begin{proof}
	Direct computations give that
	\begin{equation}\label{8}
		\begin{aligned}
			{}&-\Delta z_{\lambda}=4e^{2\eta_0}\Bigg[c_{\lambda}^4\Big(1+
			\frac{\eta_0}{c_{\lambda}^2}+\frac{w_0}{c_{\lambda}^4}
			+\frac{w_{\lambda}-w_0}{c_{\lambda}^4}\Big)\Big(\frac{1}{1-r_{\lambda}^2|x|^2}\Big)^2\\
			{}&\quad\quad\quad\quad\,\,\cdot\exp\Bigg\{\frac{2(w_{\lambda}-w_0)+2w_0+\eta_0^2}{c_{\lambda}^2}
			+\frac{2\eta_0(w_{\lambda}-w_0)+2\eta_0w_0}{c_{\lambda}^4}\\
				{}&\quad\quad\quad\quad\,\,+\frac{(w_{\lambda}-w_0)^2+2w_0(w_{\lambda}-w_0)+w_0^2}{c_{\lambda}^6}\Bigg\}
			-c_{\lambda}^4-c_{\lambda}^2(\eta_0+\eta_0^2+2w_0)\Bigg].
		\end{aligned}
	\end{equation}
	From Theorem \ref{thm1} we know $c_{\lambda}\to\infty$ as $\lambda\to0$. Then it follows from Lemma \ref{lem2} that for every $R>0$, $w_{\lambda}(r)-w_0(r)=o(1)$ as $\lambda\to0$ uniformly for $r\in[0,R]$. By using a Taylor expansion:
	\begin{equation}\label{9}\nonumber
		\begin{aligned}
			{}&\exp\Bigg\{\frac{2(w_{\lambda}-w_0)+2w_0+\eta_0^2}{c_{\lambda}^2}+\frac{2\eta_0(w_{\lambda}-w_0)+2\eta_0w_0}{c_{\lambda}^4}\\
			{}&\quad\quad\quad\quad\quad\quad\quad\quad\quad\quad\quad\quad+\frac{(w_{\lambda}-w_0)^2+2w_0(w_{\lambda}-w_0)+w_0^2}{c_{\lambda}^6}\Bigg\}\\
			={}&1+\frac{2(w_{\lambda}-w_0)+2w_0+\eta_0^2}{c_{\lambda}^2}+\frac{2\eta_0w_0}{c_{\lambda}^4}+\frac{(2w_0+\eta_0^2)^2}{2c_{\lambda}^4}+o(1)c_{\lambda}^{-4},
		\end{aligned}
	\end{equation}
	with $o(1)\to0$ as $\lambda\to0$ uniformly for $r\in[0,R]$, we get
	\begin{equation}\label{10}
		-\Delta z_{\lambda}=4e^{2\eta_0}
		\Big(w_0+2w_0^2+4\eta_0w_0+2w_0\eta_0^2+\eta_0^3
		+\frac{1}{2}\eta_0^4+2z_{\lambda}+o(1)+o(1)z_{\lambda}\Big)
	\end{equation}
	with $o(1)\to0$ as $\lambda\to0$ uniformly for $r\in[0,R]$. Observing that $z_{\lambda}(0)=z_{\lambda}'(0)=0$, from ODE theory it follows that $z_{\lambda}(r)$ is uniformly bounded in $r\in[0,R]$ and by elliptic estimates $z_{\lambda}\to z_0$ in $C_{\rm{loc}}^1(\mathbb{R}^2)$, where $z_0$ is the unique solution to \eqref{5}.
	
	Moreover, it follows from Lemmas \ref{lem1} and \ref{lem2} that
	\begin{equation}\label{11}
		f:=w_0+2w_0^2+4\eta_0w_0+2w_0\eta_0^2+\eta_0^3+\frac{1}{2}\eta_0^4=O( \ln^4|x|).
	\end{equation}
	Using Lemmas \ref{lem4} and \ref{lem5}, we know $z_0$ satisfies \eqref{6} where
	\begin{equation}\label{12}
		\beta=-\frac{2}{\pi}\int_{\mathbb{R}^2}\frac{|x|^2-1}{(1+|x|^2)^3}f(x) dx.
	\end{equation}
	By integrating by parts we compute
		\begin{align*}
			&\int_{\mathbb{R}^2}\frac{|x|^2-1}{(1+|x|^2)^3}w_0(x) dx=\frac{\pi^3}{18}-\frac{7\pi}{12},\\
			&\int_{\mathbb{R}^2}\frac{|x|^2-1}{(1+|x|^2)^3}w_0^2(x) dx=\left(\frac{625}{216}-\frac{4}{9}Z(3)\right)\pi-\frac{\pi^3}{81}-\frac{\pi^5}{45},\\
			&\int_{\mathbb{R}^2}\frac{|x|^2-1}{(1+|x|^2)^3}\eta_0(x)w_0(x) dx=\left(\frac{125}{72}-\frac{2}{3}Z(3)\right)\pi-\frac{2\pi^3}{27},\\
			&\int_{\mathbb{R}^2}\frac{|x|^2-1}{(1+|x|^2)^3}w_0(x)\eta_0^2(x) dx=\left(\frac{16}{9}Z(3)-\frac{409}{54}\right)\pi+\frac{35\pi^3}{162}+\frac{\pi^5}{45},\\
			&\int_{\mathbb{R}^2}\frac{|x|^2-1}{(1+|x|^2)^3}\eta_0^3(x) dx=-\frac{21\pi}{4},\\
			&\int_{\mathbb{R}^2}\frac{|x|^2-1}{(1+|x|^2)^3}\eta_0^4(x) dx=\frac{45\pi}{2},
		\end{align*}
	where $Z$ denotes the Euler-Riemann zeta function. Combining \eqref{11}, \eqref{12} and above equalities yields \eqref{7} finally.
\end{proof}

\begin{remark}\label{rema}
	We can still obtain the next order expansion by the similar way if we define $\zeta_\lambda=c_\lambda^2(z_\lambda-z_0)$. Indeed, if $\lambda\to0$,  we will have $\zeta_\lambda\to\zeta_0$ in $C^1_{\rm{loc}}(\mathbb{R}^2)$ where $\zeta_0$ is the unique solution to the ODE
	\begin{equation*}
		-\Delta\zeta_0=4e^{2\eta_0}
		\Big(
		z_0+3w_0^2+4\eta_0z_0+4w_0z_0+3w_0\eta_0^2+6\eta_0w_0^2+2\eta_0^2z_0+4w_0\eta_0^3+\frac{1}{2}\eta_0^5+2\zeta_0
		\Big)
	\end{equation*}
	with $\zeta_0(0)=0,\ \zeta_0'(0)=0.$ Moreover, $\zeta_0$ satisfies $\zeta_0=\gamma \ln r+O(1)$ as $r\to\infty$ with some $\gamma\in\mathbb{R}.$
\end{remark}

We are in a position to give
\begin{lemma}\label{lem7}
	If $\lambda\rightarrow 0$, then it holds that
	\begin{equation}\label{14}
		\|\nabla_{\mathbb{B}^2}u_{\lambda}\|_{L^2(\mathbb{B}^2)}^2\geq4\pi+\frac{4\pi}{c_{\lambda}^4}+o(c_{\lambda}^{-4}).
	\end{equation}
\end{lemma}

\begin{proof}
	From Theorem \ref{thm1} we know $c_{\lambda}\to\infty$ as $\lambda\to0$.
	Fix $S=s_{\lambda}=o(1)e^{c_{\lambda}}$ with $s_{\lambda}\geq c_{\lambda}^p$ for some $p>2$. By a change of variable, we have
	\begin{equation}\label{15}\nonumber
		\begin{aligned}
			{}&\int_{B_{r_{\lambda}s_{\lambda}}}\lambda u_{\lambda}^2e^{u_{\lambda}^2}\left(\frac{2}{1-|x|^2}\right)^2 dx=\int_{B_{s_{\lambda}}}\left(1+\frac{\eta_{\lambda}}{c_{\lambda}^2}\right)e^{\left(2+\frac{\eta_{\lambda}}{c_{\lambda}^2}\right)\eta_{\lambda}}\left(\frac{2}{1-r_{\lambda}^2|x|^2}\right)^2 dx\\
			={}&\int_{B_{s_{\lambda}}}\left(1+\frac{\eta_{\lambda}}{c_{\lambda}^2}\right)\left(-\Delta\eta_0-\frac{\Delta w_0}{c_{\lambda}^2}-\frac{\Delta z_0}{c_{\lambda}^4}+\frac{\Psi_{\lambda}(r,\psi_{\lambda})}{c_{\lambda}^6}\right) dx,
		\end{aligned}
	\end{equation}
	where $\eta_{\lambda}:=\eta_0+\frac{w_0}{c_{\lambda}^2}+\frac{z_0}{c_{\lambda}^4}+\frac{\psi_{\lambda}}{c_{\lambda}^6}$ and
	\begin{equation}\label{16}\nonumber
		\Psi_{\lambda}(r,\psi_{\lambda}):=c_{\lambda}^6\left[\left(1+\frac{\eta_{\lambda}}{c_{\lambda}^2}\right)e^{\left(2+\frac{\eta_{\lambda}}{c_{\lambda}^2}\right)\eta_{\lambda}}\left(\frac{2}{1-r_{\lambda}^2r^2}\right)^2+\Delta\eta_0+\frac{\Delta w_0}{c_{\lambda}^2}++\frac{\Delta z_0}{c_{\lambda}^4}\right].
	\end{equation}
	By Lemmas \ref{lem1}, \ref{lem2} and \ref{lem6} we know
	\begin{equation}\label{xi}
		|\eta_0|+|w_0|+|z_0|=O(\xi),
	\end{equation}
	where $\xi$ is defined in the proof of Lemma \ref{lem3}. Thus, with a Taylor expansion we deduce that
	\begin{equation}\label{17}\nonumber
		\Psi_{\lambda}(r,\psi_{\lambda})=4e^{2\eta_0}\left(2\psi_{\lambda}+O(\psi_{\lambda})\left(O(c_{\lambda}^{-6}\psi_{\lambda})+O(c_{\lambda}^{-2}\xi^2)\right)+O(\xi^6)+O(c_{\lambda}^6r_{\lambda}^2r^2)\right)
	\end{equation}
	provided that
	\begin{equation}\label{18}\nonumber
		|c_{\lambda}^{-6}\psi_{\lambda}|\leq\delta,\quad c_{\lambda}^{-1}\xi\leq\delta,\quad r_{\lambda}r\leq\delta,
	\end{equation}
	for some $\delta>0$ small enough. Similar to the proof of Lemma \ref{lem3}, from the contraction mapping theorem we infer that, there exist large constants $\tilde{T}>0$ and $\tilde{M}>0$ such that
	\begin{equation}\label{19}
		\begin{aligned}
			&|\psi_{\lambda}(r)|\leq C(\tilde{T}),\quad\text{for}\ r\in[0,\tilde{T}],\\
			&|\psi_{\lambda}(r)|\leq C(\tilde{T})+\tilde{M}( \ln r- \ln \tilde{T}),\quad\text{for}\ r\in[\tilde{T},S].
		\end{aligned}
	\end{equation}
	Therefore we get $\Psi_{\lambda}(r,\psi_{\lambda})=e^{2\eta_0}O(\xi^6)$, which implies that
	\begin{equation}\label{20}\nonumber
		\int_{B_{s_{\lambda}}}|\Psi_{\lambda}(r,\psi_{\lambda})| dx\leq C\int_{\mathbb{R}^2}\frac{\xi^6(|x|)}{(1+|x|^2)^2} dx\leq C.
	\end{equation}
	Similarly, we have
	\begin{equation}\label{21}\nonumber
		\max\{|\Delta\eta_0|,|\Delta w_0|,|\Delta z_0|\}=e^{2\eta_0}O(\xi^4),
	\end{equation}
	which leads to
	\begin{equation}\label{22}\nonumber
		\int_{B_{s_{k}}}\xi\max\{|\Delta\eta_0|,|\Delta w_0|,|\Delta z_0|\} dx\leq C\int_{\mathbb{R}^2}\frac{\xi^5(|x|)}{(1+|x|^2)^2} dx\leq C.
	\end{equation}
	Moreover, from \eqref{xi} and \eqref{19} we see
	\begin{equation}\label{eta}\nonumber
		1+\frac{\eta_{\lambda}}{c_{\lambda}^2}=1+\frac{\eta_0}{c_{\lambda}^2}+\frac{w_0}{c_{\lambda}^4}+o(c_{\lambda}^{-5}).
	\end{equation}
	Summing up we get
	\begin{align*}\label{23}\nonumber
		{}&\int_{B_{r_{\lambda}s_{\lambda}}}\lambda u_{\lambda}^2e^{u_{\lambda}^2}\left(\frac{2}{1-|x|^2}\right)^2dx\\ ={}&\int_{B_{s_{\lambda}}}\left(-\Delta\eta_0-\frac{\eta_0\Delta\eta_0+\Delta w_0}{c_{\lambda}^2}-\frac{w_0\Delta\eta_0+\eta_0\Delta w_0+\Delta z_0}{c_{\lambda}^4}\right) dx+o(c_{\lambda}^{-5}).
	\end{align*}
	Computing directly yields that
	\begin{equation}\label{24}\nonumber
		\int_{B_{s_{\lambda}}}(-\Delta\eta_0) dx=\int_{B_{s_{\lambda}}}4e^{2\eta_0} dx=4\pi\left(1-\frac{1}{1+s_{\lambda}^2}\right)=4\pi+o(c_{\lambda}^{-4}),
	\end{equation}
	\begin{equation}\label{25}\nonumber
		\begin{aligned}
			\int_{B_{s_{\lambda}}}(-\eta_0\Delta\eta_0-\Delta w_0) dx={}&\int_{B_{s_{\lambda}}}4e^{2\eta_0}\eta_0 dx-2\pi s_{\lambda}w_0'(s_{\lambda})\\
			={}&4\pi\left(\frac{ \ln(1+s_{\lambda}^2)}{1+s_{\lambda}^2}+\frac{1}{1+s_{\lambda}^2}-1\right)+4\pi+O(s_{\lambda}^{-2} \ln^2s_{\lambda})\\
			={}&o(c_{\lambda}^{-2}),
		\end{aligned}
	\end{equation}
	and
	\begin{equation}\label{26}\nonumber
		\begin{aligned}
			{}&\int_{B_{s_{\lambda}}}(-w_0\Delta\eta_0-\eta_0\Delta w_0-\Delta z_0) dx\\={}&\int_{B_{s_{\lambda}}}4e^{2\eta_0}(w_0+\eta_0^2+\eta_0^3+2w_0\eta_0) dx+\int_{B_{s_{\lambda}}}(-\Delta z_0) dx\\
			={}&-8\pi-\frac{2\pi^3}{3}+o(1)+2\pi\left(6+\frac{\pi^2}{3}\right)+o(1)
			=4\pi+o(1).
		\end{aligned}
	\end{equation}
	%where we used
	%\begin{equation}\label{d}\nonumber
	%	w_0'(r)=-\frac{2}{r}+O(\frac{ \ln^2r}{r^3}).
	%\end{equation}
	As a result,
	\begin{equation}\label{27}\nonumber
		\int_{B_{r_{\lambda}s_{\lambda}}}\lambda u_{\lambda}^2e^{u_{\lambda}^2}\left(\frac{2}{1-|x|^2}\right)^2 dx=4\pi+\frac{4\pi}{c_{\lambda}^4}+o(c_{\lambda}^{-4}).
	\end{equation}
	Noticing that $r_{\lambda}s_{\lambda}<1$, so we get
	\begin{equation}\label{28}\nonumber
		\|\nabla_{\mathbb{B}^2}u_{\lambda}\|_{L^2(\mathbb{B}^2)}^2=\int_{B_1}\lambda u_{\lambda}^2e^{u_{\lambda}^2}\left(\frac{2}{1-|x|^2}\right)^2 dx\geq4\pi+\frac{4\pi}{c_{\lambda}^4}+o(c_{\lambda}^{-4}).
	\end{equation}
	This ends the proof of the lemma.
\end{proof}

The following endpoint Poincar\'e--Sobolev inequality follows directly from Lemma 2.1(b) of Mancini, Sandeep and Tintarev \cite{MST-2013} by taking $N=2$ and $p=4$:

\begin{lemma}\label{9-7-prop2.1}
There exists $S_4>0$ such that
\begin{equation}\label{9-7-5}
S_4\Big(\int_{\mathbb{B}^2}|u|^4\,dV_{\mathbb{B}^2}\Big)^{1/2}
\leq Q_{1/4}(u)\,\,\,\text{for~every}\,\, u\in W^{1,2}(\mathbb{B}^2),
\end{equation}
where
\[
Q_\lambda(u):=
\int_{\mathbb B^2}|\nabla_{\mathbb B^2}u|^2
\,dV_{\mathbb B^2}
-\lambda\int_{\mathbb B^2}u^2
\,dV_{\mathbb B^2}.
\]
\end{lemma}

\begin{proposition}\label{9-7-thm1.1}
Let $u_\lambda\in W^{1,2}(\mathbb{B}^2)$ be a nonzero least-energy solution of \eqref{uniq}. Then
\begin{equation}\label{9-7-2}
\lim_{\substack{\lambda\to1/4}}
\int_{\mathbb{B}^2}|\nabla_{\mathbb{B}^2}u_\lambda|^2\,dV_{\mathbb{B}^2}=+\infty.
\end{equation}
\end{proposition}

\begin{proof}
Suppose, by contradiction, that the conclusion fails.
 Notice that the Dirichlet energies of the least-energy solutions remain uniformly bounded whenever
  $\lambda$ stays away from $1/4.$ Hence this uniform bound extends to all
  $\lambda\in(0,1/4).$ Thus, passing to a subsequence if necessary, we may assume that
\[
\int_{\mathbb B^2}
|\nabla_{\mathbb B^2}u_k|^2\,dV_{\mathbb B^2}
\le M_0
\]
for some constant $M_0>0$ independent of $k.$

Denote
$
a_k:=\|u_k\|_{L^\infty(\mathbb B^2)}.
$
It follows form Theorem \ref{thm1}(iii) that $a_k \longrightarrow0.$

The spectral Poincare inequality yields
\begin{equation}\label{eq:L2-bound}
\int_{\mathbb B^2}u_k^2\,dV_{\mathbb B^2}
\le4M_0.
\end{equation}

Testing the equation with $u_k$, we obtain
\[
Q_{\lambda_k}(u_k)
=\lambda_k\int_{\mathbb B^2}
u_k^2(e^{u_k^2}-1)\,dV_{\mathbb B^2}.
\]
Since $e^s-1\le se^s$ for $s\ge0$, we have
\[
Q_{\lambda_k}(u_k)
\le\lambda_ke^{a_k^2}
\int_{\mathbb B^2}u_k^4\,dV_{\mathbb B^2}.
\]

On the other hand, since $\lambda_k<1/4$,
Lemma \ref{9-7-prop2.1} gives
\[
S_4\Big(
\int_{\mathbb B^2}u_k^4\,dV_{\mathbb B^2}
\Big)^{1/2}
\le Q_{1/4}(u_k)
\le Q_{\lambda_k}(u_k).
\]
Consequently,
\begin{equation}\label{eq-10-13}
S_4\Big(
\int_{\mathbb B^2}u_k^4\,dV_{\mathbb B^2}
\Big)^{1/2}
\le
\lambda_ke^{a_k^2}
\int_{\mathbb B^2}u_k^4\,dV_{\mathbb B^2}.
\end{equation}

Since $u_k \not\equiv0,$ we have
$$
\int_{\mathbb B^2}u_k^4 \, dV_{\mathbb B^2}>0.
$$
Thus, dividing  \eqref{eq-10-13} by
$\left(\int_{\mathbb B^2}u_k^4,dV_{\mathbb B^2}\right)^{1/2}$, by \eqref{eq:L2-bound} we obtain
\begin{align*}
S_4
\le\lambda_ke^{a_k^2}
\Big(\int_{\mathbb B^2}u_k^4
\,dV_{\mathbb B^2}\Big)^{1/2}
\le\lambda_ke^{a_k^2}a_k
\Big(\int_{\mathbb B^2}u_k^2
\,dV_{\mathbb B^2}\Big)^{1/2}
\le2\lambda_k\sqrt{M_0}\,
a_ke^{a_k^2}
\longrightarrow0,
\end{align*}
contradicting $S_4>0$.
The proof is complete.
\end{proof}

At the end of this section,  we are ready to prove Theorem \ref{thm2}.
\begin{proof}[\textbf{Proof of Theorem \ref{thm2}}]
	Observe that solutions of problem \eqref{eq1} are in fact critical points of the functional $F(u)$ under the constraint $\int_{\mathbb{B}^2}|\nabla_{\mathbb{B}^2}u_{\lambda}|^2 dV_{\mathbb{B}^2}=\Lambda$. Define
	\begin{equation}\label{29}
		E(\lambda):=\|\nabla_{\mathbb{B}^2}u_{\lambda}\|_{L^2(\mathbb{B}^2)}^2.
	\end{equation}
	Then $E$ is continuous with respect to $\lambda$. From Theorem \ref{thm1} we have $\lim\limits_{\lambda\to0}E(\lambda)=4\pi$. Hence, together with Lemma \ref{lem7} and Proposition \ref{9-7-thm1.1}, we complete the proof.
\end{proof}

\section{A precise characterization of $c_\lambda$}\label{cc}
In this section, we wish to provide a precise relation between $\lambda$ and $c_\lambda$ when $\lambda\to0$ by various pointwise estimates of positive solutions to \eqref{uniq}, which is important to prove the uniqueness result.

Denote by $G(x,\cdot)$ the Green's function
of $-\Delta$ in $B_{1}$, i.e. the solution to
\begin{equation}\label{greensyst}
	\begin{cases}
		-\Delta G(x,\cdot)= \delta_x  &{\text{in}~B_{1}}, \\[1mm]
		G(x,\cdot)=0  &{\text{on}~\partial B_{1}},
	\end{cases}
\end{equation}
where $\delta_x$ is the Dirac function. We have the following well known decomposition formula of $G(x,y)$,
\begin{equation}\label{GreenS-H}
	G(x,y)=S(x,y)-H(x,y)  ~\mbox{for}~(x,y)\in B_{1}\times B_{1},
\end{equation}
where  $S(x,y):=-\frac{1}{2\pi} \ln |x-y|$ and $H(x,y)$ is the regular part of $G(x,y)$. Furthermore, for any $x\in B_{1}$, the Robin function is defined as
\begin{equation}\label{Robinf}
	\mathcal{R}(x):=H(x,x).
\end{equation}

The preliminary relation between $\lambda$ and $c_\lambda$ when $\lambda\to0$ is as follows.
\begin{proposition}\label{lem-7-31-2}
	There holds that $\lim\limits_{\lambda\to0}\lambda c^{2}_{\lambda}=1$.
	%	\begin{equation}\label{31-2}
		%		\lim_{\lambda\rightarrow 0} (\lambda c^{2}_{\lambda})=1.
		%	\end{equation}
\end{proposition}

To prove Proposition \ref{lem-7-31-2}, we need the following lemmas.
\begin{lemma}\label{lem2-1}
	For any $0<\sigma<1$, defining $u_{\lambda}^{\sigma}=\min\{u_{\lambda}, \sigma c_{\lambda}\}$, then
	$$
	\lim\limits_{\lambda\rightarrow 0}\int_{\mathbb{B}^2}
	|\nabla_{\mathbb{B}^2}u_{\lambda}^{\sigma}|^2dV_{\mathbb{B}^2}
	=4\pi\sigma.
	$$
\end{lemma}
\begin{proof}
	Since $u_\lambda$ satisfies the equation
	\begin{equation}\label{eq1-add}
		-\Delta_{\mathbb{B}^2}u_\lambda=\lambda u_\lambda e^{u_\lambda^2},\ x\in\mathbb{B}^2,\\
	\end{equation}
	multiplying both sides of equation \eqref{eq1-add} by $u_\lambda^\sigma$ and using the integration by parts, we derive that
	\begin{equation}\begin{split}
			\lim\limits_{\lambda\rightarrow 0}\int_{\mathbb{B}^2}|\nabla_{\mathbb{B}^2} u_\lambda^\sigma|^2dV_{\mathbb{B}^2}&=\lim\limits_{\lambda\rightarrow 0}\lambda\int_{\mathbb{B}^2}e^{u_\lambda^2}u_\lambda u_\lambda^\sigma dV_{\mathbb{B}^2}\\
			&\geq \lim\limits_{L\rightarrow +\infty}\lim\limits_{\lambda\rightarrow 0} \lambda\int_{B_{Lr_\lambda}}e^{u_\lambda^2}u_\lambda u_\lambda^\sigma \left(\frac{2}{1-|x|^2}\right)^2dx\\
			&= \lim\limits_{L\rightarrow +\infty}\lim\limits_{\lambda\rightarrow 0} 4c_\lambda r_{\lambda}^2\lambda \sigma c_\lambda\int_{B_{L}}e^{u_\lambda^2(r_\lambda x)}dx\\
			&=4\pi \sigma .
	\end{split}\end{equation}
	On the other hand, multiplying both sides of equation \eqref{eq1-add} by $\big(u_\lambda-\sigma c_\lambda\big)^{+}$, we also derive that
	\begin{equation*}\begin{split}
			\lim\limits_{\lambda\rightarrow 0}\int_{\mathbb{B}^2}|\nabla_{\mathbb{B}^2}\big(u_\lambda-\sigma c_\lambda\big)^{+}|^2dV_{\mathbb{B}^2}&=\lim\limits_{\lambda\rightarrow 0}\lambda\int_{\mathbb{B}^2}e^{u_\lambda^2}u_\lambda \big(u_\lambda-\sigma c_{\lambda}\big)^{+}dV_{\mathbb{B}^2}\\
			&\geq  \lim\limits_{L\rightarrow +\infty}\lim\limits_{\lambda\rightarrow 0} 4c_\lambda r_{\lambda}^2\lambda (1-\sigma) c_\lambda \int_{B_{L}}e^{u_\lambda^2(r_\lambda x)}dx\\
			&=4\pi(1-\sigma).
	\end{split}\end{equation*}	
	Hence we have
	\begin{equation}\begin{split}
			4\pi \sigma + 4\pi(1-\sigma)&\leq \lim\limits_{\lambda\rightarrow 0}\int_{\mathbb{B}^2}\Big(|\nabla_{\mathbb{B}^2} u_\lambda^\sigma|^2+|\nabla_{\mathbb{B}^2}\big(u_\lambda-\sigma c_\lambda\big)^{+}|^2\Big)dV_{\mathbb{B}^2}\\
			&=\lim\limits_{\lambda\rightarrow 0}\int_{\mathbb{B}^2}|\nabla_{\mathbb{B}^2}u_\lambda|^2dV_{\mathbb{B}^2}=4\pi.
	\end{split}\end{equation}
	This implies that $$\lim\limits_{\lambda\rightarrow 0}\int_{\mathbb{B}^2}|\nabla_{\mathbb{B}^2}u_{\lambda}^{\sigma}|^2dV_{\mathbb{B}^2}=4\pi\sigma.$$ Then we accomplish the proof of Lemma \ref{lem2-1}.
\end{proof}

\begin{lemma}\label{lem2-2}
	For any $\varphi\in C_{c}^{\infty}(\mathbb{B}^2)$, there holds $$\lim\limits_{\lambda\rightarrow 0}\lambda\int_{\mathbb{B}^2}c_{\lambda}u_\lambda e^{u_\lambda^2}\varphi dV_{\mathbb{B}^2}=4\pi\varphi(0).$$
\end{lemma}
\begin{proof}
	We divide $\mathbb{B}^2$ into the following three parts
	$$\{u_\lambda\geq \sigma c_\lambda\}\backslash B_{Lr_\lambda},\ \{u_\lambda\leq \sigma c_\lambda\}, B_{Lr_\lambda}.$$
	Denote the integrals on the above three domains by $I_1$, $I_2$ and $I_3$, respectively.
	For $I_1$, we have
	\begin{equation*}\begin{split}
			I_1&=\lim\limits_{L\rightarrow +\infty}\lim\limits_{\lambda\rightarrow 0}\lambda\int_{\{u_\lambda\geq \sigma c_\lambda\} \backslash B_{Lr_\lambda}}c_{\lambda}u_\lambda e^{u_\lambda^2}\varphi dV_{\mathbb{B}^2}\\
			&\leq \lim\limits_{L\rightarrow +\infty}\lim\limits_{\lambda\rightarrow 0} \frac{1}{\sigma}\|\varphi\|_{L^{\infty}(\mathbb{B}^2)}\left(\int_{\mathbb{B}^2}|\nabla_{\mathbb{B}^2}u_\lambda|^2dV_{\mathbb{B}^2}-\int_{B_{Lr_\lambda}}\lambda u_\lambda^2e^{u_\lambda^2}dV_{\mathbb{B}^2}\right)\\
			&\leq \lim\limits_{L\rightarrow +\infty} \frac{1}{\sigma}\|\varphi\|_{L^{\infty}(\mathbb{B}^2)}
			\Big(4\pi-\int_{B_{L}}4e^{2\eta_0}dx\Big)=0.\\
	\end{split}\end{equation*}
	For $I_2$, since $\lim\limits_{\lambda\rightarrow 0} \lambda c_{\lambda}=0$ because of the boundedness of $\lambda c_\lambda^2$ from the proof of Proposition 4.1, we obtain
	\begin{equation*}\begin{split}
			I_2&=\lim\limits_{L\rightarrow +\infty}\lim\limits_{\lambda\rightarrow 0}\lambda\int_{\{u_\lambda\leq \sigma c_\lambda\}}c_{\lambda}u_\lambda e^{u_\lambda^2}\varphi dV_{\mathbb{B}^2}\\
			&\leq \lim\limits_{L\rightarrow +\infty}\lim\limits_{\lambda\rightarrow 0}\lambda c_{\lambda}\|\varphi\|_{L^{\infty}(\mathbb{B}^2)}\int_{\mathbb{B}^2}u_\lambda^\sigma e^{|u_\lambda^\sigma|^2}dV_{\mathbb{B}^2}=0,\\
	\end{split}\end{equation*}
	where we use the Moser-Trudinger inequality \eqref{MT} for $\int_{\mathbb{B}^2}u_\lambda^\sigma e^{|u_\lambda^\sigma|^2}dV_{\mathbb{B}^2}$ since $\lim\limits_{\lambda\rightarrow 0}\int_{\mathbb{B}^2}|\nabla_{\mathbb{B}^2}u_{\lambda}^{\sigma}|^2dV_{\mathbb{B}^2}=4\pi\sigma$.
	For $I_{3}$, we have
	\begin{equation*}\begin{split}
			I_3&=\lim\limits_{L\rightarrow +\infty}\lim\limits_{\lambda\rightarrow 0}\lambda\int_{B_{Lr_\lambda}}c_{\lambda}u_\lambda e^{u_\lambda^2}\varphi dV_{\mathbb{B}^2}\\
			&=\lim\limits_{L\rightarrow +\infty}\lim\limits_{\lambda\rightarrow 0} \varphi(0)\int_{B_{Lr_\lambda}}\lambda u_\lambda^2e^{u_\lambda^2}dV_{\mathbb{B}^2}\\
			&=\varphi(0)\lim\limits_{L\rightarrow +\infty}\lim\limits_{\lambda\rightarrow 0}\int_{B_{L}}\Big(\frac{u_{\lambda}(x)}{c_{\lambda}}\Big)^{2}e^{(\frac{\eta_{\lambda}(x)}{c_{\lambda}^2}+2)\eta_{\lambda}(x)}\frac{4}{(1-|r_{\lambda}x|^{2})^{2}}dx\\
			&=\varphi(0)\lim\limits_{L\rightarrow +\infty}\int_{B_{L}}4e^{2\eta_0}dx=4\pi\varphi(0).\\
	\end{split}\end{equation*}
	Combining the estimates of $I_{1}$, $I_{2}$ and $I_{3}$, we derive that $$\lim\limits_{\lambda\rightarrow 0}\lambda\int_{\mathbb{B}^2}c_{\lambda}u_\lambda e^{u_\lambda^2}\varphi dV_{\mathbb{B}^2}=4\pi \varphi(0).$$
	This accomplishes the proof of Lemma \ref{lem2-2}.
\end{proof}

\begin{lemma}\label{lem-7-31-1}
	There holds $c_\lambda u_\lambda \rightharpoonup 4\pi G$ in $W^{1,q}(\mathbb{B}^2)$ for any $q< 2$, where $G$ is the Green's function of the Laplacian operator $-\Delta_{\mathbb{B}^2}$ and $G(x)=-\frac{1}{2\pi}\ln |x|$. Furthermore, $c_\lambda u_\lambda$ converges strongly to $4\pi G$ in $C^{2}_{\rm loc}(\mathbb{B}^2 \backslash \{0\})$.
\end{lemma}

\begin{proof}
	From the proof of Lemma \ref{lem2-2}, one can easily see that $f_\lambda=\lambda c_{\lambda}u_\lambda e^{u_\lambda^2}$ is bounded in $L^1(\mathbb{B}^2)$. Then one can derive that $c_\lambda u_\lambda$ is bounded in $W^{1,q}(\mathbb{B}^2)$ for any $1<q<2$ (refer to \cite{Li,Li1} for detailed proofs). Hence there exists $G\in W^{1,q}(\mathbb{B}^2)$ such that $c_\lambda u_\lambda\rightharpoonup 4\pi G$ in $W^{1,q}(\mathbb{B}^2)$ and
	$G$ satisfies equation
	\begin{equation}\begin{cases}
			&-\Delta_{\mathbb{B}^2}G=\frac{1}{4}\delta_0,\ \ x\in \mathbb{B}^2,\\
			&G(x)\to0,\ \text{when}\ \rho(x)\to\infty.
	\end{cases}\end{equation}
	Obviously, $G(x)=-\frac{1}{2\pi}\ln |x|$ since the above equation is equivalent to
	\begin{equation}\begin{cases}
			&-\Delta G=\delta_0,\ \ x\in B_1,\\
			&G(x)=0,\ \ x\in \partial B_1.
	\end{cases}\end{equation}
	Furthermore, for any $\delta>0$, we have $$\lim\limits_{\lambda\rightarrow 0}\int_{\mathbb{B}^2 \backslash B_{\mathbb{B}^2}(0, \delta)}|\nabla_{\mathbb{B}^2}u_\lambda|^2dV_{\mathbb{B}^2}=0$$ through Lemma \ref{lem2-1}. Hence from Moser-Trudinger inequality in the Poincar\'{e} disk $\mathbb{B}^2$ \eqref{MT}, for any $r\geq 2$, there holds
	$$\int_{\mathbb{B}^2 \backslash B_{\mathbb{B}^2}(0, \delta)}|f_\lambda|^rdV_{\mathbb{B}^2}\leq (\lambda c_\lambda)^r \int_{\mathbb{B}^2 \backslash B_{\mathbb{B}^2}(0,\delta)}u_\lambda^r e^{ru_\lambda^2}dV_{\mathbb{B}^2}\lesssim 1.$$
	Standard elliptic estimate gives that $c_\lambda u_\lambda$ strongly converges to $G$ in $C^2_{\rm loc}(\mathbb{B}^2\backslash\{0\})$.
	Then we accomplish the proof of Lemma \ref{lem-7-31-1}.	
\end{proof}

Recall that
\begin{equation}\label{eq-appA-1}
	-\Delta u_\lambda = \lambda u_\lambda e^{u_\lambda^2}\Big(\frac{2}{1-|x|^{2}}\Big)^{2} \hbox{ in }B_{1},\,\,\,
	u_\lambda=0\,\,\,\text{on}\,\,\,\partial B_{1},
\end{equation}
where $\lambda r^{2}_{\lambda}c^{2}_\lambda e^{c^{2}_{\lambda}}=1.$

We perform the change of variables
\begin{equation}\label{eq-appA-2}
	t=\ln\left(1+\frac{r^2}{r_\lambda^2}\right).
\end{equation}
Letting $\tilde{u}_{\lambda}(t):=u_{\lambda}(r),$ then we can rewrite equation \eqref{eq-appA-1} as
\begin{equation}\label{eq-appA-3}
	e^{t}\left(\left(1-e^{-t}\right) \tilde{u}_\lambda'\right)' = -\frac{\tilde{u}_\lambda}{c_\lambda^2}
	e^{2t+\tilde{u}_\lambda^2-c_\lambda^2}
	\Big(\frac{1}{1-r^{2}_{\lambda}(e^{t}-1)}\Big)^{2}.
\end{equation}

\begin{lemma}\label{lemma-appA}(Lemma 5.1, \cite{DT})
	The solution $\varphi$ of
	$${\mathcal L}\left(\varphi\right)=e^t\left(\left(1-e^{-t}\right) \varphi'\right)' + 2\varphi=F$$
	with $\varphi(0)=0$ and $F$ smooth is
	$$
	\varphi(t)=\int_0^t e^{-s}F(s)\left(\left(1-2e^{-t}\right)\left(1-2e^{-s}\right)  \ln \frac{e^t-1}{e^s-1} + 4\left(e^{-s}-e^{-t}\right)\right)ds.
	$$
\end{lemma}

Let us define
\begin{equation}\label{eq-appA-4}
	\varphi_0(t)=\int_0^t e^{-s}\left(s-s^2\right)\left(\left(1-2e^{-t}\right)\left(1-2e^{-s}\right)  \ln \frac{e^t-1}{e^s-1} + 4\left(e^{-s}-e^{-t}\right)\right)ds,
\end{equation}
\begin{equation}\label{psi}
	\begin{aligned}
		\psi_0(t)=\int_0^t &e^{-s}\left(-\varphi_0(s)-2\varphi_0^2(s)+4s\varphi_0(s)-2s^2\varphi_0(s)+s^3-\frac{1}{2}s^4\right)\\
		&\cdot\left(\left(1-2e^{-t}\right)\left(1-2e^{-s}\right)  \ln \frac{e^t-1}{e^s-1} + 4\left(e^{-s}-e^{-t}\right)\right)ds,
	\end{aligned}
\end{equation}
\begin{equation}\label{chi}
	\begin{aligned}
		&\chi_0(t)=\int_0^t e^{-s}\Bigg(-\psi_0(s)-3\varphi_0^2(s)+4s\psi_0(s)-4\varphi_0(s)\psi_0(s)-3\varphi_0(s)s^2+6\varphi_0^2(s)s\\&-2\psi_0(s)s^2
		+4\varphi_0(s)s^3+\frac{1}{2}s^5\Bigg)
		\cdot\left(\left(1-2e^{-t}\right)\left(1-2e^{-s}\right)  \ln \frac{e^t-1}{e^s-1} + 4\left(e^{-s}-e^{-t}\right)\right)ds,
	\end{aligned}
\end{equation}
so that, by lemma \ref{lemma-appA}, there holds
\begin{equation}\label{eq-appA-5}
	{\mathcal L}\left(\varphi_0\right)(t) = t-t^2,
\end{equation}
\begin{equation}\label{Lpsi}
	{\mathcal L}\left(\psi_0\right)(t) =-\varphi_0(t)-2\varphi_0^2(t)+4t\varphi_0(t)-2t^2\varphi_0(t)+t^3-\frac{1}{2}t^4,
\end{equation}
\begin{equation}\label{Lchi}
	\begin{aligned}
	{\mathcal L}\left(\chi_0\right)(t) ={}&-\psi_0(t)-3\varphi_0^2(t)+4t\psi_0(t)-4\varphi_0(t)\psi_0(t)\\
	{}&-3\varphi_0(t)t^2+6\varphi_0^2(t)t-2\psi_0(t)t^2+4\varphi_0(t)t^3+\frac{1}{2}t^5.
	\end{aligned}
\end{equation}
Moreover, we have $\varphi_0(t)=w_0(r)$, $\psi_0(t)=z_0(r)$, $\chi_0(t)=\zeta_0(r)$,
\begin{equation}\label{eq-appA-6}
	\left\vert \varphi_0(t)+t\right\vert \le C_0\hbox{ and }\varphi_0'(t)\to -1\hbox{ as }t\to +\infty,
\end{equation}
\begin{equation}\label{psi0}
	\left\vert \psi_0(t)-\frac{\beta}{2}t\right\vert \le C_0\hbox{ and }\psi_0'(t)\to \frac{\beta}{2}\hbox{ as }t\to +\infty,
\end{equation}
\begin{equation}\label{chi0}
	\left\vert \chi_0(t)-\frac{\gamma}{2}t\right\vert\le C_0\hbox{ and }\chi_0'(t)\to \frac{\gamma}{2}\hbox{ as }t\to +\infty,
\end{equation}
for some $C_0>0$, where $w_0$ is defined in Lemma \ref{lem2}, $z_0,\beta$ are defined in Lemma \ref{lem6} and $\zeta_0,\gamma$ are defined in Remark \ref{rema}.

Set
\begin{equation}\label{eq-appA-7}
	\tilde{u}_{\lambda}(t)=c_{\lambda}-\frac{t}{c_{\lambda}}
	+\frac{\varphi_{0}(t)}{c^{3}_{\lambda}}+\frac{\psi_0(t)}{c_\lambda^5}+\frac{\chi_0(t)}{c_\lambda^7}+R_{\lambda}(t).
\end{equation}

\begin{lemma}\label{claim-appA-3}
	Fix $\delta\in (0,1).$
	There exists $D_0>0$ such that
	$$
	\left\vert R_\lambda'(t)\right\vert \le D_0 c_\lambda^{-9} \hbox{ for all }0\le t\le \min\left\{c_{\lambda}^{2}-T_{\lambda},
	\ln\left(1+\frac{\delta^{2}}{r^{2}_{\lambda}}\right)\right\},
	$$
	where $T_\lambda$ is any sequence such that $T_\lambda =o\left(c_\lambda\right)$ and $c_\lambda^ke^{-T_\lambda}\to 0$ as $\lambda\to 0$ for all $k$.
\end{lemma}

\begin{proof}
	Fix such a sequence $T_\lambda$. Let $D_0>0$ that we shall choose later. Since $R_\lambda'(0)=0$, there exists $0<\tilde{t}_\lambda\le \min\left\{c_{\lambda}^{2}-T_{\lambda}, \ln\left(1+\frac{\delta^{2}}{r^{2}_{\lambda}}\right)\right\}
	$ such that
	\begin{equation}\label{eq-claim-appA-3-1}
		\left\vert R_\lambda'(t)\right\vert \le D_0 c_\lambda^{-9} \hbox{ for all }0\le t\le \tilde{t}_\lambda.
	\end{equation}
	Define
	\begin{equation}
		t_\lambda:=\max\left\{\tilde{t}_\lambda:\left\vert R_\lambda'(t)\right\vert \le D_0 c_\lambda^{-9} \hbox{ for all }0\le t\le \tilde{t}_\lambda\right\}.
	\end{equation}
	Since $R_\lambda(0)=0$, we know
	\begin{equation}\label{eq-claim-appA-3-2}
		\left\vert R_\lambda(t)\right\vert \le D_0 c_{\lambda}^{-9} t\,\,\, \hbox{ for all }0\le t\le t_\lambda.
	\end{equation}
	Assume for contradiction that $t_\lambda<\min\left\{c_{\lambda}^{2}-T_{\lambda}, \ln\left(1+\frac{\delta^{2}}{r^{2}_{\lambda}}\right)\right\}$. We have
	\begin{equation}\label{eq-claim-appA-3-2b}
		\left\vert R_\lambda'\left(t_\lambda\right)\right\vert = D_0 c_\lambda^{-9}.
	\end{equation}
	This is the statement we will contradict by an appropriate choice of $D_0$.
	Let us use \eqref{eq-appA-3}, \eqref{eq-appA-5}, \eqref{Lpsi}, \eqref{Lchi} and \eqref{eq-appA-7} to write that
	$$
	{\mathcal L}\left(R_\lambda\right) = F_\lambda,
	$$
	where
	\begin{align*}
	F_\lambda ={}& \frac{1}{c_\lambda} - \frac{\tilde{u}_\lambda}{c_\lambda^2}
	e^{2t+\tilde{u}_\lambda^2-c_\lambda^2}
	\Big(\frac{1}{1-r^{2}_{\lambda}(e^{t}-1)}\Big)^{2} + 2 R_\lambda \\
	{}&-c_\lambda^{-3}\left({\mathcal{L}}(\varphi_0)-2\varphi_0\right)-c_\lambda^{-5}(\mathcal{L}(\psi_0)-2\psi_0)-c_\lambda^{-7}(\mathcal{L}(\chi_0)-2\chi_0).
	\end{align*}
	For $0\le t\le \min\left\{t_\lambda,T_\lambda\right\}$,
	we have that
	$$
	r_{\lambda}^{2}(e^{t}-1)\lesssim e^{-\alpha^2 c^{2}_{\lambda}}e^{T_\lambda}=o(c_\lambda^{-9}),
	$$
	which implies that
	\begin{equation}\label{add-6}
		\Big(\frac{1}{1-r_{\lambda}^{2}(e^{t}-1)}\Big)^{2}
		=1+o(c_\lambda^{-9}),\,\,\text{for}\,\,0\le t\le \min\left\{t_\lambda,T_\lambda\right\}.
	\end{equation}
	For $0\le t\le \min\left\{t_\lambda,T_\lambda\right\}$,
	we also have
	\begin{align*}
	2t+\tilde{u}_\lambda^2-c_\lambda^2 = {}&\frac{t^2}{c_\lambda^2} -\frac{2t\varphi_0}{c_\lambda^4}-\frac{2t\psi_0}{c_\lambda^6}-\frac{2t\chi_0}{c_\lambda^8}+2c_\lambda R_\lambda \\
	{}&+ \frac{2\varphi_0}{c_\lambda^2}+\frac{2\psi_0}{c_\lambda^4}+\frac{2\chi_0}{c_\lambda^6}+\frac{\varphi_0^2}{c_\lambda^6}+\frac{2\varphi_0\psi_0}{c_\lambda^8}+ o\left(c_\lambda^{-8}\right)\\
	={}&o(1),
	\end{align*}
	so that
	%	\begin{equation*}
		%		\begin{aligned}
			%			\left( 2t+\tilde{u}_\lambda^2-c_\lambda^2 \right)^2 = {}&\frac{4t^2c_\lambda R_\lambda}{c_\lambda^2}+\frac{4\varphi_0^2+4t^2\varphi_0+t^4}{c_\lambda^4}+\frac{8\varphi_0\psi_0-8t\varphi_0^2+4t^2\psi_0-4t^3\varphi_0}{c_\lambda^6}\\
			%			{}&+\frac{4\varphi_0^3+4\psi_0^2+8\varphi_0\chi_0-16t\varphi_0\psi_0+6t^2\varphi_0^2+4t^2\chi_0-4t^3\psi_0}{c_\lambda^8}\\
			%			{}&+\frac{12\varphi_0^2\psi_0-4t\varphi_0^3-8t\psi_0^2-16t\varphi_0\chi_0+12t^2\varphi_0\psi_0-4t^3\chi_0}{c_\lambda^{10}}+o(c_\lambda^{-8}),
			%		\end{aligned}
		%	\end{equation*}
	%	\begin{equation*}
		%		\begin{aligned}
			%			\left( 2t+\tilde{u}_\lambda^2-c_\lambda^2 \right)^3 = \frac{}{c_\lambda^4}+ \frac{}{c_\lambda^6}+ \frac{}{c_\lambda^8}+ \frac{}{c_\lambda^{10}}+ \frac{}{c_\lambda^{12}}+o(c_\lambda^{-8}).
			%		\end{aligned}
		%	\end{equation*}
	%	and
	$$
	\left\vert 2t+\tilde{u}_\lambda^2-c_\lambda^2 \right\vert^4\lesssim(1+t^8)c_\lambda^{-8},
	$$
	thanks to \eqref{eq-appA-6}-\eqref{chi0} and \eqref{eq-claim-appA-3-2}.  Moreover, by \eqref{eq-appA-7} we get
	$$
	\frac{\tilde{u}_\lambda}{c_\lambda^2} = c_\lambda^{-1} - c_\lambda^{-3} t +c_\lambda^{-5}\varphi_0 +c_\lambda^{-7}\psi_0 +c_\lambda^{-9}\chi_0+ o\left(c_\lambda^{-10}\right).
	$$
	Thus, by using
	\begin{align*}
		{}&\left\vert e^{2t+\tilde{u}_\lambda^2-c_\lambda^2}-1-(2t+\tilde{u}_\lambda^2-c_\lambda^2)-\frac{1}{2}\left( 2t+\tilde{u}_\lambda^2-c_\lambda^2 \right)^2-\frac{1}{6}\left( 2t+\tilde{u}_\lambda^2-c_\lambda^2 \right)^3\right\vert\\
	\lesssim	{}&\left\vert 2t+\tilde{u}_\lambda^2-c_\lambda^2 \right\vert^4\lesssim(1+t^8)c_\lambda^{-8},
	\end{align*}
	we get
	$$
	\left\vert F_\lambda\right\vert \le D_1 \left(1+t^8\right)c_\lambda^{-9}
	$$
	for all $0\le t\le \min\left\{t_\lambda,T_\lambda\right\},$ where $D_1$ depends on $C_0$ but not on $D_0$. We can use the representation formula of Lemma \ref{lemma-appA} to deduce that
	\begin{equation}\label{7-30-2}
		\begin{split}
			{}&\left\vert R_\lambda'(t)\right\vert \\
			\le{}& D_1c_\lambda^{-9} \int_0^t e^{-s}\left(1+s^8\right)\left\vert 2e^{-t}\left(1-2e^{-s}\right)  \ln \frac{e^t-1}{e^s-1} +\frac{e^t-2}{e^t-1}\left(1-2e^{-s}\right)+ 4e^{-t}\right\vert ds \\
			\le{}&  D_2 c_\lambda^{-9}
		\end{split}
	\end{equation}
	for all $0\le t\le \min\left\{t_\lambda,T_\lambda
	\right\},$ where $D_2$ depends only on $C_0$, not on $D_0$. Up to choosing $D_0>2D_2$, we get that $t_\lambda>T_\lambda$ thanks to \eqref{eq-claim-appA-3-2b}. Moreover we have that
	\begin{equation}\label{eq-claim-appA-3-3}
		\left\vert R_\lambda'\left(T_\lambda\right)\right\vert \le D_2 c_\lambda^{-9}.
	\end{equation}
	
	For all $T_\lambda\le t\le t_\lambda$, we have
	\begin{equation*}
		r_{\lambda}^{2}(e^{t}-1)\leq r_\lambda^2\Big(e^{ \ln\big(1+\frac{\delta^2}{r_\lambda^2}\big)}-1\Big)=\delta^2
	\end{equation*}
	and
	\begin{equation*}
		2t+\tilde{u}_\lambda^2-c_\lambda^2=\frac{t^2}{c_\lambda^2}+O(1).
	\end{equation*}
	Then thanks to \eqref{eq-appA-5}-\eqref{chi0}, we can write that
	$$
	\left\vert F_\lambda(t)\right\vert \le Cc_\lambda e^{t^2/c_\lambda^2}
	$$
	for some $C>0$, depending on $D_0$ and $C_0$. As a result, we deduce that
	\begin{equation}\label{7-30-3}
		\begin{split}
			{}&\left\vert R_\lambda'(t)-R_\lambda'\left(T_\lambda\right)\right\vert\\
			\le{}& C c_\lambda\int_{T_\lambda}^t e^{s^2/c_\lambda^2-s}\left\vert 2e^{-t}\left(1-2e^{-s}\right)  \ln \frac{e^t-1}{e^s-1} +\frac{e^t-2}{e^t-1}\left(1-2e^{-s}\right)+ 4e^{-t}\right\vert ds\\
			\le{}&  Cc_\lambda\int_{T_\lambda}^t e^{s^2/c_\lambda^2-s}ds
			= O\left(c_\lambda\int_{T_\lambda}^{c_\lambda^2-T_{\lambda}} e^{s^2/c_\lambda^2-s}ds\right)\\
			={}& O\left( c_\lambda \int_{T_\lambda}^{\frac{1}{2}c_\lambda^2} e^{s^2/c_\lambda^2-s}\, ds\right)
			= O\left(c_\lambda \int_{T_\lambda}^{\frac{1}{2}c_\lambda^2} e^{-\frac{1}{2}s}\, ds\right)\\
			={}&O\left(c_\lambda e^{-\frac{1}{2}T_\lambda}\right)=o\left(c_\lambda^{-9}\right).
		\end{split}
	\end{equation}
	Combining with \eqref{eq-claim-appA-3-3} gives that
	$$
	\left\vert R_\lambda'(t)\right\vert \le D_2c_\lambda^{-9}+o\left(c_\lambda^{-9}\right).
	$$
	This proves that \eqref{eq-claim-appA-3-2b} is impossible up to choosing $D_0> 2D_2$. This ends the proof of the lemma.
\end{proof}

If we want to push the estimates a little bit further, we can get

\begin{lemma}\label{add-lem-5}
	Fix $\delta\in(0,1)$. There exists $E_0>0$ such that
	$$
	\left\vert \tilde{u}_\lambda-c_\lambda+\frac{t}{c_\lambda}
	-\frac{\varphi_0(t)}{c_\lambda^3}-\frac{\psi_0(t)}{c_\lambda^5}-\frac{\gamma t}{2c_\lambda^7}\right\vert \le E_0c_\lambda^{-7}
	$$
	for all $0\le t\le  \ln\left(1+\frac{\delta^{2}}{r^{2}_{\lambda}}\right)$.
\end{lemma}

\begin{proof}
	It is clear that the assertion holds for any $$0\le t\le \min\left\{c_\lambda^2-T_\lambda, \ln\left(1+\frac{\delta^{2}}{r^{2}_{\lambda}}\right)\right\}$$ thanks to Lemma \ref{claim-appA-3} and \eqref{chi0}.
	If $ \ln\left(1+\frac{\delta^{2}}{r^{2}_{\lambda}}\right)\leq c^{2}_{\lambda}-T_{\lambda},$ then we end the proof.
	So we may assume that $c^{2}_{\lambda}-T_{\lambda} \leq  \ln\left(1+\frac{\delta^{2}}{r^{2}_{\lambda}}\right).$
	By \eqref{eq-appA-7} we know that
	\begin{equation}\label{eq-claim-appA-4-1}
		\tilde{u}_\lambda\left(c_\lambda^2-T_\lambda\right)= \frac{T_\lambda}{c_\lambda} +\frac{\varphi_0(c_\lambda^2-T_\lambda)}{c_\lambda^3}+\frac{\psi_0(c_\lambda^2-T_\lambda)}{c_\lambda^5}+\frac{\chi_0(c_\lambda^2-T_\lambda)}{c_\lambda^7}
		+O\left(c_\lambda^{-7}\right),
	\end{equation}
	\begin{equation}\label{eq-claim-appA-4-2}
		\tilde{u}_\lambda'\left(c_\lambda^2-T_\lambda\right)= -\frac{1}{c_\lambda} +\frac{\varphi_0'(c_\lambda^2-T_\lambda)}{c_\lambda^3}+\frac{\psi_0'(c_\lambda^2-T_\lambda)}{c_\lambda^5}+\frac{\chi_0'(c_\lambda^2-T_\lambda)}{c_\lambda^7}+O\left(c_\lambda^{-9}\right).
	\end{equation}
	Let $t_\lambda=c_\lambda^2-\alpha_\lambda$ for $c_\lambda^2- \ln\left(1+\frac{\delta^{2}}{r^{2}_{\lambda}}\right)\le \alpha_\lambda\le T_\lambda$. From equation \eqref{eq-appA-3},
	we have
	\begin{equation}\label{add8-16-1}
		\big[(1-e^{-t})\tilde{u}'_\lambda(t)\big]'=-\frac{\tilde{u}_\lambda(t) }{c^{2}_{\lambda}} e^{t+\tilde{u}_\lambda(t)^2-c_\lambda^2}
		\Big(\frac{1}{1-r^{2}_{\lambda}(e^{t}-1)}\Big)^{2}.
	\end{equation}
	Integrating \eqref{add8-16-1} between $c_\lambda^2-T_\lambda$ and $t_\lambda,$  integrating by parts we have
	\begin{equation}\label{u_t}
		\begin{aligned}
			\tilde{u}_\lambda(t_\lambda)={}& \tilde{u}_\lambda(c_\lambda^2-T_\lambda)+\tilde{u}_\lambda'
			(c_\lambda^2-T_\lambda)\left(1-e^{T_\lambda-c_\lambda^2}\right)
			\ln \left(\frac{e^{c_\lambda^2-\alpha_\lambda}-1}
			{e^{c_\lambda^2-T_\lambda}-1}\right)\\
			{}&-\frac{1}{c_\lambda^2}\int_{c_\lambda^2-T_\lambda}^{c_\lambda^2-
				\alpha_\lambda}
			\ln \left(\frac{e^{t_\lambda}-1}{e^t-1}\right)\tilde{u}_\lambda(t) e^{t+\tilde{u}_\lambda(t)^2-c_\lambda^2}
			\Big(\frac{1}{1-r^{2}_{\lambda}(e^{t}-1)}\Big)^{2}dt.
		\end{aligned}
	\end{equation}
	Similar to \eqref{u_t}, by \eqref{eq-appA-5}-\eqref{Lchi} we can get
	\begin{equation}\label{varphi_t}
		\begin{aligned}
			\varphi_0(t_\lambda)={}&\varphi_0(c_\lambda^2-T_\lambda)+\varphi_0'
			(c_\lambda^2-T_\lambda)\left(1-e^{T_\lambda-c_\lambda^2}\right)
			\ln \left(\frac{e^{c_\lambda^2-\alpha_\lambda}-1}
			{e^{c_\lambda^2-T_\lambda}-1}\right)\\
			{}&+\int_{c_\lambda^2-T_\lambda}^{c_\lambda^2-
				\alpha_\lambda} \ln \left(\frac{e^{t_\lambda}-1}{e^t-1}\right)
			(\mathcal{L}(\varphi_0)(t)-2\varphi_0(t))e^{-t}dt,
		\end{aligned}
	\end{equation}
	\begin{equation}\label{psi_t}
		\begin{aligned}
			\psi_0(t_\lambda)={}&\psi_0(c_\lambda^2-T_\lambda)+\psi_0'
			(c_\lambda^2-T_\lambda)\left(1-e^{T_\lambda-c_\lambda^2}\right)
			\ln \left(\frac{e^{c_\lambda^2-\alpha_\lambda}-1}
			{e^{c_\lambda^2-T_\lambda}-1}\right)\\
			{}&+\int_{c_\lambda^2-T_\lambda}^{c_\lambda^2-
				\alpha_\lambda} \ln \left(\frac{e^{t_\lambda}-1}{e^t-1}\right)(\mathcal{L}(\psi_0)(t)-2\psi_0(t))e^{-t}dt,
		\end{aligned}
	\end{equation}
	\begin{equation}\label{chi_t}
		\begin{aligned}
			\chi_0(t_\lambda)={}&\chi_0(c_\lambda^2-T_\lambda)+\chi_0'
			(c_\lambda^2-T_\lambda)\left(1-e^{T_\lambda-c_\lambda^2}\right)
			\ln \left(\frac{e^{c_\lambda^2-\alpha_\lambda}-1}
			{e^{c_\lambda^2-T_\lambda}-1}\right)\\
			{}&+\int_{c_\lambda^2-T_\lambda}^{c_\lambda^2-
				\alpha_\lambda} \ln \left(\frac{e^{t_\lambda}-1}{e^t-1}\right)(\mathcal{L}(\chi_0)(t)-2\chi_0(t))e^{-t}dt.
		\end{aligned}
	\end{equation}
	By \eqref{eq-appA-5}-\eqref{chi0}, \eqref{eq-claim-appA-4-1}-\eqref{chi_t} and $\alpha_\lambda\leq T_\lambda=o(c_\lambda)$, we derive that
	\begin{equation*}
		\frac{1}{c_\lambda^3}\int_{c_\lambda^2-T_\lambda}^{c_\lambda^2-
			\alpha_\lambda} \ln \left(\frac{e^{t_\lambda}-1}{e^t-1}\right)(\mathcal{L}(\varphi_0)(t)-2\varphi_0(t))e^{-t}dt=o(c_\lambda^{-7}),
	\end{equation*}
	\begin{equation*}
		\frac{1}{c_\lambda^5}\int_{c_\lambda^2-T_\lambda}^{c_\lambda^2-
			\alpha_\lambda} \ln \left(\frac{e^{t_\lambda}-1}{e^t-1}\right)(\mathcal{L}(\psi_0)(t)-2\psi_0(t))e^{-t}dt=o(c_\lambda^{-7}),
	\end{equation*}
	\begin{equation*}
		\frac{1}{c_\lambda^7}\int_{c_\lambda^2-T_\lambda}^{c_\lambda^2-
			\alpha_\lambda} \ln \left(\frac{e^{t_\lambda}-1}{e^t-1}\right)(\mathcal{L}(\chi_0)(t)-2\chi_0(t))e^{-t}dt=o(c_\lambda^{-7}),
	\end{equation*}
	so that
	\begin{equation}\label{eq-claim-appA-4-3}
		\begin{aligned}
			{}&\left\vert\tilde{u}_\lambda\left(t_\lambda\right)
			-c_\lambda+\frac{t_\lambda}{c_\lambda}-\frac{\varphi_0(t_\lambda)}{c_\lambda^3}-\frac{\psi_0(t_\lambda)}{c_\lambda^5}-\frac{\gamma t_\lambda}{2c_\lambda^7}\right\vert\\
			\leq{}&\frac{E_1}{c_\lambda^7}
			+\frac{1}{c_\lambda^2}\int_{c_\lambda^2-T_\lambda}^{c_\lambda^2-
				\alpha_\lambda}
			\ln \left(\frac{e^{t_\lambda}-1}{e^t-1}\right)\tilde{u}_\lambda(t) e^{t+\tilde{u}_\lambda(t)^2-c_\lambda^2}
			\Big(\frac{1}{1-r^{2}_{\lambda}(e^{t}-1)}\Big)^{2}dt,
		\end{aligned}
	\end{equation}
	where $E_1$ depends only on $C_0$.
	Assume that the statement of the lemma holds up to $t_\lambda$, i.e.
	$$
	t_\lambda:=\max\left\{\tilde{t}_\lambda:\left\vert \tilde{u}_\lambda-c_\lambda+\frac{t}{c_\lambda}
	-\frac{\varphi_0(t)}{c_\lambda^3}-\frac{\psi_0(t)}{c_\lambda^5}-\frac{\gamma t}{2c_\lambda^7}\right\vert \le E_0c_\lambda^{-7}\hbox{ for all }0\le t\le \tilde{t}_\lambda\right\},
	$$
	where $E_0$ satisfies $E_0>2E_1$, then we can write that
	$$
	\ln \left(\frac{e^{t_\lambda}-1}{e^t-1}\right)\tilde{u}_\lambda(t) e^{t+\tilde{u}_\lambda(t)^2-c_\lambda^2} \Big(\frac{1}{1-r^{2}_{\lambda}(e^{t}-1)}\Big)^{2}= O\left(c_\lambda^{-1} \left(1+s^2\right)e^{-s}\right)
	$$
	in $\left[c_\lambda^2-T_\lambda,t_\lambda\right]$ with $t=c_\lambda^2-s$ so that it is easily checked that
	$$
	\frac{1}{c_\lambda^2}\int_{c_\lambda^2-T_\lambda}^{c_\lambda^2-
		\alpha_\lambda}
	\ln \left(\frac{e^{t_\lambda}-1}{e^t-1}\right)\tilde{u}_\lambda(t) e^{t+\tilde{u}_\lambda(t)^2-c_\lambda^2}\Big(\frac{1}{1-r^{2}_{\lambda}(e^{t}-1)}\Big)^{2}dt = o\left(c_\lambda^{-7}\right),
	$$
	which infers by \eqref{eq-claim-appA-4-3} that
	$$\left\vert\tilde{u}_\lambda\left(t_\lambda\right)
	-c_\lambda+\frac{t_\lambda}{c_\lambda}-\frac{\varphi_0(t_\lambda)}{c_\lambda^3}-\frac{\psi_0(t_\lambda)}{c_\lambda^5}-\frac{\gamma t_\lambda}{2c_\lambda^7}\right\vert\leq \frac{E_1}{c_\lambda^7}+o(c_\lambda^{-7}).$$
	This contradicts with the maximality of $t_\lambda$ and thus ends the proof of lemma.
\end{proof}

\begin{proof}[\textbf{Proof of Proposition \ref{lem-7-31-2}}]
	By Lemma \ref{lem-7-31-1}, we have
	\begin{equation}\label{31-3}
		c_{\lambda}u_{\lambda}\left(\delta\right)= -2\ln\delta+o_{\lambda}(1).
	\end{equation}
	Noting that $\lambda r_{\lambda}^{2}c^{2}_{\lambda}e^{c^{2}_{\lambda}}=1$, $\varphi_0(t)=w_0(r)$ and $\psi_0(t)=z_0(r)$,
	we can apply Lemma \ref{add-lem-5} to obtain
	\begin{equation}\label{31-4}
		\begin{aligned}
			{}&c_{\lambda} u_{\lambda}\left(\delta\right) \\= {}&c_\lambda\left[c_\lambda-\frac{1}{c_\lambda}\ln\left(1+\frac{\delta^{2}}{r^{2}_{\lambda}}\right)+\frac{w_0(\delta)}{c_\lambda^3}+\frac{z_0(\delta)}{c_\lambda^5}+\frac{\gamma}{2c_\lambda^7}\ln\left(1+\frac{\delta^{2}}{r^{2}_{\lambda}}\right)+O(c_\lambda^{-7})\right]\\
			={}&
			-\left(1+O(c_\lambda^{-6})\right)\ln\left(\lambda c_{\lambda}^2\right)-2\ln\delta+c^{-2}_{\lambda}\left[w_0(\delta)+c_\lambda^{-2}\left(z_0(\delta)+\frac{\gamma}{2}\right)\right]+O(c_\lambda^{-6}).
		\end{aligned}
	\end{equation}
	It follows from \eqref{31-3} and \eqref{31-4} that
	$$
	\ln(\lambda c_{\lambda}^2)=o_{\lambda}(1),\hskip.1cm
	$$
	which ends the proof of this proposition.
\end{proof}

Next we present a more precise characterization of the relation between $\lambda$ and $c_\lambda$ when $\lambda\to0$.
%First we need to obtain a precise relation between $\lambda$ and $c_{\lambda}.$
\begin{proposition}\label{add-lem6}
	There holds that
	\begin{equation}\label{lambda_gamma}
		c_\lambda= \lambda^{-\frac{1}{2}}+A\lambda^{\frac{1}{2}}+B\lambda^{\frac{3}{2}}
		+o_{\lambda}\big(\lambda^{\frac{3}{2}}\big),
	\end{equation}
	for some constants $A, B$.
\end{proposition}

In order to prove Proposition \ref{add-lem6}, we have to make some preparations.
\begin{lemma}\label{adlem1}
	For any $\varphi(x)\in C_c^{\infty}(\mathbb{B}^2),$
	there holds that
	\begin{equation}\label{add-L1-9-11}
		\lim_{\lambda\rightarrow 0}c^{2}_{\lambda}\int_{\mathbb{B}^2}\left(\lambda c_{\lambda}u_{\lambda}e^{u^{2}_{\lambda}}
		-\pi\delta_{0}\right)\varphi dV_{\mathbb{B}^2}=4\pi\varphi(0)+4\pi\int_{\mathbb{B}^2}G\varphi dV_{\mathbb{B}^2},
	\end{equation}
	where $\delta_0$ is the Dirac function.
\end{lemma}

\begin{proof}
	For any $0<\sigma<1$, pick $\widetilde{s_\lambda}$ such that $u(r_\lambda \widetilde{s_\lambda})\leq \sigma c_{\lambda}.$  Through monotonicity formula of Lemma \ref{lem3}, one can easily check that $u(r_\lambda \widetilde{s_\lambda})\leq \sigma c_{\lambda}$ if $\widetilde{s_\lambda}^2=e^{(1-\sigma)c_\lambda^2}$.
	First, we have
	\begin{equation}\label{8-1-1}
		\begin{aligned}
			{}&c^{2}_{\lambda}\int_{\mathbb{B}^2}\left(\lambda c_{\lambda}u_{\lambda}e^{u^{2}_{\lambda}}
			-\pi\delta_{0}\right)\varphi dV_{\mathbb{B}^2}\\
			={}&c^{2}_{\lambda}\int_{B_{r_{\lambda}\widetilde{s_\lambda}}}\left(\lambda c_{\lambda}u_{\lambda}e^{u^{2}_{\lambda}}
			-\pi\delta_{0}\right)\varphi dV_{\mathbb{B}^2}+c^{2}_{\lambda}\int_{\mathbb{B}^2 \backslash B_{r_{\lambda}\widetilde{s_\lambda}}}\lambda c_{\lambda}u_{\lambda}e^{u^{2}_{\lambda}}\varphi dV_{\mathbb{B}^2}\\
			:={}&c^{2}_{\lambda}J_{1}+c^{2}_{\lambda}J_{2}.
		\end{aligned}
	\end{equation}
	For $J_1$, noticing that $r_{\lambda}\widetilde{s}_{\lambda}= e^{-\frac{\sigma c^{2}_{\lambda}}{2}}+o_\lambda(1)$, we have
	\begin{equation*}\label{8-1-2}
		\begin{split}
			J_{1}
			={}&
			\varphi(0)\int_{B_{r_{\lambda}\widetilde{s_\lambda}}}\left(\lambda c_{\lambda}u_{\lambda}e^{u^{2}_{\lambda}}
			-\pi\delta_{0}\right)dV_{\mathbb{B}^2}\\
			{}&+\int_{B_{r_{\lambda}\widetilde{s_\lambda}}}\left(\lambda c_{\lambda}u_{\lambda}e^{u^{2}_{\lambda}}
			-\pi\delta_{0}\right)\nabla \varphi(\theta x)\cdot x dV_{\mathbb{B}^2}\\
			={}&\varphi(0)\int_{B_{r_{\lambda}\widetilde{s_\lambda}}}\left(\lambda c_{\lambda}u_{\lambda}e^{u^{2}_{\lambda}}
			-\pi\delta_{0}\right)dV_{\mathbb{B}^2}\\
			{}&+O\Big(e^{-\frac{\sigma c^{2}_{\lambda}}{2}}\|\nabla \varphi\|_{L^{\infty}(B_{r_{\lambda}\widetilde{s_\lambda}})}\int_{B_{r_{\lambda}\widetilde{s_\lambda}}}\left(\lambda c_{\lambda}u_{\lambda}e^{u^{2}_{\lambda}}
			+\pi\delta_{0}\right)dV_{\mathbb{B}^2}\Big).
		\end{split}
	\end{equation*}
	We shall estimate $\int_{B_{r_{\lambda}\widetilde{s_\lambda}}}\lambda c_{\lambda}u_{\lambda}e^{u^{2}_{\lambda}}dV_{\mathbb{B}^2}$ and $\int_{B_{r_{\lambda}\widetilde{s_\lambda}}}\pi\delta_{0}dV_{\mathbb{B}^2}$. Recalling $\lambda r_\lambda^2 c_\lambda^2 e^{c_\lambda^2}=1$, by Lemma \ref{add-lem-5}, we have
	\begin{equation*}
		\begin{split}
			{}&\int_{B_{r_{\lambda}\widetilde{s_\lambda}}}\lambda c_{\lambda}u_{\lambda}e^{u^{2}_{\lambda}}dV_{\mathbb{B}^2}\\
			={}&\int_{B_{r_\lambda\widetilde{s_\lambda}}}\lambda c_{\lambda}
			\Big(c_{\lambda}-\frac{t}{c_{\lambda}}+\frac{\varphi_0(t)}{c^{3}_{\lambda}}+\frac{\beta t}{2c^{5}_{\lambda}}
			+O\big(c^{-5}_{\lambda}\big)\Big)\frac{4}{(1-|x|^2)^{2}}\\
			{}&\quad\quad\quad\quad\cdot\exp\left\{\left(c_{\lambda}-\frac{t}{c_{\lambda}}+\frac{\varphi_0(t)}{c^{3}_{\lambda}}+\frac{\beta t}{2c^{5}_{\lambda}}
				+O\big(c^{-5}_{\lambda}\big)\right)^{2}\right\}dx\\
			={}&\int_{B_{\widetilde{s_\lambda}}}r_\lambda^2\lambda c_{\lambda}
			\Big(c_{\lambda}-\frac{\tilde{t}}{c_{\lambda}}+\frac{\varphi_0(\tilde{t})}{c^{3}_{\lambda}}+\frac{\beta \tilde{t}}{2c^{5}_{\lambda}}
			+O\big(c^{-5}_{\lambda}\big)\Big)\frac{4}{(1-r_\lambda^2|x|^2)^{2}}dx\\
			{}&\quad\quad\quad\quad\cdot\exp\left\{\left(c_{\lambda}-\frac{\tilde{t}}{c_{\lambda}}+\frac{\varphi_0(\tilde{t})}{c^{3}_{\lambda}}+\frac{\beta \tilde{t}}{2c^{5}_{\lambda}}
				+O\big(c^{-5}_{\lambda}\big)\right)^{2}\right\}\\
			={}&\int_0^{\widetilde{s_\lambda}}2\pi r \Big(1-\frac{\tilde{t}}{c_{\lambda}^2}+\frac{\varphi_0(\tilde{t})}{c^{4}_{\lambda}}+\frac{\beta \tilde{t}}{2c^{6}_{\lambda}}
			+O\big(c^{-6}_{\lambda}\big)\Big)\frac{4}{(1-r_\lambda^2|x|^2)^{2}}dx\\
			{}&\quad\quad\quad\quad\cdot\exp\left\{\left(c_{\lambda}-\frac{\tilde{t}}{c_{\lambda}}+\frac{\varphi_0(\tilde{t})}{c^{3}_{\lambda}}+\frac{\beta \tilde{t}}{2c^{5}_{\lambda}}
				+O\big(c^{-5}_{\lambda}\big)\right)^{2}-c_\lambda^2\right\}\\
			={}&\int_0^{\ln \left(1+\widetilde{s_\lambda}^2\right)} \Big(1-\frac{\tilde{t}}{c_{\lambda}^2}+\frac{\varphi_0(\tilde{t})}{c^{4}_{\lambda}}+\frac{\beta \tilde{t}}{2c^{6}_{\lambda}}
			+O\big(c^{-6}_{\lambda}\big)\Big)\frac{4\pi e^{-\tilde{t}}}{(1-r_\lambda^2(e^{\tilde{t}}-1))^{2}}\\
			{}&\quad\quad\quad\quad\cdot\exp\left\{\left(c_{\lambda}-\frac{\tilde{t}}{c_{\lambda}}+\frac{\varphi_0(\tilde{t})}{c^{3}_{\lambda}}+\frac{\beta \tilde{t}}{2c^{5}_{\lambda}}
				+O\big(c^{-5}_{\lambda}\big)\right)^{2}-c_\lambda^2+2\tilde{t}\right\}d\tilde{t},
		\end{split}
	\end{equation*}
	where $t=\ln\left(1+\frac{r^{2}}{r^{2}_{\lambda}}\right)$ and $\tilde{t}=\ln(1+r^{2})$. Similar to Lemma \ref{add-lem-5}, taking $\widetilde{T_\lambda}=o(1)c_\lambda$ such that $c_\lambda^ke^{-\widetilde{T_\lambda}}\to0$ as $\lambda\to0$ for all $k$, then
	\begin{align*}
		{}&\left(c_{\lambda}-\frac{\tilde{t}}{c_{\lambda}}+\frac{\varphi_0(\tilde{t})}{c^{3}_{\lambda}}+\frac{\beta \tilde{t}}{2c^{5}_{\lambda}}
		+O\big(c^{-5}_{\lambda}\big)\right)^{2}-c_\lambda^2+2\tilde{t}\\
		={}&\frac{\tilde{t}^2+2\varphi_0(\tilde{t})}{c_\lambda^2}+\frac{\beta \tilde{t}-2\tilde{t}\varphi_0(\tilde{t})}{c_\lambda^4}+o(c_\lambda^{-4})=o(1) \text{ for }\tilde{t}\in(0,\widetilde{T_\lambda}),
	\end{align*}
	which deduces that
	\begin{equation*}
		\begin{split}
			{}&\int_0^{\widetilde{T_\lambda}}\Big(1-\frac{\tilde{t}}{c_{\lambda}^2}+\frac{\varphi_0(\tilde{t})}{c^{4}_{\lambda}}+\frac{\beta \tilde{t}}{2c^{6}_{\lambda}}
			+O\big(c^{-6}_{\lambda}\big)\Big)\frac{4\pi e^{-\tilde{t}}}{(1-r_\lambda^2(e^{\tilde{t}}-1))^{2}}\\
			{}&\quad\quad\,\,\cdot \exp\left\{\left(c_{\lambda}-\frac{\tilde{t}}{c_{\lambda}}+\frac{\varphi_0(\tilde{t})}{c^{3}_{\lambda}}+\frac{\beta \tilde{t}}{2c^{5}_{\lambda}}
				+O\big(c^{-5}_{\lambda}\big)\right)^{2}-c_\lambda^2+2\tilde{t}\right\}d\tilde{t}\\
			={}&\int_0^{\widetilde{T_\lambda}} \Big(1-\frac{\tilde{t}}{c_{\lambda}^2}+\frac{\varphi_0(\tilde{t})}{c^{4}_{\lambda}}+o\big(c^{-5}_{\lambda}\big)\Big)4\pi e^{-\tilde{t}}\Big(1+o\big(c_\lambda^{-4}\big)\Big)\\
			{}&\quad\quad\,\,\cdot\Big(1+\frac{\tilde{t}^2+2\varphi_0(\tilde{t})}{c_\lambda^2}+\frac{\beta \tilde{t}-2\tilde{t}\varphi_0(\tilde{t})}{c_\lambda^4}+O(1+\tilde{t}^4)c_\lambda^{-4}\Big)d\tilde{t}\\
			={}&4\pi+\frac{4\pi}{c_\lambda^2}+O\big(c_\lambda^{-4}\big).
		\end{split}
	\end{equation*}
	For $t\in \left(\widetilde{T_\lambda},\ln \left(1+\widetilde{s_\lambda}^2\right)\right)$, we have
	\begin{equation*}
		\left(c_{\lambda}-\frac{\tilde{t}}{c_{\lambda}}-\frac{\tilde{t}}{c^{3}_{\lambda}}
		+O\big(c^{-3}_{\lambda}\big)\right)^{2}-c_\lambda^2+2\tilde{t}=\frac{\tilde{t}^2}{c_\lambda^2}+O(1),
	\end{equation*}
	so that
	\begin{equation*}
		\begin{split}
			{}&\int_{\widetilde{T_\lambda}}^{\ln \left(1+\widetilde{s_\lambda}^2\right)} \Big(1-\frac{\tilde{t}}{c_{\lambda}^2}+\frac{\varphi_0(\tilde{t})}{c^{4}_{\lambda}}+\frac{\beta \tilde{t}}{2c^{6}_{\lambda}}
			+O\big(c^{-6}_{\lambda}\big)\Big)\frac{4\pi e^{-\tilde{t}}}{(1-r_\lambda^2(e^{\tilde{t}}-1))^{2}}\\
			{}&\quad\quad\quad\quad\quad\cdot \exp\left\{\left(c_{\lambda}-\frac{\tilde{t}}{c_{\lambda}}+\frac{\varphi_0(\tilde{t})}{c^{3}_{\lambda}}+\frac{\beta \tilde{t}}{2c^{5}_{\lambda}}
				+O\big(c^{-5}_{\lambda}\big)\right)^{2}-c_\lambda^2+2\tilde{t}\right\}d\tilde{t}\\
			={}&O(1)\int_{\widetilde{T_\lambda}}^{\ln \left(1+\widetilde{s_\lambda}^2\right)} e^{\frac{\tilde{t}^2}{c_\lambda^2}-\tilde{t}}d\tilde{t}
			=O(1)\int_{\widetilde{T_\lambda}}^{\frac{1}{2}c_\lambda^2} e^{\frac{\tilde{t}^2}{c_\lambda^2}-\tilde{t}}d\tilde{t}\\
			={}&O(1)\int_{\widetilde{T_\lambda}}^{\frac{1}{2}c_\lambda^2} e^{-\frac{1}{2}\tilde{t}}d\tilde{t}
			=O\left( e^{-\frac{1}{2}\widetilde{T_\lambda}} \right)=O(c_\lambda^{-4}).
		\end{split}
	\end{equation*}
	According to the above estimates, we obtain
	\begin{equation}\label{555}
		\int_{B_{r_{\lambda}\widetilde{s_\lambda}}}\lambda c_{\lambda}u_{\lambda}e^{u^{2}_{\lambda}}dV_{\mathbb{B}^2}=4\pi+\frac{4\pi}{c_\lambda^2}+O\big(c_\lambda^{-4}\big).
	\end{equation}
	Direct computation gives that
	\begin{equation}\label{556}
		\int_{B_{r_{\lambda}\widetilde{s_\lambda}}}\pi\delta_{0}dV_{\mathbb{B}^2}=\int_{B_{r_{\lambda}\widetilde{s_\lambda}}}\pi\delta_{0}\frac{4}{(1-|x|^2)^2}dx=4\pi.
	\end{equation}
	Therefore, we infer that
	\begin{equation}\label{8-1-5-912}
		\begin{split}
			\lim\limits_{\lambda\rightarrow 0}c^{2}_{\lambda}J_{1}
			=4\pi \varphi(0).
		\end{split}
	\end{equation}
	Since $\varphi\in C_c^\infty(\mathbb{B}^2)$, we may assume that the support of $\varphi$ is contained in $B_{\mathbb{B}^2}(0,R)$. Then we have
	\begin{equation}\nonumber
		c^{2}_{\lambda}J_{2}=\lambda c_\lambda^2\int_{B_{\mathbb{B}^2}(0,R) \backslash B_{r_{\lambda}\widetilde{s_{\lambda}}}}\Big(\left(c_\lambda u_\lambda-4\pi G\right)e^{u_\lambda^2}\varphi +4\pi G(e^{u_\lambda^2}-1)\varphi +4\pi G \varphi\Big) dV_{\mathbb{B}^2}.
	\end{equation}
	For the first term, by Lemma \ref{lem-7-31-1} and Proposition \ref{lem-7-31-2}, we can check that
	\begin{equation}\label{ass1}
		\begin{split}
			{}&\lambda c_\lambda^2\int_{B_{\mathbb{B}^2}(0,R) \backslash B_{r_{\lambda}\widetilde{s_{\lambda}}}}(c_\lambda u_\lambda-4\pi G)e^{u_\lambda^2}\varphi  dV_{\mathbb{B}^2}\\
			\leq{}&\lambda c_\lambda^2\|\varphi\|_{L^\infty(\mathbb{B}^2)}\Big(\int_{B_{\mathbb{B}^2}(0,R) \backslash B_{r_{\lambda}\widetilde{s_{\lambda}}}}|c_\lambda u_\lambda-4\pi G|^{p'} dV_{\mathbb{B}^2}\Big)^{\frac{1}{p'}}\Big(\int_{B_{\mathbb{B}^2}(0,R) \backslash B_{r_{\lambda}\widetilde{s_{\lambda}}}}e^{p u_\lambda^2}dV_{\mathbb{B}^2}\Big)^\frac{1}{p}\\
			={}&(1+o_\lambda(1))\|\varphi\|_{L^\infty(\mathbb{B}^2)}o_\lambda(1)\Big(\int_{B_{\mathbb{B}^2}(0,R) \backslash B_{r_{\lambda}\widetilde{s_{\lambda}}}}e^{p u_\lambda^2}dV_{\mathbb{B}^2}\Big)^\frac{1}{p}\\
			={}&o_\lambda(1)
		\end{split}
	\end{equation}
	with $\frac{1}{p}+\frac{1}{p'}=1$, where from Lemma \ref{lem2-1} we use the fact that
	$$
	\lim\limits_{\lambda\rightarrow 0}\int_{B_{\mathbb{B}^2}(0,R) \backslash B_{r_{\lambda}\widetilde{s_{\lambda}}}}|\nabla_{\mathbb{B}^2} u_\lambda|^{2}dV_{\mathbb{B}^2}\leq
	\lim\limits_{\lambda\rightarrow 0}\int_{\mathbb{B}^2}|\nabla_{\mathbb{B}^2}u_{\lambda}^{\sigma}|^2dV_{\mathbb{B}^2}=4\pi \sigma
	$$
	and the Moser-Trudinger inequality \eqref{MT}. For the second term, we also have
	\begin{equation}\label{ass2}
		\begin{split}
			{}&\lambda c_\lambda^2\int_{B_{\mathbb{B}^2}(0,R) \backslash B_{r_{\lambda}\widetilde{s_{\lambda}}}}4\pi G(e^{u_\lambda^2}-1)\varphi  dV_{\mathbb{B}^2}\\
			\leq{}&\lambda c_\lambda^2\|\varphi\|_{L^\infty(\mathbb{B}^2)}\Big(\int_{B_{\mathbb{B}^2}(0,R) \backslash B_{r_{\lambda}\widetilde{s_{\lambda}}}}|4\pi G|^{p'}dV_{\mathbb{B}^2}\Big)^{\frac{1}{p'}}\Big(\int_{B_{\mathbb{B}^2}(0,R) \backslash B_{r_{\lambda}\widetilde{s_{\lambda}}}}|e^{ u_\lambda^2}-1|^{p}dV_{\mathbb{B}^2}\Big)^\frac{1}{p}\\
			={}&(1+o_\lambda(1))\|\varphi\|_{L^\infty(\mathbb{B}^2)}\Big(\int_{B_{\mathbb{B}^2}(0,R) \backslash B_{r_{\lambda}\widetilde{s_{\lambda}}}}|4\pi G|^{p'}dV_{\mathbb{B}^2}\Big)^{\frac{1}{p'}}o_\lambda(1)
			=o_\lambda(1)
		\end{split}
	\end{equation}
	with $\frac{1}{p}+\frac{1}{p'}=1$. For the last term, though Lebesgue dominated convergence theorem we get
	\begin{equation}\label{ass3}
		\lim_{\lambda\to0}\lambda c_\lambda^2\int_{B_{\mathbb{B}^2}(0,R) \backslash B_{r_{\lambda}\widetilde{s_{\lambda}}}}4\pi G \varphi dV_{\mathbb{B}^2}=\int_{\mathbb{B}^2} 4\pi G \varphi dV_{\mathbb{B}^2}.
	\end{equation}
	Therefore, we infer that
	\begin{equation}\label{8-1-5}
		\lim_{\lambda\to0}c_\lambda^2 J_2=\int_{\mathbb{B}^2} 4\pi G \varphi dV_{\mathbb{B}^2}.
	\end{equation}
	Combining the estimates \eqref{8-1-5-912} and \eqref{8-1-5}, we accomplish the proof of Lemma \ref{adlem1}.
\end{proof}

The proof of the following lemma gives a precise expansion of $c_{\lambda} u_{\lambda}$ away from the origin which is very crucial to prove
the uniqueness result. To our best knowledge, this is the first result
which tells us how $c_{\lambda} u_{\lambda}$ converges to $ 4\pi G(x).$

\begin{lemma}\label{ad}
	There holds that
	\begin{equation}\label{lambda_gamm1a}
		c_\lambda= \lambda^{-\frac{1}{2}}+A\lambda^{\frac{1}{2}}
		+o_{\lambda}\big(\lambda^{\frac{1}{2}}\big),
	\end{equation}
	for some constant $A$.		
\end{lemma}

\begin{proof}
	Observe that $c_\lambda^2(c_\lambda u_\lambda-4\pi G)$ satisfies equation
	\begin{equation*}
		\begin{aligned}
			-\Delta_{\mathbb{B}^2}\left(c_\lambda^2(c_\lambda u_\lambda-4\pi G)\right)
			=c^{2}_{\lambda}\left(\lambda c_{\lambda}u_{\lambda}e^{u^{2}_{\lambda}}
			-\pi\delta_{0}\right)=:\widetilde{f_\lambda}.
		\end{aligned}
	\end{equation*}
	Since $\widetilde{f_\lambda}$ can be seen as the linear functional on $L^{\infty}(\mathbb{B}^2)$, according to the principle of uniform boundness, it follows from Lemma \ref{adlem1} that $\widetilde{f_\lambda}$ is bounded in $L^1_{\rm loc}(\mathbb{B}^2)$. Hence one can similarly derive that $c_\lambda^2(c_\lambda u_\lambda-4\pi G)$ is bounded in $W^{1,q}_{\rm loc}(\mathbb{B}^2)$. From the proof of Lemma \ref{adlem1}, we see that $\widetilde{f_\lambda}$ is bounded in $L^{r}_{\rm loc}(\mathbb{B}^2 \backslash \{0\})$. Then standard elliptic estimates give that
	$c_\lambda^2(c_\lambda u_\lambda-4\pi G)$ is locally bounded in $C^2_{\rm loc}(\mathbb{B}^2 \backslash \{0\})$. Hence there exists $W(x)\in C^2_{\rm loc}(\mathbb{B}^2 \backslash \{0\})$ such that $c_\lambda^2(c_\lambda u_\lambda-4\pi G)$ strongly converges to $W(x)$ in
	$C^{1,\gamma}_{\rm loc}(\mathbb{B}^2 \backslash \{0\})$ and $W$ satisfies the equation
	\begin{equation}\label{addi2}
		\begin{cases}
			&-\Delta_{\mathbb{B}^2}W=\pi \delta_0+4\pi G,\quad x\in\mathbb{B}^2,\\
			&W\to0,\ \text{when}\ \rho(x)\to\infty.
		\end{cases}
	\end{equation}
	Hence we have
	\begin{equation}\label{addi3}
		c_{\lambda} u_{\lambda}(\delta)=-2\ln\delta+\frac{W(\delta)}{c^{2}_{\lambda}}
		+o_{\lambda}\Big(\frac{1}{c^{2}_{\lambda}}\Big).
	\end{equation}		
	Combining Proposition \ref{lem-7-31-2}, \eqref{31-4} and \eqref{addi3}, we derive that
	\begin{equation}\label{add-2-7-1}
		\lambda c_{\lambda}^2=1+\tilde{A}\lambda+o_{\lambda}(\lambda),
	\end{equation}
	which implies that
	\begin{equation}\nonumber
		c_\lambda= \lambda^{-\frac{1}{2}}+A\lambda^{\frac{1}{2}}
		+o_{\lambda}\big(\lambda^{\frac{1}{2}}\big),
	\end{equation}
	for some constant $A$.
	%where $A=w_0(\delta)-W(\delta)$ is a constant independent of $\delta.$
	Then we complete the proof of \eqref{lambda_gamm1a}.
\end{proof}

\begin{lemma}\label{addi}
	For any $\varphi(x)\in C_c^{\infty}(\mathbb{B}^2),$
	there holds that
	\begin{equation}\label{addi1}
		\lim_{\lambda\rightarrow 0}c^{4}_{\lambda}\int_{\mathbb{B}^2}\left(\lambda c_{\lambda}u_{\lambda}e^{u^{2}_{\lambda}}
		-\pi\delta_{0}-c_\lambda^{-2}4\pi G-c_\lambda^{-2}\pi\delta_0\right)\varphi dV_{\mathbb{B}^2}\leq C\|\varphi\|_{L^\infty(\mathbb{B}^2)},
	\end{equation}
	where $\delta_0$ is the Dirac function.
\end{lemma}	

\begin{proof}
	First, we can derive that
	\begin{equation}\label{8-1-1}
		\begin{aligned}
			{}&c^{4}_{\lambda}\int_{\mathbb{B}^2}\left(\lambda c_{\lambda}u_{\lambda}e^{u^{2}_{\lambda}}
			-\pi\delta_{0}-c_\lambda^{-2}4\pi G-c_\lambda^{-2}\pi \delta_0\right)\varphi dV_{\mathbb{B}^2}\\
			={}&c^{4}_{\lambda}\int_{B_{r_{\lambda}\widetilde{s_\lambda}}}(\lambda c_{\lambda}u_{\lambda}e^{u^{2}_{\lambda}}
			-\pi\delta_{0}-c_\lambda^{-2}4\pi G-c_\lambda^{-2}\pi \delta_0)\varphi dV_{\mathbb{B}^2}\\
			{}&+c^{4}_{\lambda}\int_{\mathbb{B}^2 \backslash B_{r_{\lambda}\widetilde{s_\lambda}}}(\lambda c_{\lambda}u_{\lambda}e^{u^{2}_{\lambda}}-c_\lambda^{-2}4\pi G)\varphi dV_{\mathbb{B}^2}\\
			:={}&c^{4}_{\lambda}J_{3}+c^{4}_{\lambda}J_{4}.
		\end{aligned}
	\end{equation}
	By \eqref{555}, \eqref{556} and the mean value theorem we have $\lim\limits_{\lambda\to0}c^{4}_{\lambda}J_{3}\leq C\|\varphi\|_{L^{\infty}(\mathbb{B}^2)}$. Since $\varphi\in C_c^\infty(\mathbb{B}^2)$, we may assume that the support of $\varphi$ is contained in $B_{\mathbb{B}^2}(0,R)$. Then, similar to \eqref{ass1}, by Lemmas \ref{lem-7-31-1} and \ref{ad} we can check that
	\begin{equation}\label{ass4}\nonumber
		\begin{split}
			{}&\lim_{\lambda\to0}c^{4}_{\lambda}J_{4}=\lim_{\lambda\to0}c^{4}_{\lambda}\int_{B_{\mathbb{B}^2}(0,R) \backslash B_{r_{\lambda}\widetilde{s_\lambda}}}(\lambda c_{\lambda}u_{\lambda}e^{u^{2}_{\lambda}}-c_\lambda^{-2}4\pi G)\varphi dV_{\mathbb{B}^2}\\
			={}&\lim_{\lambda\to0}\int_{B_{\mathbb{B}^2}(0,R) \backslash B_{r_{\lambda}\widetilde{s_{\lambda}}}}\Big(\lambda c_\lambda^2 \cdot c^2_\lambda(c_{\lambda}u_{\lambda}-4\pi G)e^{u^{2}_{\lambda}}\\
			{}&\quad\quad\quad\quad\quad\quad\quad\quad\quad+\lambda c_\lambda^2 4\pi G c^2_\lambda(e^{u_\lambda^2}-1)+c^2_\lambda(\lambda c_\lambda^2-1)4\pi G\Big)\varphi dV_{\mathbb{B}^2}\\
			\leq{}&\lim_{\lambda\to0}\int_{B_{\mathbb{B}^2}(0,R) \backslash B_{r_{\lambda}\widetilde{s_{\lambda}}}}\Big(\lambda c_\lambda^2 \cdot c^2_\lambda(c_{\lambda}u_{\lambda}-4\pi G)e^{u^{2}_{\lambda}}\\
			{}&\quad\quad\quad\quad\quad\quad\quad\quad\quad+\lambda c_\lambda^2 4\pi G c^2_\lambda u_\lambda^2 e^{u_\lambda^2}+c^2_\lambda(\lambda c_\lambda^2-1)4\pi G\Big)\varphi dV_{\mathbb{B}^2}\\
			=
			{}&\lim_{\lambda\to0}\int_{B_{\mathbb{B}^2}(0,R) \backslash B_{r_{\lambda}\widetilde{s_{\lambda}}}}
			\Big((W+o_{\lambda}(1))e^{u^{2}_{\lambda}}\\
			{}&\quad\quad\quad\quad\quad\quad\quad\quad\quad+4\pi G (4\pi G +o_\lambda(1))^2 e^{u_\lambda^2}
			+c^2_\lambda(A\lambda +o_{\lambda}(\lambda))4\pi G\Big)\varphi dV_{\mathbb{B}^2}\\
			\leq {}&
			C\|\varphi\|_{L^\infty(\mathbb{B}^2)}\\
			{}&\cdot\lim_{\lambda\to0}\Bigg\{\big(\int_{B_{\mathbb{B}^2}(0,R) \backslash B_{r_{\lambda}\widetilde{s_{\lambda}}}}(W+o_{\lambda}(1))^{p'}dV_{\mathbb{B}^2}\big)^{\frac{1}{p'}}
			\big(\int_{B_{\mathbb{B}^2}(0,R) \backslash B_{r_{\lambda}\widetilde{s_{\lambda}}}}e^{pu^{2}_{\lambda}}dV_{\mathbb{B}^2}\big)^{\frac{1}{p}}
			\\{}&\quad+\big(\int_{B_{\mathbb{B}^2}(0,R) \backslash B_{r_{\lambda}\widetilde{s_{\lambda}}}}[4\pi G(4\pi G+o_{\lambda}(1))^{2}]^{p'} dV_{\mathbb{B}^2}\big)^{\frac{1}{p'}}
			\big(\int_{B_{\mathbb{B}^2}(0,R) \backslash B_{r_{\lambda}\widetilde{s_{\lambda}}}}e^{pu^{2}_{\lambda}}dV_{\mathbb{B}^2}\big)^{\frac{1}{p}}
			\\{}&\quad+\big(\int_{B_{\mathbb{B}^2}(0,R) \backslash B_{r_{\lambda}\widetilde{s_{\lambda}}}}(A+o_{\lambda}(1))^{p'}dV_{\mathbb{B}^2}\big)^{\frac{1}{p'}}
			\big(\int_{B_{\mathbb{B}^2}(0,R) \backslash B_{r_{\lambda}\widetilde{s_{\lambda}}}}(4\pi G)^pdV_{\mathbb{B}^2}\big)^{\frac{1}{p}}\Bigg\}
			\\
			\leq {}&C\|\varphi\|_{L^\infty(\mathbb{B}^2)}
		\end{split}
	\end{equation}
	Combining the above two estimates we accomplish the proof.
\end{proof}

\begin{remark}\label{addi9}
	We would like to point out that
	to prove \eqref{addi1}, we split
	$\mathbb{B}^2$ into a small ball $B_{r_{\lambda}\widetilde{s_\lambda}}$
	(see \eqref{8-1-1}) and its exterior $\mathbb{B}^2\setminus B_{r_{\lambda}\widetilde{s_\lambda}}.$
	The choice of $\widetilde{s_\lambda}$ is very crucial. On the one hand, as this radius $r_{\lambda}\widetilde{s_\lambda}$ tends to zero, the interior estimates can be obtained via a refined blow-up analysis near the origin (involving $\delta_{0},G$ etc.), while the
	exterior integral can be easily established convergence by the truncated technique and Moser-Trudinger inequality, where we need to apply
	both \eqref{addi3} and \eqref{add-2-7-1}.
\end{remark}

The following Lemma is very important, as it provides a refined expansion that was absent in previous results obtained away from the blow-up point.
It plays a key role to prove Proposition \ref{add-lem6} later.
\begin{lemma}\label{ad-lem2}
	There holds
	\begin{equation}\label{7-30-10}
		c_{\lambda} u_{\lambda}(x)=4\pi G(x)+\frac{W(x)}{c_\lambda^2}+\frac{\widetilde{W}(x)}{c^{4}_{\lambda}}
		+o_{\lambda}\Big(\frac{1}{c^{4}_{\lambda}}\Big),\,\,\,x\in \partial B_{\delta},
	\end{equation}
	with some functions $G$, $W$ and $\widetilde{W}.$
\end{lemma}

\begin{proof}
	Obviously, $c_\lambda^2[c_\lambda^2(c_\lambda u_\lambda-4\pi G)-W]$ satisfies equation
	\begin{equation*}
		\begin{aligned}
			{}&-\Delta_{\mathbb{B}^2}\left(c_\lambda^2[c_\lambda^2(c_\lambda u_\lambda-4\pi G)-W]\right)\\
			={}&c^{4}_{\lambda}\left(\lambda c_{\lambda}u_{\lambda}e^{u^{2}_{\lambda}}
			-\pi\delta_{0}-c_\lambda^{-2}4\pi G-c_\lambda^{-2}\pi \delta_0\right)=:\widetilde{\widetilde{f_\lambda}}.
		\end{aligned}
	\end{equation*}
	Since $\widetilde{\widetilde{f_\lambda}}$ can be seen as the linear functional on $L^{\infty}(\mathbb{B}^2)$, according to the principle of uniform boundness, it follows from Lemma \ref{addi} that $\widetilde{\widetilde{f_\lambda}}$ is bounded in $L^1_{\rm loc}(\mathbb{B}^2)$. Hence one can similarly derive that $c_\lambda^2[c_\lambda^2(c_\lambda u_\lambda-4\pi G)-W]$ is bounded in $W^{1,q}_{\rm loc}(\mathbb{B}^2)$. From the proof of Lemma \ref{addi}, we see that $\widetilde{\widetilde{f_\lambda}}$ is bounded in $L^{r}_{\rm loc}(\mathbb{B}^2 \backslash \{0\})$. By standard elliptic estimates we get
	$c_\lambda^2[c_\lambda^2(c_\lambda u_\lambda-4\pi G)-W]$ is locally bounded in $C^2_{\rm loc}(\mathbb{B}^2 \backslash \{0\})$. Hence there exists $\widetilde{W}(x)\in C^2_{\rm loc}(\mathbb{B}^2 \backslash \{0\})$ such that $c_\lambda^2[c_\lambda^2(c_\lambda u_\lambda-4\pi G)-W]$ strongly converges to $\widetilde{W}(x)$ in
	$C^{1,\gamma}_{\rm loc}(\mathbb{B}^2 \backslash \{0\})$. Then we accomplish the proof of Lemma \ref{ad-lem2}.
\end{proof}

Now we will prove Proposition \ref{add-lem6}.
\begin{proof}[\textbf{Proof of Proposition \ref{add-lem6}}]
	By Lemma \ref{ad-lem2}, we have
	\begin{equation}\label{7-30-1}
		c_{\lambda}u_{\lambda}\left(\delta\right)= -2\ln\delta+\frac{W(\delta)}{c_\lambda^2}+\frac{\widetilde{W}(\delta)}{c^{4}_{\lambda}}+o_{\lambda}
		\Big(\frac{1}{c^{4}_{\lambda}}\Big).
	\end{equation}
	It follows from \eqref{31-4}, \eqref{lambda_gamm1a} and \eqref{7-30-1} that
	$$
	\lambda c_{\lambda}^2=1+\tilde{A}\lambda+\tilde{B}\lambda^2+o_{\lambda}(\lambda^2),
	$$
	which implies that \eqref{lambda_gamma} holds for some constants $A, B$.
\end{proof}

\section{Uniqueness of positive solutions}\label{s6}
Since it follows from Proposition \ref{pro} that $u_{\lambda}$ is radial up to some M\"{o}bius transformation, in order to prove Theorem \ref{unique},
by the uniqueness of Cauchy initial value problem for ODEs,
we only need to prove
the following result.
\begin{proposition}\label{thuniq}
	Let $u_\lambda^{(1)}$ and $u_\lambda^{(2)}$ be two positive solutions to equation
	\begin{equation}\label{lam}
		\left\{
		\begin{aligned}
			&-\Delta u=\lambda u e^{u^2}\left(\frac{2}{1-|x|^2}\right)^2,&x&\in B_1,\\
			%		&u_k>0,&x&\in B_1,\\
			&u=0,&x&\in\partial B_1.\\
		\end{aligned}
		\right.
	\end{equation}
	Then  there exists  $\lambda_0>0$ such that
	\[u_\lambda^{(1)}(x)\equiv u_\lambda^{(2)}(x),   ~\mbox{for any}~~ \lambda\in (0,\lambda_0),   x\in B_{\delta r_{\lambda}},\]
	where $\delta>0$ is a small fixed constant.
\end{proposition}
To prove Proposition \ref{thuniq}, we mainly use a Pohozaev identity of the positive solution $u_{\lambda}$ from scaling and
a contradiction argument.

Let $u_\lambda^{(1)}$ and $u_\lambda^{(2)}$ be two positive solutions to problem \eqref{lam}.
Let us assume that they concentrate at the same point $0$. We denote
\begin{equation*}
	c_\lambda^{(l)}:=u^{(l)}_{\lambda}(0)=\max_{\overline{B_{\delta}}}u^{(l)}_{\lambda}(x)  ~\mbox{for some small fixed}~\delta>0
\end{equation*}
and
$$
r^{(l)}_{\lambda}:=\Big(\lambda \big(c_\lambda^{(l)}\big)^2e^{(c_\lambda^{(l)})^2}\Big)^{-1/2}, ~~\mbox{for} ~~l=1,2.
$$

Now we have a basic result on $\frac{r^{(1)}_{\lambda}}{r^{(2)}_{\lambda}}.$
\begin{lemma}\label{theta12}
	It holds that
	\begin{equation*}\label{llsb}
		\frac{r^{(1)}_{\lambda}}{r^{(2)}_{\lambda}}=1+o_{\lambda}\big(\lambda\big).
	\end{equation*}
	%where $o_{\lambda}\big(1\big)$ denotes the quantity which goes to zero as $\lambda$ goes to zero.
\end{lemma}

\begin{proof}
	Firstly, we have
	\begin{equation}\label{5.1-1}
		\frac{r^{(1)}_{\lambda}}{r^{(2)}_{\lambda}}=
		\frac{c^{(2)}_\lambda}{c^{(1)}_\lambda} e^{-\frac{1}{2}\big((c^{(1)}_\lambda)^2-(c^{(2)}_\lambda)^2\big)}.
	\end{equation}
	From Proposition \ref{add-lem6}, we find
	\begin{equation}\label{5.1-2}
		\frac{c^{(2)}_\lambda}{c^{(1)}_\lambda} =1+o_{\lambda}\big(\lambda^{2}\big),
	\end{equation}
	and
	\begin{equation}\label{5.1-3}
		\begin{split}
			\big(c^{(1)}_\lambda\big)^2-\big(c^{(2)}_\lambda\big)^2
			&= \big(c^{(1)}_\lambda+c^{(2)}_\lambda\big)\cdot \big(c^{(1)}_\lambda-c^{(2)}_\lambda\big)
			=o_{\lambda}\big(\lambda\big).
		\end{split}
	\end{equation}
	Hence we can deduce from \eqref{5.1-1}, \eqref{5.1-2} and \eqref{5.1-3} that
	\begin{equation*}
		\frac{r^{(1)}_{\lambda}}{r^{(2)}_{\lambda}}= \Big(1+o_{\lambda}\big(\lambda^{2}\big)\Big)  \big(1+o_{\lambda}(\lambda)\big)=1+o_{\lambda}\big(\lambda\big).
	\end{equation*}
\end{proof}

In the following, we will consider the same quadratic form already introduced in \eqref{P},
\begin{equation}\label{p1uv}
	\begin{split}
		P^{(1)}(u,v):={}&P(\delta r^{(1)}_{\lambda},u,v)\\={}&- 2\delta r^{(1)}_{\lambda}\int_{\partial B_{\delta r^{(1)}_{\lambda}}} \big\langle \nabla u ,\nu\big\rangle \big\langle \nabla v,\nu\big\rangle  d\sigma + \delta r^{(1)}_{\lambda}  \int_{\partial B_{\delta r^{(1)}_{\lambda}}} \big\langle \nabla u , \nabla v \big\rangle  d\sigma.
	\end{split}
\end{equation}

Noting that if $u$ and $v$ are harmonic in $ B_d\backslash \{0\}$, then by Lemma \ref{indep_d}, we know that $P^{(1)}(u,v)$
is independent of $\delta r^{(1)}_{\lambda}\in (0,d]$.

Now if $u_{\lambda}^{(1)}\not \equiv u_{\lambda}^{(2)}$ in $B_{\delta r_{\lambda}}$, we set
\begin{equation}\label{eta-def}
	\xi_{\lambda}:=\frac{u_{\lambda}^{(1)}-u_{\lambda}^{(2)}}
	{\|u_{\lambda}^{(1)}-u_{\lambda}^{(2)}\|_{L^{\infty}(B_{\delta r_{\lambda}})}}.
\end{equation}
Then $\xi_{\lambda}$ satisfies $\|\xi_{\lambda}\|_{L^{\infty}(B_{\delta r_{\lambda}})}=1$ and
\begin{equation}\label{eta-equa}
	\begin{aligned}
	- \Delta \xi_{\lambda}={}&-\frac{\Delta u_{\lambda}^{(1)}-\Delta u_{\lambda}^{(2)}}
	{\|u_{\lambda}^{(1)}-u_{\lambda}^{(2)}\|_{L^{\infty}(B_{\delta r_{\lambda}})}}\\
	={}& \frac{\lambda \Big(u_{\lambda}^{(1)}e^{(u_{\lambda}^{(1)})^2} -  u_{\lambda}^{(2)}e^{(u_{\lambda}^{(2)})^2}\Big)}
	{\|u_{\lambda}^{(1)}-u_{\lambda}^{(2)}\|_{L^{\infty}(B_{\delta r_{\lambda}})}}\frac{4}{(1-|x|^{2})^{2}}
	=:\frac{4}{(1-|x|^{2})^{2}}E_{\lambda}\xi_{\lambda},
	\end{aligned}
\end{equation}
where
\begin{equation}\label{Dlambda-def}
	E_{\lambda}(x):=\lambda e^{(u_{\lambda}^{(1)})^2}+2 \lambda  u_{\lambda}^{(2)}  \displaystyle\int_{0}^1
	F_t(x) e^{ (F_t(x))^2}  dt,
\end{equation}
with $F_t(x):=tu_{\lambda}^{(1)}(x)+(1-t)u_{\lambda}^{(2)}(x)$.

\begin{lemma}\label{theta-D}
	There holds
	\begin{equation}\label{Dlambda-equa}
		\frac{4\big(r^{(1)}_{\lambda}\big)^{2}
			E_{\lambda}\big(r^{(1)}_{\lambda}x\big)}{(1-|r^{(1)}_{\lambda}x|^{2})^{2}}  =
		8e^{2\eta_{0}(x)} \big(1+ O(\lambda) \big),
	\end{equation}
	uniformly on compact sets.
	Particularly, we have
	\begin{equation}\label{Dlambda-lim}
		\frac{4\big(r^{(1)}_{\lambda}\big)^{2}
			E_{\lambda}\big(r^{(1)}_{\lambda}x\big)}{(1-|r^{(1)}_{\lambda}x|^{2})^{2}}  \rightarrow
		8e^{2\eta_{0}(x)} ~  ~\mbox{in}~C_{\rm loc}^0\big(\R^2\big).
	\end{equation}
\end{lemma}

\begin{proof}
	Firstly, we have
	\begin{equation}\label{5.5-1}
		\begin{split}
			{}&\big(r^{(1)}_{\lambda}\big)^{2}
			E_{\lambda}\big(r^{(1)}_{\lambda}x\big)\frac{4}{(1-|r^{(1)}_{\lambda}x|^{2})^{2}}   \\={}&
			\frac{8 \lambda\big(r_\lambda^{(1)}\big)^2 u_{\lambda}^{(2)} \big(r^{(1)}_{\lambda}x\big )}{(1-|r^{(1)}_{\lambda}x|^{2})^{2}}  \displaystyle\int_{0}^1
			F_t\big(r^{(1)}_{\lambda}x\big )e^{ \big(F_t(r^{(1)}_{\lambda}x)\big)^2}dt\\
			{}&+  \frac{4\lambda}{(1-|r^{(1)}_{\lambda}x|^{2})^{2}} \big(r^{(1)}_{\lambda}\big)^{2} e^{\big(u_{\lambda}^{(1)}(r^{(1)}_{\lambda}x )\big)^2}
		\end{split}
	\end{equation}
	and from \eqref{eq5} we have
	\begin{equation}\label{5.5-2}
		\begin{split}
			F_t\big(r^{(1)}_{\lambda}x\big ) &= tu_\lambda^{(1)}
			\big(r^{(1)}_{\lambda}x\big )+(1-t) u_\lambda^{(2)}\big(r^{(1)}_{\lambda}x\big )\\
			&= t\Big(c^{(1)}_\lambda+\frac{\eta^{(1)}_\lambda(x)}{c^{(1)}_\lambda}\Big)
			+\big(1-t\big) \bigg(c^{(2)}_\lambda+\frac{\eta^{(2)}_\lambda\Big(\frac{ r^{(1)}_\lambda }
				{r^{(2)}_\lambda}x\Big)}
			{c^{(2)}_\lambda}\bigg) \\
			&= \Big(c^{(1)}_\lambda+\frac{\eta^{(1)}_\lambda(x)}{c^{(1)}_\lambda}\Big)+\big(1-t\big) f_\lambda(x),
		\end{split}
	\end{equation}
	where
	$$
	f_\lambda(x):=c^{(2)}_\lambda-c^{(1)}_\lambda+\frac{\eta^{(2)}_\lambda\Big(\frac{ r^{(1)}_\lambda }
		{r^{(2)}_\lambda}x\Big)}{c^{(2)}_\lambda}
	-\frac{\eta^{(1)}_\lambda\big(x\big)} {c^{(1)}_\lambda}.
	$$
	Moreover, by  \eqref{lambda_gamma} and Lemma \ref{theta12}, direct computation gives
	\begin{equation}\label{6-8-1}
		\begin{split}
			f_\lambda(x)
			&= o_{\lambda}\big(\lambda^{\frac{3}{2}}\big)+ \frac{\eta_\lambda^{(2)}(x)+o_{\lambda}(\lambda)}
			{c_\lambda^{(2)}} -\frac{\eta^{(1)}_\lambda\big(x\big)} {c^{(1)}_\lambda}  = o_{\lambda}\big(\lambda^{\frac{3}{2}}\big)+O\left(\frac{\big|\eta^{(2)}_\lambda(x)-\eta^{(1)}_\lambda(x)\big|}
			{c^{(1)}_\lambda}\right).
		\end{split}
	\end{equation}
	Since it follows
	from Lemma \ref{lem2}
	that $w_0^{(1)}$ and $w_0^{(2)}$ are the solutions of
	\[
	-\Delta u - 8e^{2\eta_{0}}u = 4e^{2\eta_{0}}\big(\eta_{0}^2+\eta_{0}\big),\,\,\,x\in \R^{2},
	\]
	by Lemma \ref{lem3.1} and the fact that $w_0^{(l)}(x)\ (l=1,2)$ are radial functions, we have
	\[
	w_0^{(1)}(x)-w_0^{(2)}(x)=\alpha_0 \frac{1-|x|^2}{1+|x|^2}.
	\]
	
	Noting that
	\begin{equation*}
		\begin{split}
			w_\lambda^{(1)}(0)-w_\lambda^{(2)}(0) =& \big(c_\lambda^{(1)}\big)^2 \big(\eta_\lambda^{(1)}(x)-\eta_{0}(x)\big)\Big|_{x=0}
			-\big(c_\lambda^{(2)}\big)^2 \big(\eta_\lambda^{(2)}(x)-\eta_{0}(x)\big)\Big|_{x=0} \\
			=& \big(c_\lambda^{(1)}\big)^3 \Big(u_\lambda^{(1)}\big(r_\lambda^{(1)}x\big)
			-c_\lambda^{(1)}\Big)\bigg|_{x=0} - \big(c_\lambda^{(1)}\big)^2 \eta_{0}(0) \\
			& -\big(c_\lambda^{(2)}\big)^3 \Big(u_\lambda^{(2)}\big(r_\lambda^{(2)}x\big)
			-c_\lambda^{(2)}\Big)\bigg|_{x=0} +\big(c_\lambda^{(2)}\big)^2 \eta_{0}(0) = 0,
		\end{split}
	\end{equation*}
	then Lemma \ref{lem2} gives that
	$w_0^{(1)}(0)-w_0^{(2)}(0)=0$,
	which implies that $\alpha_0=0$.
	Therefore, we derive $w_0^{(1)}(x)-w_0^{(2)}(x)=0$. Thus,
	it follows form \eqref{6-8-1},
	Lemma \ref{lem1}  and Lemma \ref{lem2} that
	\begin{equation}\label{5.5-3}
		\begin{split}
			f_\lambda(x) &= o_{\lambda}\big(\lambda^{\frac{3}{2}}\big)+O\Bigg(\frac{\big|\eta^{(2)}_\lambda(x)-\eta^{(1)}_\lambda(x)\big|}
			{c^{(1)}_\lambda}\Bigg)\\
			&= o_{\lambda}\big(\lambda^{\frac{3}{2}}\big)+O\Bigg(\frac{\big|w^{(2)}_0(x)-w^{(1)}_0(x)\big|}
			{\big(c^{(1)}_\lambda\big)^3}\Bigg) = o_{\lambda}\big(\lambda^{\frac{3}{2}}\big).
		\end{split}
	\end{equation}
	Hence we can deduce from \eqref{lambda_gamma}, \eqref{5.5-2} and \eqref{5.5-3} that
	\begin{equation*}
		\begin{split}
			e^{\big( F_t (r^{(1)}_{\lambda}x) \big)^2}
			= {}&e^{\big(c^{(1)}_\lambda+\frac{\eta^{(1)}_\lambda(x)}{c^{(1)}_\lambda}\big)^2
				+(1-t)^2 f^2_\lambda(x)+2\big(c^{(1)}_\lambda+\frac{\eta^{(1)}_\lambda(x)}{c^{(1)}_\lambda}\big) (1-t)f_\lambda(x)}\\
			={}& e^{\big(c^{(1)}_\lambda+\frac{\eta^{(1)}_\lambda(x)}{c^{(1)}_\lambda}\big)^2} \big(1+o_{\lambda}(\lambda)\big),
		\end{split}
	\end{equation*}
	which implies
	\begin{equation}\label{5.5-4}
		\begin{split}
			{}& \big(r_\lambda^{(1)}\big)^2 \int^1_0 F_t\big(r^{(1)}_{\lambda}x\big ) e^{\big( F_t(r^{(1)}_{\lambda}x) \big)^2}  dt \\
			={}& \big(r_\lambda^{(1)}\big)^2 e^{\big(c^{(1)}_\lambda+\frac{\eta^{(1)}_\lambda(x)}{c^{(1)}_\lambda}\big)^2} \big(1+o_{\lambda}(\lambda)\big)  \int^1_0 \Bigg[\Big(c^{(1)}_\lambda+\frac{\eta^{(1)}_\lambda(x)}{c^{(1)}_\lambda}\Big)
			+\big(1-t\big) f_\lambda(x) \Bigg] dt  \\
			={}& \frac{1}{\lambda c_\lambda^{(1)}} e^{2\eta^{(1)}_\lambda(x)+\frac{(\eta^{(1)}_\lambda(x))^2}{( c^{(1)}_\lambda)^2} }\big(1+o_{\lambda}(\lambda)\big)\Big(1+\frac{ \eta^{(1)}_\lambda(x) }{\big(c^{(1)}_\lambda\big)^2}
			+o_{\lambda}\big(\lambda^2\big)\Big).
		\end{split}
	\end{equation}
	Letting  $t=0$ in \eqref{5.5-2}, we have
	\begin{equation}\label{5.5-5}
		\begin{split}
			u_{\lambda}^{(2)}\big(r^{(1)}_{\lambda}x\big )
			=c^{(1)}_\lambda+\frac{\eta^{(1)}_\lambda\big(x\big)}{c^{(1)}_\lambda}+f_\lambda(x).
		\end{split}
	\end{equation}
	Also, noting that $r_{\lambda}=O(e^{-\frac{\alpha^{2}c^{2}_{\lambda}}{2}})=O(\lambda)$, we have
	\begin{equation}\label{5.5-add1}
		\begin{split}
			\frac{1}{(1-|r_{\lambda}x|^{2})^{2}}=
			1+2|r_{\lambda}x|^{2}+O(|r_{\lambda}x|^{4})
			=1+O(r^{2}_{\lambda})=1+O\big(\lambda\big).
		\end{split}
	\end{equation}
	Hence, by Lemma \ref{lem2}, \eqref{5.5-3}, \eqref{5.5-4}, \eqref{5.5-5} and \eqref{5.5-add1}, we get
	\begin{align}\label{5.5-6}
		{}&\frac{8 \lambda \big(r_\lambda^{(1)}\big)^2 u_{\lambda}^{(2)}\big(r^{(1)}_{\lambda}x\big ) }{(1-|r_{\lambda}x|^{2})^{2}} \displaystyle\int_{0}^1 F_t\big(r^{(1)}_{\lambda}x\big )
		e^{ \big(F_t(r^{(1)}_{\lambda}x)\big)^2}dt \\
		={}& \frac{8}{c_\lambda^{(1)}} e^{2\eta^{(1)}_\lambda(x)+\frac{(\eta^{(1)}_\lambda(x))^2}{( c^{(1)}_\lambda)^2} }
		\Big(1+o_{\lambda}\big(\lambda\big)\Big)\big(1+O(\lambda)\big)\Big(1+\frac{ \eta^{(1)}_\lambda(x) }{\big( c^{(1)}_\lambda\big)^2}+o_{\lambda}\big(\lambda^{2}
		\big)\Big)  \notag\\
		{}&\cdot \Big(c^{(1)}_\lambda+\frac{\eta^{(1)}_\lambda\big(x\big)}{c^{(1)}_\lambda}+f_\lambda(x)\Big) \notag\\
		={}& 8e^{2\eta^{(1)}_\lambda(x)+\frac{(\eta^{(1)}_\lambda(x))^2}{( c^{(1)}_\lambda)^2} }
		\Big(1+o_{\lambda}\big(\lambda\big)\Big)\big(1+O(\lambda)\big)  \Big(1+\frac{ \eta^{(1)}_\lambda(x) }{\big( c^{(1)}_\lambda\big)^2}
		+o_{\lambda}\big(\lambda^{2}
		\big)\Big) \notag\\
		{}&\cdot  \Big(1+\frac{ \eta^{(1)}_\lambda(x) }{\big( c^{(1)}_\lambda\big)^2}+o_{\lambda}\big(\lambda^{2}
		\big)\Big) \notag\\
		={}& 8e^{2\eta_{0}(x)}\Big(1+O\big(\lambda\big)\Big) \big(1+ O(\lambda)\big)  \big(1+o_{\lambda}(\lambda)\big)  \Big(1+\frac{ \eta^{(1)}_\lambda(x) }{\big( c^{(1)}_\lambda\big)^2}+o_{\lambda}\big(\lambda^{2}\big)\Big)^2 \notag\\
		={}& 8e^{2\eta_{0}(x)} \big(1+ O(\lambda ) \big). \notag
	\end{align}
	By \eqref{5.5-add1}, Lemmas \ref{lem1} and \ref{lem-7-31-2}, we can find that
	\begin{equation}\label{5.5-7}
		\begin{split}
			\frac{4\lambda \big(r^{(1)}_{\lambda}\big)^{2} e^{\big(u_{\lambda}^{(1)}(r^{(1)}_{\lambda}x)
					\big)^2}}{(1-|r_{\lambda}x|^{2})^{2}}
			&= \frac{4}{\big(c_\lambda^{(1)}\big)^2} e^{ 2\eta_\lambda^{(1)}(x)+\frac{(\eta_\lambda^{(1)}(x))^2}
				{(c_\lambda^{(1)})^2} }\Big(1+O\big(\lambda\big)\Big)\\
			&= O\bigg(\frac{1}{\big(c_\lambda^{(1)}\big)^2} e^{2\eta_{0}(x)} \bigg)=O\big(\lambda e^{2\eta_{0}(x) } \big).
		\end{split}
	\end{equation}
	Then, substituting \eqref{5.5-6} and \eqref{5.5-7} into \eqref{5.5-1}, we derive \eqref{Dlambda-equa}. Also \eqref{Dlambda-lim} can be obtained from the above computations and Lemma \ref{lem1}.
\end{proof}

Now applying the blow up analysis, we establish an estimate on $\xi_\lambda$.
\begin{proposition}\label{prop_tildeeta}
	Let $\widetilde{\xi}_{\lambda}(x):=\xi_{\lambda}\big(r^{(1)}_{\lambda}x\big)$, where $\xi_{\lambda}$ is defined in \eqref{eta-def}. Then by taking
	a subsequence if necessary, we have
	\begin{equation}\label{tildeeta-lim}
		\widetilde{\xi}_{\lambda}(x) \to \widetilde{\alpha}_{0} \frac{1-|x|^2}{1+|x|^2}
		~~~\mbox{in}~~~C^1_{\rm loc}\big(\R^2\big),~\mbox{as}~\lambda \to 0,
	\end{equation}
	where $\widetilde{\alpha}_{0}$ is some constant.
\end{proposition}

\begin{proof}
	Since $|\widetilde{\xi}_{\lambda}|\leq 1$,
	by Lemma \ref{theta-D} and the standard elliptic regularity theory, we find that
	\[
	\widetilde{\xi}_{\lambda}(x) \in C^{1,\gamma}\big(B_R(0)\big)~\mbox{and}~
	\|\widetilde{\xi}_{\lambda}\|_{C^{1,\gamma}(B_R(0))} \leq C,
	\]
	for any fixed large $R$ and some $\gamma \in (0,1)$, where $C$ is independent of $\lambda$. Then there exists a subsequence (still denoted by $\widetilde{\xi}_{\lambda}$) such that
	$$\widetilde{\xi}_{\lambda}(x)\to \xi_0(x)~~\mbox{in}~C^1\big(B_R(0)\big).$$
	By \eqref{eta-equa}, it is easy to check that $\widetilde{\xi}_{\lambda}$ satisfies
	\begin{equation*}
		-\Delta \widetilde{\xi}_{\lambda} =\frac{4\big(r^{(1)}_{\lambda}\big)^{2}E_{\lambda}\big(r^{(1)}_{\lambda}x\big) }{(1-|r^{(1)}_{\lambda}x|^{2})^{2}} \widetilde{\xi}_{\lambda}    ~\mbox{in}~B_{\delta}(0).
	\end{equation*}
	Then by Lemma \ref{theta-D} and \eqref{5.5-add1},
	letting $\lambda \to 0$, we find that $\xi_0$ satisfies
	\begin{equation*}
		-\Delta \xi_0 = 8e^{2\eta_{0}} \xi_0 \,\,\,\,    \mbox{in} \,\,\,   \R^2.
	\end{equation*}
	Since $u_{\lambda}^{(l)}(x)\ (l=1,2)$ is radial, we know that $\xi_{\lambda}(x)$
	is radial which implies that $\widetilde{\xi}_{\lambda} $ is also radial.
	Hence, by Lemma \ref{lem3.1}, we have
	\[
	\xi_0(x)= \widetilde{\alpha}_{0} \frac{1-|x|^2}{1+|x|^2},
	\]
	where $\widetilde{\alpha}_{0}$ is some constant.
\end{proof}

\begin{proposition}
	For any small fixed constant $\delta>0$, we have the following local Pohozaev identity
	\begin{equation}\label{p1_ueta}
		\begin{split}
			{}&P^{(1)}\big(u_\lambda^{(1)}+u_\lambda^{(2)},\xi_{\lambda}\big)\\
			={}& 8\delta r^{(1)}_{\lambda}\int_{\partial  B_{\delta r^{(1)}_{\lambda}}}
			\frac{\widetilde{E}_{\lambda}\xi_{\lambda}}{(1-|x|^{2})^{2}} d\sigma
			-16\int_{  B_{\delta r^{(1)}_{\lambda}}} \frac{\widetilde{E}_{\lambda}\xi_{\lambda}}{(1-|x|^{2})^{3}} dx
			-16\int_{  B_{\delta r^{(1)}_{\lambda}}} \frac{\widetilde{E}_{\lambda}\xi_{\lambda}x\cdot x}{(1-|x|^{2})^{3}} dx,
		\end{split}
	\end{equation}
	where $P^{(1)}$
	is the quadratic form in \eqref{p1uv},
	$\nu=\big(\nu_{1},\nu_2\big)$ is the unit  outward normal of $\partial  B_{\delta r^{(1)}_{\lambda}}$ and
	\begin{equation}\label{def_tildeE}
		\widetilde{E}_{\lambda}(x):= \lambda  \displaystyle\int_{0}^1
		\Big(tu_{\lambda}^{(1)}(x)+(1-t)u_{\lambda}^{(2)}(x)\Big) e^{\big(tu_{\lambda}^{(1)}(x)+(1-t)u_{\lambda}^{(2)}(x)\big)^2}dt.
	\end{equation}
\end{proposition}

\begin{proof}
	By \eqref{puu}, we have
	\begin{equation*}
		\begin{split}
			{}& P^{(1)}\big(u_\lambda^{(1)}+u_\lambda^{(2)},\xi_{\lambda}\big) \\
			={}& \frac{1}{\|u_\lambda^{(1)}-u_\lambda^{(2)}\|_{L^{\infty}(B_{\delta r_{\lambda}})}} \left(P^{(1)}\big(u_\lambda^{(1)},u_\lambda^{(1)}\big) - P^{(1)}\big(u_\lambda^{(2)},u_\lambda^{(2)}\big)\right) \\
			={}& \frac{1}{\|u_\lambda^{(1)}-u_\lambda^{(2)}\|_{L^{\infty}(B_{\delta r_{\lambda}})}}
			\Bigg[ 4\delta r^{(1)}_{\lambda}\lambda  \int_{\partial B_{\delta r^{(1)}_{\lambda}}}
			\frac{( e^{(u^{(1)}_\lambda)^2} - e^{(u^{(2)}_\lambda)^2} )}{(1-|x|^{2})^{2}}  d\sigma
			\\{}&- 8\lambda \int_{B_{\delta r^{(1)}_{\lambda}}}  \frac{ e^{(u^{(1)}_\lambda)^2}-e^{(u^{(2)}_\lambda)^2} }
			{(1-|x|^{2})^{3}}  dx
			- 8\lambda \int_{B_{\delta r^{(1)}_{\lambda}}}  \frac{( e^{(u^{(1)}_\lambda)^2}-e^{(u^{(2)}_\lambda)^2} )x\cdot x}
			{(1-|x|^{2})^{3}}  dx\Bigg] \\
			={}& 8\delta r^{(1)}_{\lambda} \int_{\partial B_{\delta r^{(1)}_{\lambda}}} \frac{\widetilde{E}_{\lambda}\xi_{\lambda} }{(1-|x|^{2})^{2}} d\sigma
			-16\int_{B_{\delta r^{(1)}_{\lambda}}} \frac{\widetilde{E}_{\lambda}\xi_{\lambda}}{(1-|x|^{2})^{3}}  dx-16\int_{B_{\delta r^{(1)}_{\lambda}}} \frac{\widetilde{E}_{\lambda}\xi_{\lambda} x\cdot x}{(1-|x|^{2})^{3}}  dx,
		\end{split}
	\end{equation*}
	which completes the proof.
\end{proof}

\begin{proposition}\label{prop-A}
	Let $\widetilde{\alpha}_{0}$ be the constant in \eqref{tildeeta-lim}. Then
	$\widetilde{\alpha}_{0}=0.$
\end{proposition}

\begin{proof}
	First, by Lemma \ref{lem2}, \eqref{tildeeta-lim} and Lemma \ref{add-lem6}, we have
	\begin{equation}\label{5.9-add4}
		\begin{split}
			{}&LHS~of ~\eqref{p1_ueta}\\
			={}&-2\delta r_{\lambda}^{(1)}\int_{\partial B_{\delta r^{(1)}_\lambda }}
			\langle \nabla (u_{\lambda}^{(1)}+u_{\lambda}^{(2)}), \nu\rangle \langle \nabla \xi_{\lambda},\nu\rangle d\sigma\\
			{}&+\delta r_{\lambda}^{(1)}\int_{\partial B_{\delta r^{(1)}_\lambda }}\langle \nabla (u_{\lambda}^{(1)}+u_{\lambda}^{(2)}), \nabla \xi_{\lambda}\rangle  d\sigma
			\\
			={}&-2\delta\int_{\partial B_{\delta }}
			\langle \nabla [(u_{\lambda}^{(1)}+u_{\lambda}^{(2)})(r_{\lambda}^{(1)}x)], \nu\rangle \langle \nabla \widetilde{\xi}_{\lambda},\nu\rangle d\sigma\\
			{}&+\delta \int_{\partial B_{\delta  }}\langle \nabla [(u_{\lambda}^{(1)}+u_{\lambda}^{(2)})(r_{\lambda}^{(1)}x)], \nabla \widetilde{\xi}_{\lambda}\rangle  d\sigma\\
			={}&-\delta \int_{\partial B_{\delta  }}\langle \nabla [(u_{\lambda}^{(1)}+u_{\lambda}^{(2)})(r_{\lambda}^{(1)}x)], \nabla \widetilde{\xi}_{\lambda}\rangle  d\sigma\\
			={}&-\delta \int_{\partial B_{\delta  }}
			\Big(\big(\frac{1}{c^{(1)}_\lambda}+\frac{1}{c^{(2)}_\lambda}\big)\nabla \eta_{0}+\big(\frac{1}{(c^{(1)}_\lambda)^{3}}
			+\frac{1}{(c^{(2)}_\lambda)^{3}}\big)\nabla w_{0}+\frac{\nabla z^{(1)}_{\lambda}}{(c^{(1)}_\lambda)^{5}}+\frac{\nabla z^{(2)}_{\lambda}}{(c^{(2)}_\lambda)^{5}}\Big)\\
			{}&\cdot\nabla \Big(\widetilde{\alpha}_{0}\frac{1-|x|^{2}}{1+|x|^{2}}+o_{\lambda}(1)\Big)   d\sigma\\
			={}&-\frac{32\pi \delta^{4}}{(1+\delta^{2})^{3}}\frac{\widetilde{\alpha}_{0}
			}{c^{(1)}_\lambda}
			-\widetilde{A}_{1}\frac{\widetilde{\alpha}_{0}
			}{(c^{(1)}_\lambda)^{3}}+o_{\lambda}\Big(\frac{1}{(c^{(1)}_\lambda)^{3}}\Big),
		\end{split}
	\end{equation}
	where $$\widetilde{A}_{1}=2\delta \int_{\partial B_{\delta}}\nabla w_{0}\cdot\nabla \frac{1-|x|^{2}}{1+|x|^{2}}d\sigma.$$
	Similar to the proof of \eqref{5.5-4}, we have
	\begin{equation}\label{tilde_D-equa}
		\begin{split}
			\big(r^{(1)}_\lambda\big)^2 \widetilde{E}_\lambda\big(r^{(1)}_\lambda x \big)
			&= \lambda \big(r^{(1)}_\lambda\big)^2 \int_{0}^{1} F_t\big(r^{(1)}_{\lambda}x\big)
			e^{ \big(F_t(r^{(1)}_{\lambda}x)\big)^2 }  dt \\
			&= \frac{1}{c_\lambda^{(1)}} e^{2\eta^{(1)}_\lambda(x)+\frac{(\eta^{(1)}_\lambda(x))^2}{( c^{(1)}_\lambda)^2}}
			\big(1+o_{\lambda}(\lambda)\big)\Big(1+\frac{ \eta^{(1)}_\lambda(x) }{\big( c^{(1)}_\lambda\big)^2}
			+o_{\lambda}\big(\lambda^{2}\big)\Big) \\
			&= \frac{1}{c_\lambda^{(1)}} e^{2\eta_{0}(x)}\Big(1+\frac{\eta^{2}_{0}+\eta_{0}+2w_{0}}{\big( c^{(1)}_\lambda\big)^2}
			+o_{\lambda}\big(\frac{1}{\big( c^{(1)}_\lambda\big)^2}\big)\Big)
		\end{split}
	\end{equation}
	uniformly on compact sets. Applying \eqref{tilde_D-equa}, \eqref{5.5-add1} and Proposition \ref{prop_tildeeta}, we have
	\begin{equation}\label{5.9-add6}
		\begin{split}
			{}&8\delta r_{\lambda}^{(1)}\int_{\partial B_{\delta r^{(1)}_\lambda }}
			\frac{\widetilde{E}_{\lambda}(x)\xi_{\lambda}(x)}{(1-|x|^{2})^{2}} d\sigma\\
			={}&8\delta (r_{\lambda}^{(1)})^{2}\int_{\partial B_{\delta  }}
			\frac{\widetilde{E}_{\lambda}(r_{\lambda}^{(1)}x)\widetilde{\xi}_{\lambda}(x)}{(1-|r_{\lambda}^{(1)}x|^{2})^{2}} d\sigma\\
			={}&8\delta \int_{\partial B_{\delta }}
			\frac{e^{2\eta_{0}(x)}}{c_{\lambda}^{(1)}}\Big(1+\frac{\eta^{2}_{0}+\eta_{0}+2w_{0}}{\big( c^{(1)}_\lambda\big)^2}
			+o_{\lambda}\big(\frac{1}{\big( c^{(1)}_\lambda\big)^2}\big)\Big)\\
			{}&\cdot(1+O(\lambda))\Big(\widetilde{\alpha}_{0} \frac{1-|x|^2}{1+|x|^2}+o_{\lambda}(1)\Big) d\sigma\\
			={}&\frac{16\pi \delta^{2}(1-\delta^{2})}{(1+\delta^{2})^{3}}\frac{\widetilde{\alpha}_{0}}{c_{\lambda}^{(1)}}
			+\widetilde{A}_{2}\frac{\widetilde{\alpha}_{0}}{(c_{\lambda}^{(1)})^{3}}+o_{\lambda}\Big(\frac{1}{(c_{\lambda}^{(1)})^{3}}\Big),
		\end{split}
	\end{equation}
	where
	$$
	\widetilde{A}_{2}:= 8\delta\int_{\partial B_{\delta }}e^{2\eta_{0}(x)}(\eta^{2}_{0}+\eta_{0}+2w_{0})\frac{1-|x|^2}{1+|x|^2} d\sigma.
	$$
	Similar to \eqref{5.9-add6}, we also have
	\begin{equation}\label{5.9-add7}
		\begin{split}
			{}&-16\int_{ B_{\delta r^{(1)}_\lambda }}
			\frac{\widetilde{E}_{\lambda}(x)\xi_{\lambda}(x)}{(1-|x|^{2})^{3}} dx\\
			={}&-16 (r_{\lambda}^{(1)})^{2}\int_{ B_{\delta  }}
			\frac{\widetilde{E}_{\lambda}(r_{\lambda}^{(1)}x)\widetilde{\xi}_{\lambda}(x)}{(1-|r_{\lambda}^{(1)}x|^{2})^{3}} dx\\
			={}&-\frac{16}{c_{\lambda}^{(1)}} \int_{ B_{\delta }}
			e^{2\eta_{0}(x)}\Big(1+\frac{\eta^{2}_{0}+\eta_{0}+2w_{0}}{\big( c^{(1)}_\lambda\big)^2}
			+o_{\lambda}\big(\frac{1}{\big( c^{(1)}_\lambda\big)^2}\big)\Big)\\
			{}&\cdot\Big( \widetilde{\alpha}_{0} \frac{1-|x|^2}{1+|x|^2}+
			o_{\lambda}(1)\Big)(1+O(\lambda)) dx\\
			={}&\frac{-16\pi \delta^{2}}{(1+\delta^{2})^{2}}\frac{\widetilde{\alpha}_{0}}{c_{\lambda}^{(1)}}
			+\widetilde{A}_{3}\frac{\widetilde{\alpha}_{0}}{(c_{\lambda}^{(1)})^{3}}+o_{\lambda}\Big(\frac{1}{(c_{\lambda}^{(1)})^{3}}\Big)\Big),
		\end{split}
	\end{equation}
	where
	$$
	\widetilde{A}_{3}=-16\int_{ B_{\delta }}e^{2\eta_{0}(x)}(\eta^{2}_{0}+\eta_{0}+2w_{0})\frac{1-|x|^{2}}{1+|x|^{2}} dx.
	$$
	Similar to \eqref{5.9-add7}, we have
	\begin{equation}\label{5.9-add8}
		\begin{split}
			{}&-16\int_{ B_{\delta r^{(1)}_\lambda }}
			\frac{\widetilde{E}_{\lambda}(x)\xi_{\lambda}(x)x\cdot x}{(1-|x|^{2})^{3}} dx\\
			={}&-16 (r_{\lambda}^{(1)})^{2}\int_{ B_{\delta  }}
			\frac{\widetilde{E}_{\lambda}(r_{\lambda}^{(1)}x)\widetilde{\xi}_{\lambda}(x)r^{(1)}_\lambda x\cdot r^{(1)}_\lambda x }{(1-|r_{\lambda}^{(1)}x|^{2})^{3}} dx\\
			={}&-\frac{16}{c_{\lambda}^{(1)}} \int_{ B_{\delta }}
			e^{2\eta_{0}(x)}\Big(1+\frac{\eta^{2}_{0}+\eta_{0}+2w_{0}}{\big( c^{(1)}_\lambda\big)^2}
			+o_{\lambda}\big(\frac{1}{\big( c^{(1)}_\lambda\big)^2}\big)\Big)\\
			{}&\cdot\Big( \widetilde{\alpha}_{0} \frac{1-|x|^2}{1+|x|^2}+o_{\lambda}(1)\Big)r^{(1)}_\lambda x\cdot r^{(1)}_\lambda x (1+O(\lambda)) dx\\
			={}&o_{\lambda}\Big(\frac{1}{(c_{\lambda}^{(1)})^{3}}\Big).
		\end{split}
	\end{equation}
	From \eqref{5.9-add6} to \eqref{5.9-add8}, we have
	\begin{equation}\label{5.9-add5}
		\begin{split}
			RHS~of ~\eqref{p1_ueta}
			&=-\frac{32\pi \delta^{4}}{(1+\delta^{2})^{3}}\frac{\widetilde{\alpha}_{0}}{c_{\lambda}^{(1)}}
			+
			(\widetilde{A}_{2}+\widetilde{A}_{3})\frac{\widetilde{\alpha}_{0}}{(c_{\lambda}^{(1)})^{3}}
			+o_{_{\lambda}}\Big(\frac{1}{(c_{\lambda}^{(1)})^{3}}\Big).
		\end{split}
	\end{equation}
	Noting that $\widetilde{A}_{1}+\widetilde{A}_{2}+\widetilde{A}_{3}\neq 0,$
	\eqref{5.9-add4} and \eqref{5.9-add5} yield that
	$\widetilde{\alpha}_{0}=0.$
\end{proof}

Now we are in a position to prove Proposition \ref{thuniq}.
\begin{proof}[\textbf{Proof of Proposition \ref{thuniq}}]
	Suppose $u^{(1)}_\lambda\not\equiv u^{(2)}_\lambda$ in $B_{\delta r_{\lambda}}$ and let $$\xi_{\lambda}:=\frac{u^{(1)}_\lambda- u^{(2)}_\lambda}{
		\|u^{(1)}_\lambda-u^{(2)}_\lambda\|_{L^\infty{(B_{\delta r_{\lambda}})}}}.$$ We have
	\begin{equation}\label{5.9-add5-1}
		\|\xi_{\lambda}\|_{L^\infty{(B_{\delta r_{\lambda}})}}=1.
	\end{equation}
	Taking $\widetilde{\xi}_{\lambda}(x):=
	\xi_{\lambda}\big(r^{(1)}_{\lambda}x\big)$, by Propositions \ref{prop_tildeeta} and \ref{prop-A}, we have
	\begin{equation*}\label{loc_o}
		\|\widetilde{\xi}_{\lambda}\|_{L^{\infty}(B_R)}=o_{\lambda}\big(1\big)~\mbox{for~any}~R>0,
	\end{equation*}
	which implies that
	\begin{equation}\label{5.9-add5-2}
		\|\xi_{\lambda}\|_{L^\infty{(B_{\delta r_{\lambda}})}}=o_{\lambda}\big(1\big).
	\end{equation}
	We get a contradiction from \eqref{5.9-add5-1} and \eqref{5.9-add5-2}. Therefore, we infer that $u^{(1)}_\lambda\equiv u^{(2)}_\lambda$ in $B_{\delta r_{\lambda}}.$
\end{proof}

Finally we prove Theorem \ref{unique}.
\begin{proof}[\textbf{Proof of Theorem \ref{unique}}]
	It follows from  Proposition \ref{thuniq} that $u^{(1)}_\lambda(0)= u^{(2)}_\lambda(0).$ Then by the existence and uniqueness theorem for solutions of Cauchy initial value problem for ODEs, we prove that $u^{(1)}_\lambda(x)\equiv u^{(2)}_\lambda(x)$ in $B_{1}.$
\end{proof}

\appendix

\section{Previous results and basic preliminaries}\label{sa}

In this section, we will give some known results and some
preliminaries which are used before.
\begin{lemma}[Lemma 4.3, \cite{EG2004}]\label{lem3.1}
	Let $\eta_{0}$ be the function defined in \eqref{eq5} and $v\in C^2(\R^2)\cap L^{\infty}(\R^2)$ be a solution of the following problem
	\begin{equation*}
		\begin{cases}
			-\Delta v=8e^{2\eta_{0}}v  ~\mbox{in}~\R^2,\\[1mm]
			\displaystyle\int_{\R^2}|\nabla v|^2dx<\infty.
		\end{cases}
	\end{equation*}
	Then it holds that
	\begin{equation*}
		v(x)= \alpha_0\frac{1-|x|^2}{1+|x|^2}+\sum^2_{i=1}{\alpha_i}\frac{x_i}{1+|x|^2}
	\end{equation*}
	with some $\alpha_0,\alpha_1,\alpha_2\in \R$.
\end{lemma}

Recall that 0 is the local maximum point of the solution $u_\lambda$.
Let us define the following quadratic form
\begin{equation}\label{P}
	\begin{split}
		P(d,u,v):=&- 2d\int_{\partial B_d}\big\langle \nabla u ,\nu\big\rangle \big\langle \nabla v,\nu\big\rangle  d\sigma + d \int_{\partial B_d} \big\langle \nabla u , \nabla v \big\rangle d\sigma,
	\end{split}
\end{equation}
where $u,v\in C^{2}(\overline{B_{1}})$, $d>0$ is a small constant such that $B_{d}\subset B_{1}$ and $ \nu= ( \nu_1, \nu_2)$ is the unit outward normal of $\partial B_d$.

By Lemma 2.4 in \cite{LPP-2022}, we have the following property about the above quadratic form.
\begin{lemma}\label{indep_d}
	If $u$ and $v$ are harmonic in $ B_d\backslash \{0\}$, then $P(d,u,v)$ is independent of $d>0$.
\end{lemma}

From Lemma \ref{indep_d}, we will write $P(d,u,v)$ as $P(u,v)$ for simplicity.
Next, we have the following identity about the quadratic form $P$ on the solution $u_\lambda$.
\begin{lemma}\label{lem2.3}
	Let $u_\lambda\in C^2(B_{1})$ be a solution of problem \eqref{lam} and $\delta>0$ is a fixed small constant such that $B_{\delta r_{\lambda}}\subset B_{1}$. Then
	\begin{equation}\label{puu}
		P\big(u_\lambda,u_\lambda\big)= 4
		\delta r_{\lambda}\lambda  \int_{\partial B_{\delta r_{\lambda}}}\frac{e^{u^2_\lambda}}{(1-|x|^{2})^{2}}   d\sigma-8\lambda \int_{B_{\delta r_{\lambda}}} \frac{ e^{u^2_\lambda}}{(1-|x|^{2})^{3}}  dx
		-8\lambda \int_{B_{\delta r_{\lambda}}} \frac{ e^{u^2_\lambda}|x|^2}{(1-|x|^{2})^{3}}  dx.
	\end{equation}
\end{lemma}

\begin{proof}
	Multiplying $\big\langle x, \nabla u_\lambda \big\rangle$ on both sides of equation \eqref{lam}  and integrating on $B_{\delta r_{\lambda}}$, we have
	\[\begin{split}
		{}&\frac{1}{2}\int_{\partial B_{\delta r_{\lambda}}} \big\langle x,\nu\big\rangle |\nabla u_\lambda|^2  d\sigma
		-\int_{\partial B_{\delta r_{\lambda}}}\frac{\partial u_\lambda}{\partial\nu} \big\langle x, \nabla u_\lambda \big\rangle  d\sigma \\
		={}&2\lambda  \int_{\partial B_{\delta r_{\lambda}}}\frac{e^{u^2_\lambda}}{(1-|x|^{2})^{2}} \big\langle x,\nu\big\rangle  d\sigma - 4
		\lambda \int_{B_{\delta r_{\lambda}}} \frac{e^{u^2_\lambda}}{(1-|x|^{2})^{3}}  dx
		- 4\lambda \int_{B_{\delta r_{\lambda}}} \frac{e^{u^2_\lambda}|x|^2 }{(1-|x|^{2})^{3}} dx,
	\end{split}
	\]
	which together with \eqref{P} implies \eqref{puu}.
\end{proof}

\section{Decay estimates of positive solutions}\label{sb}

In this section, we will study on decay properties of entire positive solutions. Let $u_{\lambda}$ be a positive symmetric solution of \eqref{uniq}.
As a function on $B_1,u_{\lambda}=u_{\lambda}(|\zeta|),\zeta\in \R^{2},|\zeta|<1$ and
\begin{equation}\label{sb-1}
	\Big(\frac{1-|\zeta|^{2}}{2}\Big)^{2}\Delta u_{\lambda}+\lambda u_{\lambda} e^{u_{\lambda}^{2}}=0.
\end{equation}
Setting $|\zeta|:=\tanh \frac{t}{2},l=\sinh t,$ for simplicity of notations, we still denote $u_{\lambda}(t)=u_{\lambda}(\tanh \frac{t}{2}),$
then it is easy to see that
\begin{equation*}\label{norm-6-20}
	\int_{B_1}|u_{\lambda}|^{2}d\zeta=\omega_{2}\int_{0}^{\infty}l|u_{\lambda}|^{2}dt,\,\,\int_{B_1}|\nabla u_{\lambda}|^{2}d\zeta=\omega_{2}\int_{0}^{\infty}l|u'_{\lambda}|^{2}dt
\end{equation*}
where $\omega_2=2\pi$. So it follows from Poincar\'{e} inequality
\begin{equation}\label{sb-2}
	u_{\lambda}\in H^{1}(B_1)\,\,\,\Leftrightarrow \,\,\,\omega_{2}\int_{0}^{+\infty}l (|u_{\lambda}|^{2}+| u'_{\lambda}|^{2})dt=
	\int_{B_1}(|u_{\lambda}|^{2}+|\nabla u_{\lambda}|^{2})d \zeta<+\infty.
\end{equation}
\eqref{sb-1} can be rewritten as
\begin{equation}\label{sb-3}
	u''_{\lambda}+\coth tu'_{\lambda}+\lambda u_{\lambda}e^{u_{\lambda}^2}=0,\,\,\,t>0,\,\,\,u_\lambda'(0)=0,
\end{equation}	
as well as
\begin{equation}\label{sb-4}
	(lu_{\lambda}')'+\lambda lu_{\lambda}e^{u_{\lambda}^2}=0,\,\,\,t>0,\,\,\,u'_{\lambda}(0)=0.
\end{equation}	
From Proposition \ref{pro}, we know the positive solutions of \eqref{uniq} have hyperbolic symmetry, so if there
is no confusion, we will write
$u_{\lambda}(\tanh \frac{t}{2})$ as $u_{\lambda}(t).$
\begin{lemma}[Lemma 3.1, \cite{MS-2008}]\label{lem-b1}
	Denote by $H$ the closure of $C_c^\infty(\mathbb{B}^2)$ with respect to the norm
	\begin{equation*}
		\|\cdot\|=\int_{\mathbb{B}^2}\Big(|\nabla_{\mathbb{B}^2}\cdot|^2-\frac{1}{4}|\cdot|^2\Big)dV_{\mathbb{B}^2}.
	\end{equation*}
	Let $u_{\lambda}\in H$ be a symmetric function. Then
	\begin{equation}\label{sb-5}
		\|u_{\lambda}\|^{2}_{H}=\omega_{2}\int_{0}^{\infty}l\Bigg[\bigg(u'_{\lambda}+\frac{1}{2}\tanh \frac{t}{2} u_{\lambda}\bigg)^{2}+\frac{u_{\lambda}^{2}}{(2 \cosh \frac{t}{2})^{2}}\Bigg]dt<+ \infty.
	\end{equation}
\end{lemma}
Now, let us observe that if $u_{\lambda}$ solves \eqref{sb-1} and
$$
I_{u_{\lambda}}(t)=:\frac{(u'_{\lambda})^{2}}{2}+\frac{\lambda}{2}e^{u_{\lambda}^{2}},
$$
then
$$
\frac{d}{dt}I_{u_{\lambda}}(t)=u'_{\lambda}(u''_{\lambda}+\lambda u_{\lambda} e^{u_{\lambda}^{2}})=-\coth t(u_{\lambda}')^{2}\leq 0,\,\,\,\forall t>0.
$$
Particularly, $u_{\lambda}$ and $u'_{\lambda}$ remain bounded and hence $u_{\lambda}$ is defined for every $t.$

\begin{lemma}\label{lem-b2}
	Let $u_\lambda$ be a positive solution of \eqref{uniq}. If $0<\lambda<\frac{1}{4},$ then
	$u'_{\lambda}(t)<0$ for every $t>0$ and $\lim\limits_{t\rightarrow+\infty}u_{\lambda}(t)
	=\lim\limits_{t\rightarrow+\infty}u'_{\lambda}(t)=0.$
\end{lemma}

\begin{proof}
	If $0<\lambda<\frac{1}{4},$ equation \eqref{sb-4} implies that
	$$
	(lu'_{\lambda})'(t)=-\lambda l u_{\lambda} e^{u_{\lambda}^{2}}<0,\,\,\,\forall t>0,
	$$
	and there $u_{\lambda}'(t)<0,\,\,\,\forall t>0,$ since $u_{\lambda}'(0)=0.$
	Particularly, there exists $u_{\lambda}(\infty)=:\lim_{t\rightarrow +\infty}u_{\lambda}(t)$ and since $I_{u_{\lambda}}(t)$ is decreasing,
	$u'_{\lambda}(\infty)=\lim_{t\rightarrow +\infty}u'_{\lambda}(t)$ exists as well and is zero because $u_{\lambda}$ is positive.
	Finally, \eqref{sb-3} yields
	$u''_{\lambda}(\infty)=-\lambda u_{\lambda}(\infty)e^{u_{\lambda}^{2}(\infty)},$ necessarily $u''_{\lambda}(\infty)=0$ and hence $u_{\lambda}(\infty)=0.$
\end{proof}

\begin{lemma}\label{lem-b3}
	Let $u_{\lambda}>0$ be a solution of \eqref{sb-3} satisfying $\|\nabla_{\mathbb{B}^2} u_{\lambda}\|_2^2\leq M_0.$ Then
	\begin{equation}\label{b3-1}
		C_{1}(T)e^{-\frac{A_{T}+\sqrt{A^{2}_{T}-4\lambda }}{2}t}\leq u_{\lambda}(t)\leq C_{2}(T)e^{-\frac{1+\sqrt{1-4\lambda B_{T}}}{2}t},
	\end{equation}
	and
	\begin{equation}\label{b3-2}
		\tilde{C}_{1}(T)e^{-\frac{A_{T}+\sqrt{A^{2}_{T}-4\lambda }}{2}t}\leq (-u'_{\lambda}(t))\leq\tilde{ C}_{2}(T)e^{-\frac{1
				+\sqrt{1-4\lambda B_{T}}}{2}t},
	\end{equation}
	where
	$$
	A_{T}=\frac{e^{2T}+1}{e^{2T}-1}, \,\, B_{T}= e^{\frac{4M_0}{\pi} \frac{1}{e^{T}+e^{-T}-2}},$$
	$$\mu^{\pm}(T)=\frac{-A_{T}\pm\sqrt{(A_{T})^{2}-4\lambda}}{2},\,\,
	\nu^{\pm}(T)=\frac{-1\pm\sqrt{1-4\lambda B_{T}}}{2},
	$$
	$C_{i}(T)$ and $\tilde{C}_{i}(T)$ are positive constants which are defined in \eqref{b3-20} and \eqref{b3-20-1} respectively.
\end{lemma}

\begin{proof}
	Noting that $\|\nabla_{\mathbb{B}^2} u_\lambda\|_2^2\leq M_0,$  then from \eqref{Poincare} we have
	\begin{equation*}
		\begin{split}
			\int_{\mathbb{B}^2}|u_{\lambda}|^{2}dV_{\mathbb{B}^2}=\omega_{2}\int_{0}^{+\infty} l u^{2}_{\lambda}(t)dt\leq 4M_0,
		\end{split}
	\end{equation*}
	which implies that
	\begin{equation*}\label{6-21-1}
		\int_{0}^{+\infty} l u^{2}_{\lambda}(t)dt \leq \frac{2M_{0}}{\pi}.
	\end{equation*}
	Since it follows from Lemma \ref{lem-b2} that $u_{\lambda}(t)$ is decreasing in $(0,+\infty),$ we have
	\begin{equation*}
		\begin{split}
			u_{\lambda}^{2}(T)\int_{0}^{T} l dt\leq \int_{0}^{+\infty} l u^{2}_{\lambda}(t)dt\leq \frac{2M_0}{\pi},
		\end{split}
	\end{equation*}
	which implies that
	\begin{equation}\label{6-21-2}
		\begin{split}
			u_{\lambda}^{2}(T)\leq \frac{4M_0}{\pi} \frac{1}{e^{T}+e^{-T}-2}.
		\end{split}
	\end{equation}
	By \eqref{6-21-2} and the fact that $\coth t$ is increasing about $t\in (0,+\infty),$ we infer that
	\begin{equation*}\label{6-21-3}
		\begin{aligned}
		1<\coth t\leq \coth T=\frac{e^{2T}+1}{e^{2T}-1}:=A_{T},\,\,\,\,\forall t\geq T,\\
		u_{\lambda}
		e^{u_{\lambda}^{2}} \leq u_{\lambda}(t)e^{\frac{4M_0}{\pi} \frac{1}{e^{T}+e^{-T}-2}}:=B_{T}u_{\lambda}(t),\,\,\,\,\forall t\geq T.
		\end{aligned}
	\end{equation*}
	Since again by Lemma \ref{lem-b2} $u'_{\lambda}$ is negative, we have for $t \geq T,$
	\begin{equation}\label{b3-3}
		u''_{\lambda}+A_{T}u'_{\lambda}
		+\lambda u_{\lambda}\leq u''_{\lambda}+\coth t u'_{\lambda} +\lambda u_{\lambda}e^{u_{\lambda}^{2}}=0
	\end{equation}
	and
	\begin{equation}\label{b3-4}
		u''_{\lambda}+u'_{\lambda}+\lambda B_{T} u_{\lambda}\geq  u''_{\lambda}+\coth t u'_{\lambda} +\lambda u_{\lambda}e^{u_{\lambda}^{2}}=0.
	\end{equation}
	Let
	\begin{equation*}\label{b3-3-add}
		\mu^{\pm}(T):=\frac{-A_{T}\pm\sqrt{(A_{T})^{2}-4\lambda}}{2},\,\,\,
		\nu^{\pm}(T):=\frac{-1\pm\sqrt{1-4\lambda B_{T}}}{2}
	\end{equation*}
	be the characteristic roots of the differential polynomials in the left hand sides of
	\eqref{b3-3} and \eqref{b3-4} respectively (notice that $\mu^{\pm}(T)$ are real and distinct for $0<\lambda<\frac{1}{4}$;
	to have $\nu^{\pm}(T)$ real and distinct we need to choose $T>\ln\Big(1-\frac{2M_{0}}{\pi\ln(4\lambda)}+2\sqrt{\frac{M^{2}_{0}}{\pi^{2}(\ln(4\lambda))^{2}}-\frac{M_{0}}{\pi\ln(4\lambda)}}\Big)$ and with this choice
	$\mu^{-}(T)<\nu^{-}(T)$).
	Let
	\begin{equation*}\label{b3-5}
		\hat{\phi}_{\lambda}:=u'_{\lambda}-\mu^{+}(T)u_{\lambda},
		\,\,\,\hat{\psi}_{\lambda}:=u'_{\lambda}-\nu^{+}(T)u_{\lambda}.
	\end{equation*}
	Then from \eqref{b3-3} and \eqref{b3-4}, we have for $t\geq T,$
	\begin{equation}\label{b3-6}
		\begin{split}
			\hat{\phi}'_{\lambda}-\mu^{-}(T)\hat{\phi}_{\lambda}
			&=u''_{\lambda}-(\mu^{+}(T)+\mu^{-}(T))u'_{\lambda}+\mu^{+}(T)
			\mu^{-}(T)u_{\lambda}
			\\&=u''_{\lambda}+A_{T}u'_{\lambda}+\lambda u_{\lambda}\leq 0
		\end{split}
	\end{equation}
	and
	\begin{equation}\label{b3-7}
		\begin{split}
			\hat{\psi}'_{\lambda}-\nu^{-}(T)\hat{\psi}_{\lambda}
			&=u''_{\lambda}-(\nu^{+}(T)+\nu^{-}(T))u'_{\lambda}+\nu^{+}(T)\nu^{-}(T)u_{\lambda}\\
			&=u''_{\lambda}+u'_{\lambda}+\lambda B_{T} u_{\lambda}\geq 0.
		\end{split}
	\end{equation}
	From \eqref{b3-6} and \eqref{b3-7}, we derive
	$(e^{-\mu^{-}(T)t}\hat{\phi}_{\lambda})'\leq 0 \leq (e^{-\nu^{-}(T)t}\hat{\psi}_{\lambda})'$
	for every  $t\geq T$ and hence, integrating such inequalities in $[\tau,t], T\leq \tau\leq t,$
	we get respectively
	\begin{equation*}\label{b3-8}
		\begin{split}
			u'_{\lambda}(t)-\mu^{+}(T)u_{\lambda}(t)=\hat{\phi}_{\lambda}(t)
			\leq (e^{-\mu^{-}(T)\tau}\hat{\phi}_{\lambda}(\tau))e^{\mu^{-}(T)t}
			:=c_{T}(\tau)e^{\mu^{-}(T)t}
		\end{split}
	\end{equation*}
	and
	\begin{equation*}\label{b3-9}
		\begin{split}
			u'_{\lambda}(t)-\nu^{+}(T)u_{\lambda}(t)=\hat{\psi}_{\lambda}(t)
			\geq  (e^{-\nu^{-}(T)\tau}\hat{\psi}_{\lambda}(\tau))e^{\nu^{-}(T)t}
			:=d_{T}(\tau)e^{\nu^{-}(T)t}
		\end{split}
	\end{equation*}
	for every $t\geq \tau \geq T.$ Again by integration, we finally get, for every $t\geq \tau \geq T,$
	\begin{equation}\label{b3-10}
		\begin{split}
			u_{\lambda}(t)&\leq e^{-\mu^{+}(T)\tau}u_{\lambda}(\tau) e^{\mu^{+}(T)t}
			+c_{T}(\tau)\frac{e^{\mu^{-}(T)t}-e^{(\mu^{-}(T)-\mu^{+}(T))\tau
					+\mu^{+}(T)t}}{\mu^{-}(T)-\mu^{+}(T)}
			\\
			&=\Big[ e^{-\mu^{+}(T)\tau}u_{\lambda}(\tau)
			-\frac{c_{T}(\tau)}{\mu^{-}(T)-\mu^{+}(T)}
			e^{(\mu^{-}(T)-\mu^{+}(T))\tau}
			\Big]e^{\mu^{+}(T)t}
			+\frac{c_{T}(\tau)e^{\mu^{-}(T)t}}{\mu^{-}(T)-\mu^{+}(T)}
		\end{split}
	\end{equation}
	and
	\begin{equation}\label{b3-11}
		\begin{split}
			u_{\lambda}(t)&\geq e^{-\nu^{+}(T)\tau}u_{\lambda}(\tau) e^{\nu^{+}(T)t}
			+d_{T}(\tau)\frac{e^{\nu^{-}(T)t}-e^{(\nu^{-}(T)-\nu^{+}(T)))\tau
					+\nu^{+}(T)t}}{\nu^{-}(T)-\nu^{+}(T)}
			\\
			&=\Big[ e^{-\nu^{+}(T)\tau}u_{\lambda}(\tau)
			-\frac{d_{T}(\tau)}{\nu^{-}(T)-\nu^{+}(T)}
			e^{(\nu^{-}(T)-\nu^{+}(T))
				\tau}
			\Big]e^{\nu^{+}(T)\tau}
			+\frac{d_{T}(\tau)e^{\nu^{-}(T)t}}{\nu^{-}(T)-\nu
				^{+}(T)}.
		\end{split}
	\end{equation}
	Now, $u_{\lambda}$ is positive, $\mu^{-}(T)<\mu^{+}(T)$ and \eqref{b3-10} yields that
	\begin{equation*}
		e^{-\mu^{+}(T)\tau}u_{\lambda}(\tau)
		-\frac{c_{T}(\tau)}{\mu^{-}(T)-\mu^{+}(T)}
		e^{(\mu^{-}(T)-\mu^{+}(T))\tau}
		\geq 0,\,\,\,\forall \tau\geq T,
	\end{equation*}
	that is
	\begin{equation*}\label{b3-12}
		\begin{split}
			u_{\lambda}(\tau)\geq
			\frac{c_{T}(\tau)
				e^{\mu^{-}(T)\tau}}{\mu^{-}(T)-\mu^{+}(T)}
			=\frac{\hat{\phi}_{\lambda}(\tau)}{\mu^{-}(T)-\mu^{+}(T)}
			=\frac{u'_{\lambda}(\tau)-\mu^{+}(T)u_{\lambda}(\tau)}{\mu^{-}(T)-\mu^{+}(T)},\,\,\,\forall \tau\geq T.
		\end{split}
	\end{equation*}
	Hence,
	\begin{equation}\label{b3-13}
		u'_{\lambda}(\tau)\geq \mu^{-}(T)u_{\lambda}(\tau),\,\,\,\forall \tau \geq T.
	\end{equation}
	And integrating on $[T,t],$ we find
	\begin{equation}\label{b3-14}
		u_{\lambda}(t)\geq u_{\lambda}(T)e^{-\mu^{-}(T)T} e^{\mu^{-}(T)t},\,\,\,\forall t \geq T.
	\end{equation}
	Similarly, from \eqref{b3-11} we obtain
	$$
	\hat{d}_{T}(\tau)=e^{-\upsilon^{+}(T)\tau}u_{\lambda}(\tau)
	-\frac{d_{T}(\tau)}{\nu^{-}(T)-\nu^{+}(T)}
	e^{(\nu^{-}(T)-\nu^{+}(T))\tau}\leq 0,\,\,\,\forall \tau\geq T.
	$$
	Otherwise, since $\nu^{-}(T)< \nu^{+}(T), $
	\eqref{b3-11} gives
	\begin{equation}\label{b3-15}
		u_{\lambda}(t)\geq \frac{1}{2}\hat{d}_{T}(\tau)e^{\nu^{+}(T)t}.
	\end{equation}
	But since $0<\lambda<\frac{1}{4}$ implies that $1+2\nu^{+}(T)>0,$ this inequality and \eqref{b3-15} yields
	\begin{equation}\label{b3-16}
		\begin{split}
			+\infty &> \int_{0}^{+\infty}l
			u_{\lambda}^{2}dt=\int_{0}^{+\infty} \sinh t u_{\lambda}
			^{2}dt\\
			&\geq \frac{1}{4}(\hat{d}_{T}(\tau))^{2}
			\int_{0}^{+\infty}\frac{e^{t}-e^{-t}}{2}e^{2\nu^{+}(T)t}dt\\
			&\geq \frac{1}{16}(\hat{d}_{T}(\tau))^{2}
			\int_{0}^{+\infty}(e^{t}-1)e^{2\nu^{+}(T)t}dt\\
			&\geq \frac{1}{32}(\hat{d}_{T}(\tau))^{2}\int_{\ln2}^{+\infty} e^{(2\nu^{+}(T)+1)t}dt\rightarrow +\infty.
		\end{split}
	\end{equation}
	This is a contradiction. Thus, $\hat{d}_{T}(\tau)\leq 0,\,\,\,\forall \tau\geq T,$ that is
	\begin{equation*}\label{b3-17}
		\begin{split}
			u_{\lambda}(\tau)&
			\leq
			\frac{d_{T}(\tau)
				e^{\nu^{-}(T)\tau}}{\nu^{-}(T)-\nu^{+}(T)}
			=\frac{\hat{\psi}_{\lambda}(\tau)}{\nu^{-}(T)-\nu^{+}(T)}
			=\frac{u'_{\lambda}(\tau)-\nu^{+}(T)u_{\lambda}(\tau)}{\nu^{-}(T)-\nu^{+}(T)},\,\,\,\forall \tau\geq T.
		\end{split}
	\end{equation*}
	Hence,
	\begin{equation}\label{b3-18}
		u'_{\lambda}(\tau)\leq \nu^{-}(T)u_{\lambda}(\tau),\,\,\,\forall \tau \geq T.
	\end{equation}
	And integrating on $[T,t],$ we find
	\begin{equation}\label{b3-19}
		u_{\lambda}(t)\leq u_{\lambda}(T)e^{-\nu^{-}(T)T} e^{\nu^{-}(T)t},\,\,\,\forall t \geq T.
	\end{equation}
	From \eqref{b3-14} and \eqref{b3-19}, we get
	\begin{equation}\label{b3-20}
		\begin{aligned}
		C_{1}(T)e^{\mu^{-}(T)t}:={}&(u_{\lambda}(T)e^{-\mu^{-}(T)T} )e^{\mu^{-}(T)t}\\
		\leq{}& u_{\lambda}(t)\leq( u_{\lambda}(T)e^{-\nu^{-}(T)T}) e^{\nu^{-}(T)t}:=C_{2}(T)e^{\nu^{-}(T)t},\,\,\,\forall t \geq T,
		\end{aligned}
	\end{equation}
	which combining \eqref{b3-13} and \eqref{b3-20} yields that
	\begin{equation}\label{b3-20-1}
		\begin{split}
			\tilde{C}_{1}(T)e^{\mu^{-}(T)t}&:=(-\nu^{-}(T)u_{\lambda}(T)e^{-\mu^{-}(T)T}) e^{\mu^{-}(T)t}\\
			&\leq (-u'_{\lambda}(t))\\
			&\leq (-\mu^{-}(T)u_{\lambda}(T)e^{-\nu^{-}(T)T} ) e^{\nu^{-}(T)t}
			:=\tilde{C}_{2}(T)e^{\nu^{-}(T)t},\,\,\,\forall t \geq T.
		\end{split}
	\end{equation}
	%From \eqref{b3-13}-\eqref{b3-18}, and \eqref{b3-14}-\eqref{b3-19}, we see that
	%\begin{equation*}
	%\begin{split}
	%-(1+\sqrt{1-4\lambda})=\lim_{t\rightarrow +\infty}\frac{\ln u^{2}}{t}&=\lim_{t\rightarrow +\infty}\frac{\ln[(\nu^{-}(\epsilon))^{2} u^{2}]}{t}
	%\leq \lim_{t\rightarrow +\infty}\frac{\ln(u')^{2} }{t}\leq \lim_{t\rightarrow +\infty}\frac{\ln[(\mu^{-}(\epsilon))^{2} u^{2}]}{t}\\
	%&=\lim_{t\rightarrow +\infty}\frac{\ln(u)^{2} }{t}=-(1+\sqrt{1-4\lambda}),
	%\end{split}
	%\end{equation*}
	%which implies that
	%$$
	%\lim_{t\rightarrow +\infty}\frac{\ln(u')^{2} }{t}=-(1+\sqrt{1-4\lambda}).
	%$$
	%Finally, taking limsup and lininf in \eqref{b3-13}-\eqref{b3-18}, then letting $\epsilon$ to zero, we see that
	%$$
	%-(1+\sqrt{1-4\lambda})=\lim_{t\rightarrow +\infty}\nu^{-}(\epsilon)\leq \lim_{t\rightarrow +\infty}\frac{u'(t)}{u(t)}\leq \lim_{t\rightarrow +\infty}\mu^{-}(\epsilon)
	%=-(1+\sqrt{1-4\lambda}).
	%$$
\end{proof}

\begin{corollary}\label{lem-b3-1}
	Let $u_{\lambda}>0$ be a solution of \eqref{sb-3} satisfying $\|\nabla_{\mathbb{B}^2} u_{\lambda}\|_2^2\leq M_0.$ Then there hold
	\begin{equation}\label{b3-1-1}
		\lim_{t\rightarrow +\infty}\frac{\ln u_{\lambda}^{2}(t)}{t}
		=\lim_{t\rightarrow +\infty}\frac{\ln (u'_{\lambda})^{2}(t)}{t}
		=\lim_{t\rightarrow +\infty}\frac{\ln (u_{\lambda}^{2}+(u'_{\lambda})^{2})(t)}{t}
		=-(1+\sqrt{1-4\lambda})
	\end{equation}
	and
	\begin{equation}\label{b3-2-1}
		\lim_{t\rightarrow +\infty}\frac{u'_{\lambda}(t)}{u_{\lambda}(t)}=-\frac{1+\sqrt{1-4\lambda}}{2}.
	\end{equation}
\end{corollary}

\begin{proof}
	By \eqref{b3-1} and \eqref{b3-2}, noting that $\lim\limits_{T\rightarrow +\infty}A_{T}=\lim\limits_{T\rightarrow +\infty}B_{T}=1,$
	we can infer that \eqref{b3-1-1} holds.
	Taking $\limsup$ and $\liminf$ in
	\eqref{b3-15} and \eqref{b3-18} respectively, then sending $T\rightarrow\infty$ we see that
	\eqref{b3-2-1} is true.
\end{proof}

\begin{remark}
	We would like to point out that in fact if
	$u_{\lambda}\in W^{1,2}(\mathbb{B}^2)$
	we can also prove that \eqref{b3-1-1} and \eqref{b3-2-1} hold. And we do not need to require that $\|\nabla_{\mathbb{B}^2} u_{\lambda}\|_2^2\leq M_0.$
\end{remark}

Let us now consider the limit case $\lambda=\frac{1}{4}.$ Here we have more precise estimates, which are needed before.

\begin{lemma}\label{lem-b4}
	Let $\lambda=\frac{1}{4}$ and $u_{\lambda}>0$ be a solution of \eqref{sb-3} satisfying \eqref{sb-2}. Then
	\begin{equation}\label{b4-1}
		\lim_{t\rightarrow +\infty}\frac{\ln u_{\lambda}^{2}}{t}=\lim_{t\rightarrow +\infty}\frac{\ln (u'_{\lambda})^{2}}{t}=\lim_{t\rightarrow +\infty}\frac{\ln[u_{\lambda}^{2}+ (u'_{\lambda})^{2}]}{t}
		=-1.
	\end{equation}
	Moreover, if $u_{\lambda}$ satisfies \eqref{sb-5}, then there is $A>0$ such that
	\begin{equation}\label{b4-2}
		e^{\frac{t}{2}}u_{\lambda}(t) \rightarrow A,\,\,\text{and}\,\,\,e^{\frac{t}{2}}u'_{\lambda}(t) \rightarrow -\frac{A}{2},\,\,\,\text{as}\,\,t\rightarrow +\infty.
	\end{equation}
\end{lemma}

\begin{proof}
	Since $\mu^{\pm}(T)$ (see the proof of Lemma \ref{lem-b3}) are real and distinct, \eqref{b3-3}-\eqref{b3-10} and hence the bound from below
	$u_{\lambda}(t)\geq C_{T}e^{\mu^{-}(T)t}$ given in \eqref{b3-14}, still holds true, because at this stage, it was just required that $u_{\lambda}(t)$ decreases to zero,
	a property satisfied here (Lemma \ref{lem-b2}). For the same reason, an upper bound for $u_{\lambda}(t)$ and \eqref{b3-13}
	will give a similar upper bound for $-u'_{\lambda}(t).$
	
	A bound from above for $u_{\lambda}(t)$ can be obtained as follows.
	First we observe that $u''_{\lambda}(t)$ which can not be definitely negative and then it is definitely positive, otherwise, it has a sequence $t_{j}\rightarrow +\infty$ of zeros, that is, denoted
	$\hat{u}_{\lambda}(t)=u'_{\lambda}(t).$ Then $\hat{u}_{\lambda}(t)<0$ and $\hat{u}'(t_{j})=0.$ Taking the derivative of \eqref{sb-3}, we get
	\begin{equation*}\label{b4-3}
		\hat{u}''_{\lambda}+\coth t\hat{u}'_{\lambda}+\bigg(2\lambda u_{\lambda}^{2}e^{u_{\lambda}^{2}}+\lambda e^{u_{\lambda}^{2}}-\frac{1}{(\sinh t)^{2}}\bigg
		)\hat{u}_{\lambda}=0
	\end{equation*}
	and we see that there is $\underline{t}$ such that
	$\hat{u}''_{\lambda}(t_{j})>0$  for $ t_{j}\geq t,$ which $\hat{u}''_{\lambda}$ has to change sign at two consecutive zeros of $\hat{u}_{\lambda}'.$
	In turn, $u''_{\lambda}>0$ for large $t$ implies that $u_{\lambda}$ (and $u'_{\lambda}$), has exponential decay. In fact, if $u_{\lambda}''>0$ for
	$t\geq \bar{t},$ then using \eqref{sb-3} we see that
	$$
	\coth t u'_{\lambda}+\lambda u_{\lambda}=-u''_{\lambda}-\lambda(e^{u^{2}_{\lambda}}-1)u_{\lambda}<0
	$$
	for $t\geq \bar{t}$ and hence
	$$
	\frac{u'_{\lambda}(t)}{u_{\lambda}(t)}\leq -\lambda\, \tanh \bar{t}
	$$
	and hence
	$$
	u_{\lambda}(t)\leq u_{\lambda}(\bar{t})e^{-\lambda \tanh\bar{t}(t-\bar{t})}.
	$$
	As a result, as observed above, a similar upper bound holds true for
	$-u'_{\lambda}.$ Now, from \eqref{sb-3} and \eqref{b3-14}, we get for
	$\lambda = \frac{1}{4},$
	\begin{equation*}
		\begin{split}
			-(l u'_{\lambda})(t)&= -(l u'_{\lambda})(T)
			+\lambda \int_{T}^{t}l u_{\lambda}e^{u_{\lambda}^{2}}\\
			&\geq\lambda C_{T}
			\int_{t_{\epsilon}}^{t}le^{-\frac{A_{T}+\sqrt{A^{2}_{T}-1}}{2}s}ds \geq \hat{C}_{T}e^{\frac{2-(A_{T}+\sqrt{A^{2}_{T}-1})}{2}t}
		\end{split}
	\end{equation*}
	and
	then
	$$
	-u'_{\lambda}(t)\geq\hat{C}_{T}e^{\frac{-A_{T}+\sqrt{A^{2}_{T}-1}}{2}t}.
	$$
	Therefore, for any given $\epsilon>0,$
	there exist positive numbers $a_{\epsilon},b_{\epsilon}$ such that
	$$
	a_{\epsilon}e^{-(\frac{1}{2}+\epsilon)t}
	\leq u_{\lambda}\leq  b_{\epsilon}e^{-(\frac{1}{4}-\epsilon)t}
	$$
	and
	$$
	a_{\epsilon}e^{-(\frac{1}{2}+\epsilon)t}\leq
	-u'_{\lambda}\leq  b_{\epsilon}e^{-(\frac{1}{4}-\epsilon)t}.
	$$
	Now, let
	$$
	\textbf{U}_{\lambda}=(u_{\lambda},u'_{\lambda}),
	F(u_{\lambda},t)=(0,u'_{\lambda}(1-\coth t)-\lambda u_{\lambda}(e^{u_{\lambda}^{2}}-1))
	$$
	and
	\begin{eqnarray*}
		\textbf{A}:= \left[
		\begin{matrix}
			0&1\\
			-\lambda&-1
		\end{matrix}
		\right].
	\end{eqnarray*}
	We can write \eqref{sb-3} as
	\begin{eqnarray*}
		\textbf{U}'_{\lambda}&=&(u'_{\lambda},u''_{\lambda})=
		(u'_{\lambda},-\lambda u_{\lambda}-u'_{\lambda}+u'_{\lambda}(1-\coth t)-
		\lambda(e^{u_{\lambda}^{2}}-1)u_{\lambda})\\
		&=&\textbf{A}u_{\lambda}+F(u_{\lambda},t)\\
		&=&\left[
		\begin{matrix}
			0&1\\
			-\lambda&-1
		\end{matrix}
		\right](u_{\lambda},u'_{\lambda})+(0,u'_{\lambda}(1-\coth t)-\lambda u_{\lambda}(e^{u_{\lambda}^{2}}-1))\\
		&=&(u'_{\lambda},-\lambda u_{\lambda}-u'_{\lambda})
		+(0,u'_{\lambda}(1-\coth t)-\lambda u_{\lambda}(e^{u_{\lambda}^{2}}-1)).
	\end{eqnarray*}
	The above analysis gives
	$$
	a:=\lim_{t\rightarrow +\infty}\frac{\ln |\emph{U}_{\lambda}|}{t}<0.
	$$
	Since $|F(u_{\lambda},t)|\leq \epsilon|\textbf{U}_{\lambda}|$ for $t$ large and $U_{\lambda}$ small, $a$
	has to be a nonpositive characteristic root of $\textbf{A}$
	(see(\cite{CL-1985}, Theorem 4.3])),
	i.e. $a=\frac{1}{2}(-1\pm\sqrt{1-4\lambda})=-\frac{1}{2}.$
	This proves \eqref{b4-1}.
	
	To prove \eqref{b4-2}, noting that applying \eqref{b4-1}, we can rewrite \eqref{sb-3} as
	$$
	u''_{\lambda}+u'_{\lambda}+\frac{1}{4}u_{\lambda}=f(t),
	$$
	where $|f(t)|\leq ce^{-\alpha t},\,\,t>0,$
	for some $\alpha>\frac{1}{2}$ and $C>0.$
	Let $\varrho=u'_{\lambda}+\frac{1}{2}u_{\lambda}.$ Then
	\begin{equation}\label{b4-2-1}
		(e^{\frac{1}{2}t}\varrho)'=e^{\frac{t}{2}}(\varrho'+\frac{1}{2}\varrho)
		=e^{\frac{t}{2}}f(t).
	\end{equation}
	Since $u_{\lambda}$ satisfying \eqref{sb-5}, we have
	$$
	\liminf_{t\rightarrow +\infty}l\bigg(u'_{\lambda}+\frac{1}{2}\tanh \frac{t}{2}u_{\lambda}\bigg)^{2}=0.
	$$
	Thus, there exists a sequence $t_{n}\rightarrow +\infty$ such that
	$e^{\frac{1}{2}t_{n}}\varrho(t_{n})\rightarrow 0.$
	Let $t<t_{n},$ integrating \eqref{b4-2-1} from $t$ to $t_{n}$
	and then taking the limits as $n\rightarrow +\infty,$ we get
	\begin{equation}\label{b4-2-2}
		e^{\frac{1}{2}t}\varrho=g(t),
	\end{equation}
	where $g(t)\leq Ce^{(\frac{1}{2}-\alpha)t}$
	for some $C>0.$
	This yields
	\begin{equation*}\label{b4-2-3}
		(e^{\frac{1}{2}t}u_{\lambda}(t))'=e^{\frac{t}{2}}\varrho(t)=g(t).
	\end{equation*}
	Integrating from 0 to $t$ and observing that
	$g(t)$ has exponential decay, we have
	$$
	e^{\frac{t}{2}}u_{\lambda}(t)=u_{\lambda}(0)+\int_{0}^{t}g(s)ds\rightarrow A\geq 0,\,\,\,
	\text{as}\,\,t\rightarrow +\infty.
	$$
	Supposing that $A=0,$ then integrating \eqref{b4-2-2} from $t$ to $+\infty$ we get
	$$
	e^{\frac{t}{2}}u_{\lambda}(t)\leq Ce^{(\frac{1}{2}-\alpha)t},
	$$
	which contradicts with \eqref{b4-1} as $\alpha>\frac{1}{2}.$ Now the estimate on $e^{\frac{t}{2}}u_\lambda'(t)$ follows from \eqref{b4-2-2}.
\end{proof}

\begin{proposition}\label{pro-b5}
	If $\lambda=\frac{1}{4},$ there is no positive solution of \eqref{uniq} in $H^{1}(\mathbb{B}^2).$
\end{proposition}

\begin{proof}
	By the asymptotic behaviors proved in Lemma \ref{lem-b4}, positive solutions of \eqref{uniq} cannot be in $H^{1}(\mathbb{B}^2)$ if $\lambda=\frac{1}{4}.$
\end{proof}

\end{document}